\documentclass[10pt]{amsart}

\usepackage[mathscr]{euscript}
\DeclareFontFamily{OT1}{rsfs}{}
\DeclareFontShape{OT1}{rsfs}{n}{it}{<-> rsfs10}{}
\DeclareMathAlphabet{\curly}{OT1}{rsfs}{n}{it}

\makeatletter
\newcommand{\eqnum}{\refstepcounter{equation}\textup{\tagform@{\theequation}}}
\makeatother

\newcommand\beq[1]{\begin{equation}\label{#1}}
	\newcommand\eeq{\end{equation}}
\newcommand\beqa{\begin{eqnarray*}}
	\newcommand\eeqa{\end{eqnarray*}}

\title[Semiorthogonal decomposition]{Semiorthogonal decompositions for categorical Donaldson-Thomas 
	theory via $\Theta$-stratifications}
\date{}
\author{Yukinobu Toda}

\usepackage{amscd}
\usepackage{amsmath}
\usepackage{amssymb}
\usepackage{amsthm}
\usepackage{float}
\usepackage[dvips]{graphicx}
\usepackage{xypic}

\usepackage{array}
\usepackage{amscd}
\usepackage[all]{xy}

\usepackage{tabulary}
\usepackage{booktabs}

\usepackage{amssymb, paralist, xspace, url, amscd, euscript, mathrsfs, stmaryrd}

\DeclareFontFamily{U}{rsfs}{%
	\skewchar\font127}
\DeclareFontShape{U}{rsfs}{m}{n}{%
	<-6>rsfs5<6-8.5>rsfs7<8.5->rsfs10}{}
\DeclareSymbolFont{rsfs}{U}{rsfs}{m}{n}
\DeclareSymbolFontAlphabet
{\mathrsfs}{rsfs}
\DeclareRobustCommand*\rsfs{%
	\@fontswitch\relax\mathrsfs}

\theoremstyle{plain}
\newtheorem{thm}{Theorem}[section]
\newtheorem{prop}[thm]{Proposition}
\newtheorem{lem}[thm]{Lemma}

\newtheorem{defi}[thm]{Definition}
\newtheorem{rmk}[thm]{Remark}
\newtheorem{cor}[thm]{Corollary}

\newtheorem{prop-defi}[thm]{Proposition-Definition}
\newtheorem{thm-defi}[thm]{Theorem-Definition}
\newtheorem{lem-defi}[thm]{Lemma-Definition}

\newtheorem{assum}[thm]{Assumption}
\newtheorem{conj}[thm]{Conjecture}
\newtheorem{exam}[thm]{Example}

\newcommand{\ssslash}{/\!\!/}

\newcommand{\aA}{\mathcal{A}}

\newcommand{\cC}{\mathcal{C}}
\newcommand{\dD}{\mathcal{D}}
\newcommand{\eE}{\mathcal{E}}
\newcommand{\fF}{\mathcal{F}}

\newcommand{\hH}{\mathcal{H}}
\newcommand{\iI}{\mathcal{I}}

\newcommand{\kK}{\mathcal{K}}
\newcommand{\lL}{\mathcal{L}}
\newcommand{\mM}{\mathcal{M}}
\newcommand{\nN}{\mathcal{N}}
\newcommand{\oO}{\mathcal{O}}
\newcommand{\pP}{\mathcal{P}}

\newcommand{\rR}{\mathcal{R}}
\newcommand{\sS}{\mathcal{S}}
\newcommand{\tT}{\mathcal{T}}

\newcommand{\vV}{\mathcal{V}}
\newcommand{\wW}{\mathcal{W}}
\newcommand{\xX}{\mathcal{X}}
\newcommand{\yY}{\mathcal{Y}}
\newcommand{\zZ}{\mathcal{Z}}

\newcommand{\fM}{\mathfrak{M}}
\newcommand{\fN}{\mathfrak{N}}

\newcommand{\fP}{\mathfrak{P}}

\newcommand{\fS}{\mathfrak{S}}

\newcommand{\fU}{\mathfrak{U}}

\newcommand{\fZ}{\mathfrak{Z}}

\newcommand{\Supp}{\mathop{\rm Supp}\nolimits}
\newcommand{\Hom}{\mathop{\rm Hom}\nolimits}

\newcommand{\dR}{\mathbf{R}}

\newcommand{\NS}{\mathop{\rm NS}\nolimits}

\newcommand{\Pic}{\mathop{\rm Pic}\nolimits}

\newcommand{\id}{\textrm{id}}

\newcommand{\ch}{\mathop{\rm ch}\nolimits}
\newcommand{\rk}{\mathop{\rm rk}\nolimits}

\newcommand{\Ext}{\mathop{\rm Ext}\nolimits}
\newcommand{\Spec}{\mathop{\rm Spec}\nolimits}
\newcommand{\rank}{\mathop{\rm rank}\nolimits}
\newcommand{\Coh}{\mathop{\rm Coh}\nolimits}

\newcommand{\ev}{\mathop{\rm ev}\nolimits}

\newcommand{\us}{\mathchar`-\rm{us}}
\newcommand{\sss}{\mathchar`-\rm{ss}}

\newcommand{\cneq}{\mathrel{\raise.095ex\hbox{:}\mkern-4.2mu=}}
\newcommand{\eqcn}{\mathrel{=\mkern-4.5mu\raise.095ex\hbox{:}}}

\newcommand{\Cok}{\mathop{\rm Cok}\nolimits}

\newcommand{\Aut}{\mathop{\rm Aut}\nolimits}

\newcommand{\Cone}{\mathop{\rm Cone}\nolimits}

\newcommand{\PPer}{\mathop{\rm Per}\nolimits}

\newcommand{\pH}{\mathop{^{{p}}\mathcal{H}}\nolimits}
\newcommand{\pcok}{\mathop{^{{p}}\mathrm{Cok}}\nolimits}

\newcommand{\dg}{\mathrm{dg}}

\newcommand{\Imm}{\mathop{\rm Im}\nolimits}

\newcommand{\Ker}{\mathop{\rm Ker}\nolimits}

\newcommand{\GL}{\mathop{\rm GL}\nolimits}

\newcommand{\cl}{\mathop{\rm cl}\nolimits}
\newcommand{\Ind}{\mathop{\rm Ind}\nolimits}
\newcommand{\ind}{\mathop{\rm ind}\nolimits}
\newcommand{\MF}{\mathop{\rm MF}\nolimits}

\newcommand{\Crit}{\mathop{\rm Crit}\nolimits}
\newcommand{\bS}{\mathbb{S}}
\newcommand{\bU}{\mathbb{U}}

\newcommand{\wt}{\mathrm{wt}}

\newcommand{\coh}{\mathrm{coh}}
\newcommand{\qcoh}{\mathrm{qcoh}}
\newcommand{\Dbc}{D^b_{\rm{coh}}}

\newcommand{\inclusion}{\ar@<-0.3ex>@{^{(}->}[r]}
\newcommand{\upinclusion}{\ar@<-0.3ex>@{^{(}->}[u]}
\newcommand{\leinclusion}{\ar@<-0.3ex>@{^{(}->}[l]}
\newcommand{\doinclusion}{\ar@<-0.3ex>@{^{(}->}[d]}
\newcommand{\diasquare}{\ar@{}[rd]|\square}

\newcommand{\bff}{\mathbf{f}}

\newcommand{\Filt}{\mathrm{Filt}}
\newcommand{\Grad}{\mathrm{Grad}}

\newcommand{\ffS}{\mathfrak{S}}
\newcommand{\ffZ}{\mathfrak{Z}}

\newcommand{\C}{\mathbb{C}^{\ast}}

\newcommand{\dDT}{\mathcal{DT}}

\makeatletter
\newcommand{\colim@}[2]{%
	\vtop{\m@th\ialign{##\cr
			\hfil$#1\operator@font colim$\hfil\cr
			\noalign{\nointerlineskip\kern1.5\ex@}#2\cr
			\noalign{\nointerlineskip\kern-\ex@}\cr}}%
}
\newcommand{\colim}{%
	\mathop{\mathpalette\colim@{}}\nmlimits@
}
\makeatother

\makeatletter
\renewcommand{\theequation}{%
	\thesection.\arabic{equation}}
\@addtoreset{equation}{section}
\makeatother

\begin{document}
	
	\begin{abstract}
We show the existence of semiorthogonal decompositions of 
Donaldson-Thomas categories for $(-1)$-shifted cotangent 
derived stacks associated with $\Theta$-stratifications
on them. Our main result gives an analogue 
of window theorem for categorical 
DT theory, which has applications 
to d-critical analogue of D/K equivalence conjecture, i.e. 
existence of fully-faithful functors (equivalences) under 
d-critical flips (flops). 
As an example of applications, we show the 
existence of fully-faithful functors
of DT categories
for Pandharipande-Thomas stable pair 
moduli spaces
under wall-crossing 
at super-rigid rational curves
for any relative reduced curve classes. 
	\end{abstract}
	
	\maketitle

	\setcounter{tocdepth}{1}
	\tableofcontents
	
	\section{Introduction}
	\subsection{Background}
	In~\cite{TocatDT}, the author 
	introduced $\C$-equivariant 
	Donaldson-Thomas (DT) categories associated with 
	$(-1)$-shifted cotangents $\Omega_{\fM}[-1]$ over quasi-smooth derived 
	stacks $\fM$ with stability conditions. 
	They are defined to be certain 
	singular support quotients of derived categories
	of coherent sheaves on $\fM$, 
which via Koszul duality are regarded as gluing of
dg-categories of $\C$-equivariant 
factorizations of super-potentials giving 
local d-critical charts of $\Omega_{\fM}[-1]$. 
	The DT category depends on a choice of a stability condition,
	and it is an interesting problem to investigate wall-crossing phenomena
	of DT categories under a change of stability conditions. 
	The above problem is motivated 
	by a categorification of wall-crossing formula of 
	DT invariants~\cite{JS, K-S}, and also a d-critical analogue of D/K equivalence 
	conjecture in birational geometry~\cite{B-O2, MR1949787}. 
	A sequence of conjectures on wall-crossing of DT categories 
	is proposed in~\cite{TocatDT}, in the context of wall-crossing of 
	one dimensional stable sheaves and Pandharipande-Thomas moduli 
	spaces of stable pairs~\cite{PT}. 
	Some fundamental results on DT categories toward the above 
	conjectures are also established in~\cite{TocatDT}, including 
	window theorem for DT categories which generalizes window theorem 
	originally developed in GIT quotient stacks~\cite{MR3327537, MR3895631}. 
	
	In~\cite{TocatDT}, we established two kinds of window theorem 
	for DT categories, by constructing window 
	subcategories in two different ways: 
	\begin{enumerate}
		\item Gluing of magic window subcategories 
		by Halpern-Leistner-Sam~\cite{HLKSAM} (see~\cite[Section~5]{TocatDT}). 
		
		\item Semiorthogonal summands of DT categories with complements 
		described by Porta-Sala categorified Hall products~\cite{PoSa}
		(see~\cite[Section~6]{TocatDT}). 
		\end{enumerate}
	The first window 
	subcategory is independent of a stability condition, but 
	only available for symmetric derived stacks (or more generally 
	derived stacks with symmetric structures). 
	The second window subcategory depends on a stability 
	condition, and is more close to the original 
	window subcategories for GIT quotient stacks~\cite{MR3327537, MR3895631}.
	A drawback of the second one is that it is available
	only in the situation that the categorified 
	Hall products are a priori given, so in~\cite[Section~6]{TocatDT}
	we only proved the window theorem for 
	some special cases of PT moduli spaces.  
	
	The purpose of this paper is to generalize the second approach 
	in~\cite[Section~6]{TocatDT}, 
	using Halpern-Leistner's theory of $\Theta$-stratifications~\cite{Halpinstab}
	instead of categorified Hall products~\cite{PoSa}. 
	Namely we show the existence of semiorthogonal decompositions 
	of DT categories associated with $\Theta$-stratifications of 
	$(-1)$-shifted cotangent derived stacks, whose 
	semiorthogonal summands give window subcategories of the second type 
	as above. 
The categorified Hall products are
	naturally interpreted in the context of $\Theta$-stratifications, 
	and it turns out that
	many of the proofs of our
	 main theorem follow 
	  from 
	  direct generalizations of the arguments in~\cite[Section~6]{TocatDT}.
	  However we also need to give several 
	  modifications and prove some fundamental properties of 
	  $\Theta$-stacks, e.g. 
	  descriptions of stacks of filtered objects of 
	$(-1)$-shifted cotangents in terms of $(-2)$-shifted conormals
	(see~Section~\ref{sec:filtst}). 
	Furthermore this generalization has more potentials for applications. 
	For example, we show the 
	existence of fully-faithful functors
	of DT categories for PT
	stable pair moduli spaces
	under wall-crossing 
	at super-rigid rational curves
	for any relative reduced curve classes, 
	which gives an evidence of the conjectures in~\cite{TocatDT} 
	that is not covered in \textit{loc.~cit.~}.

	\subsection{Semiorthogonal decompositions of DT categories}
	Let $\fM$ be a quasi-smooth and QCA derived stack over $\mathbb{C}$
	with classical truncation $\mM=t_0(\fM)$, and
	 \begin{align*}\Omega_{\fM}[-1] \cneq \Spec_{\fM}S_{\oO_{\fM}}(\mathbb{T}_{\fM}[1])
	 \end{align*} the $(-1)$-shifted cotangent 
	derived stack of $\fM$. 
Given a $\C$-invariant open substack
\begin{align}\label{intro:Nss}
	\nN^{\rm{ss}} \subset \nN \cneq t_0(\Omega_{\fM}[-1])
	\end{align}
and its complement $\zZ \subset \nN$, 
we have the triangulated subcategory 
$\cC_{\zZ} \subset \Dbc(\fM)$ 
consisting of objects whose singular supports~\cite{MR3300415}
are contained in $\zZ$. 
The $\C$-equivariant DT category in~\cite{TocatDT}
is modeled by the Verdier quotient (see Definition~\ref{defi:DTcat})
\begin{align}\label{intro:model}
	\dDT^{\C}(\nN^{\rm{ss}}) \cneq \Dbc(\fM)/\cC_{\zZ}. 
	\end{align}
The $\mathbb{Z}/2$-periodic version is also introduced in~\cite{TocatDT2}, whose basic model 
is (see Definition~\ref{defi:DTcat2})
\begin{align}\label{intro:model2}
	\dDT^{\mathbb{Z}/2}(\nN^{\rm{ss}}) \cneq 
	\Dbc(\fM_{\epsilon})/\cC_{\zZ_{\epsilon}}
	\end{align}
where $\fM_{\epsilon} \cneq \fM \times \Spec \mathbb{C}[\epsilon]$
for $\deg(\epsilon)=-1$ 
and $\zZ_{\epsilon} \cneq (\zZ \times \mathbb{A}^1) \cup (\nN \times \{0\})$.  
The above definitions are based on Koszul duality equivalences
which, in the case that $\fM$ is a quasi-smooth affine derived scheme, 
give equivalences of these categories with $\C$-equivariant 
or $\mathbb{Z}/2$-periodic triangulated categories of factorizations. 
The latter one (\ref{intro:model2}) is recovered from the former one (\ref{intro:model})
under the compact generation of $\Ind \cC_{\zZ}$
up to idempotent completion 
(see~Theorem~\ref{thm:intro:compare}), so 
in this paper we only focus on the $\C$-equivariant DT category (\ref{intro:model}). 

We are interested in the case that the open substack (\ref{intro:Nss}) is the 
semistable locus with respect to a certain stability condition. 
For a choice of $\lL \in \Pic(\mM)_{\mathbb{R}}$ with 
$l=c_1(\lL)$, 
the Hilbert-Mumford criterion with respect to 
maps from the $\Theta$-stack 
\begin{align*}
	\Theta \cneq [\mathbb{A}^1/\C] \to \nN
\end{align*}
defines the $l$-semistable locus 
$\nN^{l\sss} \subset \nN$, which generalizes Mumford's GIT semistable locus for  GIT quotient 
stacks depending on linearizations.  
The above viewpoint of defining $l$-semistable locus with respect to 
maps from the $\Theta$-stack is due to Halpern-Leistner~\cite{Halpinstab}, 
who further 
generalizes Kempf-Ness stratifications for GIT quotient 
stacks to $\Theta$-stratifications for more general stacks $\xX$.
In several cases, $\Theta$-stratifications 
are constructed from numerical data 
$(l, b)$ for $l \in H^2(\xX, \mathbb{R})$ 
and a positive definite $b \in H^4(\xX, \mathbb{R})$. 
In our situation, by taking $l \in H^2(\mM, \mathbb{R})$ as above, 
positive definite
$b \in H^4(\mM, \mathbb{R})$, and pulling them back 
to $\nN$, we have the 
$\Theta$-stratification 
\begin{align}\label{intro:strata}
	\nN=\sS_1^{\Omega} \sqcup \cdots \sqcup \sS_N^{\Omega} \sqcup \nN^{l\sss} 
	\end{align}
with center $\zZ_{i}^{\Omega} \subset \sS_i^{\Omega}$. 
We show that each $\zZ_i^{\Omega}$ is an open substack 
of $(-1)$-shifted cotangent over the stack of 
graded objects of $\fM$ (see Subsection~\ref{subsec:theta-1}). 
So we have the DT category $\dDT^{\C}(\zZ_i^{\Omega})$, together with 
the decomposition into $\C$-weight part
with respect to the canonical $B\C$-action on $\zZ_i^{\Omega}$
\begin{align*}
	\dDT^{\C}(\zZ_i^{\Omega})=\bigoplus_{j\in \mathbb{Z}}
	\dDT^{\C}(\zZ_{i}^{\Omega})_{\wt=j}. 
	\end{align*}
The following is the main result in this paper: 
\begin{thm}\label{intro:main}\emph{(Corollary~\ref{cor:sod})}
	Suppose that $\mM$ admits a good moduli space $\mM \to M$
	satisfying the formal neighborhood theorem (see Definition~\ref{defi:formalneigh}). 
	 Then for each choice of 
	$m_i \in \mathbb{R}$ for $1\le i\le N$, 
	there exists a semiorthogonal decomposition
	\begin{align*}
		\dDT^{\C}(\nN)=\langle
			 \dD_{1, < m_1}, \ldots, \dD_{N, < m_N}, 
		\wW_{m_{\bullet}}^l, \dD_{N, \ge m_N}, \ldots, \dD_{1, \ge m_1} \rangle  
		\end{align*}
	satisfying the followings: 
	\begin{enumerate}
		\item There exist semiorthogonal decompositions of the form
	\begin{align*}
		&\dD_{i, < m_i}=\left\langle \ldots, \dDT^{\C}(\zZ_i^{\Omega})_{\wt=\lceil m_i \rceil-3}, 
		\dDT^{\C}(\zZ_i^{\Omega})_{\wt=\lceil m_i \rceil-2}, 
		\dDT^{\C}(\zZ_i^{\Omega})_{\wt=\lceil m_i \rceil-1}
		\right\rangle, \\
		&\dD_{i, \ge m_i}=\left\langle 
		\dDT^{\C}(\zZ_i^{\Omega})_{\wt=\lceil m_i \rceil}, 
		\dDT^{\C}(\zZ_i^{\Omega})_{\wt=\lceil m_i \rceil+1}, 
			\dDT^{\C}(\zZ_i^{\Omega})_{\wt=\lceil m_i \rceil+2}, 
		\ldots \right\rangle.
		\end{align*}
	\item The composition functor 
	\begin{align*}
		\wW_{m_{\bullet}}^l \hookrightarrow \dDT^{\C}(\nN) \twoheadrightarrow 
		\dDT^{\C}(\nN^{l\sss})
		\end{align*}
	is an equivalence.
	\end{enumerate}
	\end{thm}
The subcategory $\wW_{m_{\bullet}}^{l}$ gives a
desired window subcategory mentioned in the 
previous subsection. 
Halpern-Leistner~\cite{HalpK32}
proved the existence of semiorthogonal decomposition 
of $\Dbc(\fM)(=\dDT^{\C}(\nN))$
associated with $\Theta$-stratifications of 
$\fM$ (not of $\nN$), under some assumption 
on weights of the obstruction spaces of $\fM$
at each center of $\Theta$-strata (see~\cite[Theorem~2.3.1]{HalpK32}). 
The $\Theta$-stratification (\ref{intro:strata})
does not necessary a pull-back of the 
$\Theta$-stratification of $\fM$, 
e.g. the semistable locus $\nN^{l\sss}$ may be
strictly bigger than the pull-back of $\mM^{l\sss}$. 
So the result of~\cite[Theorem~2.3.1]{HalpK32} does not imply 
Theorem~\ref{intro:main}. 
On the other hand if the above weight condition 
is satisfied, then the $\Theta$-stratification on $\fM$ pulls back 
to a $\Theta$-stratification on $\nN$, and the semiorthogonal 
decomposition in Theorem~\ref{intro:main} recovers the 
semiorthogonal decomposition in~\cite[Theorem~2.3.1]{HalpK32} (see Remark~\ref{rmk:sod}). 
In the case that $\fM$ is the derived moduli stack of one dimensional 
sheaves on surfaces, P{\u{a}}durairu~\cite{Tudor2}
recently constructs semi-orthogonal decomposition 
of $\Dbc(\fM)$ which is closely related to the one 
in Theorem~\ref{intro:main}, with the aim of constructing K-theoretic BPS 
Lie algebra of a surface whose quiver with super-potential version 
is obtained in~\cite{Tudor}. 


The DT category $\dDT^{\C}(\nN^{l\sss})$
depends on $l$, 
and we are interested in its behavior under a
change of $l$. 
We use the result of Theorem~\ref{intro:main}
to 
compare 
$\dDT^{\C}(\nN^{l\sss})$ under wall-crossing of $l$. 
We take another $\lL' \in \Pic(\mM)_{\mathbb{R}}$ with $l'=c_1(\lL')$
and 
set $l_{\pm}=l\pm \varepsilon l'$
for $0<\varepsilon \ll 1$. 
Then we have the wall-crossing diagram 
for the good moduli space $\nN^{l\sss} \to N^{l\sss}$
\begin{align*}
	\xymatrix{
N^{l_{+}\sss} \ar[rd] & & N^{l_{-}\sss}	\ar[ld] \\
& N^{l\sss}.  &
}
	\end{align*}
In~\cite{Toddbir}, we introduced the notion of d-critical flips (flops)
which are analogue of usual flips (flops) in birational 
geometry for Joyce's d-critical loci~\cite{JoyceD}. 
If the above diagram is a d-critical flip (flop), then 
as an analogy of D/K equivalence conjecture~\cite{B-O2, MR1949787} 
we expect the existence of a fully-faithful functor (an equivalence)
\begin{align*}
	\dDT^{\C}(\nN^{l_- \sss}) \hookrightarrow (\stackrel{\sim}{\to}) \dDT^{\C}(\nN^{l_+ \sss}). 
\end{align*}
We impose some assumption (see Assumption~\ref{assum:Wx}) 
on tangent complexes at closed points in $\nN^{l\sss}$, 
which is typically satisfied for d-critical flips, and prove the following: 
	\begin{thm}\emph{(Theorem~\ref{thm:inclu})}\label{intro:thm:inclu}
	Under Assumption~\ref{assum:Wx}, 
	we have the inclusion 
	$\wW_{m_{\bullet}^{-}}^{l_{-}} \subset 
	\wW_{m_{\bullet}^{+}}^{l_{+}}$. 
	In particular, we have the fully-faithful functor 
	\begin{align*}
		\dDT^{\C}(\nN^{l_{-}\sss})
		\hookrightarrow \dDT^{\C}(\nN^{l_{+}\sss}). 
	\end{align*}
\end{thm}

	\subsection{Application: wall-crossing at $(-1, -1)$-curve}
In~\cite{TocatDT}
we proposed several conjectures on wall-crossing of DT categories 
for Pandharipande-Thomas stable pair moduli spaces~\cite{PT}
on local surfaces, and 
proved them in several cases, e.g. 
MNOP/PT correspondence for reduced curve classes. 
As an application of Theorem~\ref{intro:thm:inclu}, we give a
further evidence to the above conjectures which 
is not covered in~\cite{TocatDT}. 

Let $S$ be a smooth projective surface over $\mathbb{C}$, 
and 
\begin{align*}
	\pi \colon X \cneq \mathrm{Tot}_S(\omega_S) \to S
	\end{align*}
the total space of canonical line bundle of $S$, 
which is a non-compact Calabi-Yau 3-fold called \textit{local surface}. 
A \textit{PT stable pair} is a pair $(F, s)$, where $F$ is a compactly 
supported one dimensional coherent sheaf on $X$ 
and $s \colon \oO_X \to F$ is surjective in dimension one. 
For $(\beta, n) \in \mathrm{NS}(S) \oplus \mathbb{Z}$, 
we have the moduli space of PT stable pairs
\begin{align*}
	P_n(X, \beta)=\{(F, s) : \mbox{ PT stable pair with }
	(\pi_{\ast}[F], \chi(F))=(\beta, n)\},
	\end{align*}
which is a quasi-projective scheme. 
The moduli space of PT stable pairs is regarded as a 
moduli space of some stable objects in the derived category. 
Indeed for a fixed ample divisor $H$ on $S$, there is a one 
parameter family of stability conditions
on the abelian category of D0-D2-D6 bound states, 
denoted by $\mu_t^{\dag}$-stability for $t \in \mathbb{R}$, 
whose $t \to \infty$ limit recovers the 
PT theory (see Subsection~\ref{subsec:PTwall}). 
The numerical PT invariants and their 
wall-crossing formula
play a key role in the study of 
curve counting invariants on Calabi-Yau 3-folds (see~\cite{Tsurvey}). 
	
We denote by $P_n^t(X, \beta)$ the moduli space of $\mu_t^{\dag}$-semistable 
objects with numerical class $(\beta, n)$. It was observed in~\cite{Toddbir}
that the wall-crossing diagram at $t>0$
\begin{align}\label{intro:dia:PT}
	\xymatrix{
		P_n^{t_+}(X, \beta) \ar[rd] & & P_n^{t_-}(X, \beta)\ar[ld] \\
		& P_n^t(X, \beta) &
	}
\end{align}
is a \textit{d-critical flip} (see~\cite[Theorem~9.13]{Toddbir}).
Here we refer to~\cite[Subsection~1.3.2]{TocatDT} for an importance 
of the above wall-crossing in the context of numerical PT invariants. 
In~\cite{TocatDT}, we introduced the categorical DT theory 
for D0-D2-D6 bound states
\begin{align*}
	\dDT^{\C}(P_n^t(X, \beta)), \ t \notin W
	\end{align*}
following the basic model (\ref{intro:model}), 
where $W \subset \mathbb{R}$ is the set of walls. 
Based on the  of D/K equivalence conjecture~\cite{B-O2, MR1949787}, 
the following conjecture was proposed in~\cite{TocatDT}: 
\begin{conj}\emph{(\cite[Conjecture~4.24]{TocatDT})}\label{intro:conj:FF:PT}
	In the diagram (\ref{intro:dia:PT}),
	there exists a fully-faithful functor 
	\begin{align*}
		\dDT^{\C}(P_n^{t_-}(X, \beta)) \hookrightarrow 
		\dDT^{\C}(P_n^{t_+}(X, \beta)).
	\end{align*} 
\end{conj}

Let $C \subset S$ be a $(-1)$-curve, and 
$f \colon S \to T$ a birational 
contraction which contracts $C$ to a smooth surface 
$T$. 
Note that $C$ is a $(-1, -1)$-curve inside $X$, which is 
contracted to a conifold singularity by a flopping contraction 
$X \to Y$. 
We take an ample divisor $H$ on $S$ of the 
following form
\begin{align*}
	H=a f^{\ast}h-C, \ a \gg 0
	\end{align*}
where $h$ is an ample divisor on $T$. As an application of
Theorem~\ref{intro:thm:inclu}, we prove the following: 
\begin{thm}\label{intro:thm2}\emph{(Theorem~\ref{thm:-1-1})}
		Conjecture~\ref{intro:conj:FF:PT} is true
		 if $t \ge 1$ and $\beta$ is $f$-reduced, i.e. 
		 $f_{\ast}\beta$ is a reduced curve class on $T$. 
	\end{thm}

In the situation of Theorem~\ref{intro:thm2}, a strictly polystable object 
at the wall is a direct sum of stable objects
with line bundles on $C$. 
The wall-crossing in Theorem~\ref{intro:thm2} is the one which appeared 
in~\cite{Tcurve2, Calab}, 
where the flop transformation of PT invariants 
was proved.
The sequence of wall-crossing diagrams in this case 
is a d-critical minimal model program, which connects
PT stable pair moduli space $P_n(X, \beta)$ with 
non-commutative stable pair moduli space $P_n^{\rm{nc}}(X, \beta)$
(see Subsection~\ref{subsec:percoh}), 
\begin{align*}
	P_n(X, \beta)=P_n^{(k)}(X, \beta) \dashrightarrow 
	P_n^{(k-1)}(X, \beta) \dashrightarrow \cdots \dashrightarrow P_n^{(0)}(X, \beta)=P_n^{\rm{nc}}(X, \beta).
\end{align*}
Here $k\gg 0$ and $P_n^{\rm{nc}}(X, \beta)$ is isomorphic to
 the moduli space of 
PT stable pairs for 
the non-commutative scheme $(Y, A_X)$, 
which is a non-commutative crepant resolution of $Y$
given in~\cite{MR2057015}. 
So Theorem~\ref{intro:thm2} implies a sequence of 
fully-faithful functors  
\begin{align*}
	\dDT^{\C}(P_n^{\rm{nc}}(X, \beta)) \hookrightarrow 
	\cdots \hookrightarrow 
	\dDT^{\C}(P_n^{(k-1)}(X, \beta))
	\hookrightarrow \dDT^{\C}(P_n(X, \beta)). 
\end{align*}
We also remark that Theorem~\ref{intro:thm2} in the case that $f_{\ast}\beta=0$
is proved in~\cite[Theorem~4.3.5]{TocatDT}, where 
PT moduli spaces are global critical loci so the original 
window theorem for GIT quotient stacks
 is enough to prove the result. 
In the situation that $f_{\ast}\beta$ is only a reduced class, 
we no longer have a global critical locus description of PT 
moduli spaces, and we need the result of Theorem~\ref{intro:main}.

\subsection{Acknowledgements}
This author is grateful to Daniel Halpern-Leistner for discussions 
on his works~\cite{MR3327537, Halpinstab, HalpK32}, 
and also thanks to Yalong Cao, Hsueh-Yung Lin and 
Wai-kit Yeung for helpful discussions. 
	The author is supported by World Premier International Research Center
	Initiative (WPI initiative), MEXT, Japan, and Grant-in Aid for Scientific
	Research grant (No.~19H01779) from MEXT, Japan.

	\subsection{Notation and convention}\label{subsec:notation}
	In this paper, all the schemes or (derived) stacks
	are locally of finite presentation over 
	$\mathbb{C}$, except formal fibers along with 
	good moduli space morphisms. 
	For a scheme or derived stack $Y$
	and a quasi-coherent sheaf $\fF$ on it, we denote by 
	$S_{\oO_Y}(\fF)$ its symmetric product
	$\oplus_{i\ge 0} \mathrm{Sym}^i_{\oO_Y}(\fF)$. 
	We omit the subscript $\oO_Y$ if it is clear from the context. 
	For a derived stack $\mathfrak{M}$, we always 
	denote by $t_0(\mathfrak{M})$ the underived
	stack given by the truncation. 
	For an algebraic group $G$ which acts on $Y$, 
	we denote by $[Y/G]$ the associated 
	quotient stack. 
	
	All the dg-categories or triangulated categories 
	are defined over 
	$\mathbb{C}$. 
	For a triangulated category $\dD$ 
	and a set of objects $\sS \subset \dD$, 
	we denote by $\langle \sS \rangle_{\rm{ex}}$ the extension 
	closure, i.e. the smallest extension closed subcategory of $\dD$
	which contains $\sS$. 
		For
		a triangulated subcategory $\dD' \subset \dD$, 
	we denote by 
	$\dD/\dD'$
	its Verdier quotient. 
	In the case that $\dD$ is a dg-category 
	and $\dD' \subset \dD$ is a dg-subcategory, 
	its Drinfeld dg-quotient~\cite{MR3037900} 
	is also denoted by $\dD/\dD'$. 
		The category of dg-categories
	$\mathrm{dgCat}$
	consists of dg-categories over $\mathbb{C}$
	with morphisms given by dg-functors. 
	By Tabuada~\cite{Tab}, there is a cofibrantly generated 
	model category structure on 
	$\mathrm{dgCat}$, 
	whose localization by weak equivalences
	is denoted by
	$\mathrm{Ho}(\mathrm{dgCat})$. 
	An equivalence between dg-categories
	is defined to be 
	an isomorphism 
	in $\mathrm{Ho}(\mathrm{dgCat})$. 
	
	For a dg-category $\dD$, 
	we denote by $\Ind \dD$ 
	its dg-categorical 
	ind-completion of $\dD$ (denoted as $\widehat{\dD}$
	in~\cite[Section~7]{Todg}, also see~\cite[Section~5.3.5]{Ltopos} in the 
	context of $\infty$-category). 
	When we discuss limits or ind-completions, 
	we (implicitly or explicitly) take these functors on dg-enhancements
	or $\infty$-categorical enhancements 
	(which will be obviously given in the context)
	and then take their homotopy categories. 
	
	Let $G$ be an algebraic group, 
$Y$ be a representation of $G$
and $\lambda \colon \C \to G$ a one parameter subgroup. 
We denote by $\langle \lambda, Y \rangle$ 
the sum of $\lambda$-weights of $Y$, 
which equals to the $\lambda$-weight of $\det (Y)$. 
By regarding $Y$ as a $\C$-representation via $\lambda$, 
we denote by $Y^{\lambda \ge 0} \subset Y$ the 
direct summand consisting of non-negative $\C$-weights, 
and $\langle \lambda, Y^{\lambda \ge 0} \rangle$ is also 
defined as above. 

For an abelian category $\aA$ and a collection of objects 
$(E_1, \ldots, E_k)$ in $\aA$, its \textit{Ext-quiver}
is the quiver whose vertex set is $\{1, \ldots, k\}$, 
and the number of arrows from $i$ to $j$ 
is $\dim \Ext^1(E_i, E_j)$.

	\section{Semiorthogonal decompositions of factorization categories: review}
	In this section, we review the dg or triangulated 
	categories of factorizations of super-potentials, 
	and their semiorthogonal decompositions
	associated with Kempf-Ness stratifications
	proved in~\cite{MR3327537, MR3895631}. 
	The above semiorthogonal decompositions 
	are local models of semiorthogonal decompositions for 
	DT categories. 
		\subsection{Review of factorization categories}\label{subsec:fact}
	Here we review the theory of factorizations 
	associated with super-potentials. The basic references are~\cite{Ornonaff, MR3366002, MR3112502}. 
	
	Let $\xX$ be a noetherian smooth algebraic stack over $\mathbb{C}$, 
	$\lL \to \xX$ a line bundle and 
	$w \in \Gamma(\xX, \lL^{\otimes 2})$ a global section. 
	A \textit{(quasi) coherent factorization} of $w$ consists of 
	\begin{align*}
		(\pP, d_{\pP}), \ d_{\pP} \colon \pP \to \pP \otimes \lL
	\end{align*}
	where $\pP$ is a (quasi) coherent sheaf on $\xX$ and 
	$d_{\pP}$ is a morphism of (quasi) coherent sheaves
	satisfying $d_{\pP}^2=w$. 
	The category of factorizations of $w$ naturally forms
	a dg-category
	denoted by $\underline{\MF}_{\star}(\xX, w)_{\rm{dg}}$ for
	$\star \in \{\qcoh, \coh\}$, whose 
	homotopy category 
	$\mathrm{HMF}_{\star}(\xX, w)$
	is a triangulated category. 
	Let 
	\begin{align}\notag
		\mathrm{Acy}_{\star}^{\rm{abs}} \subset \mathrm{HMF}_{\star}(\xX, w)
	\end{align}
	be the minimal thick triangulated subcategory 
	which contains totalizations of short exact sequences of 
	(quasi) coherent factorizations of $w$. 
	An object in $\mathrm{Acy}_{\star}^{\rm{abs}}$ is called 
	\textit{absolutely acyclic}. 
	The triangulated category of factorizations of $w$ is defined by 
	the Verdier quotient 
	\begin{align*}
		\MF_{\star}(\xX, w) \cneq 
		\mathrm{HMF}_{\star}(\xX, w)/\mathrm{Acy}^{\rm{abs}}_{\star}, \ 
		\star \in \{\qcoh, \coh\}. 
	\end{align*}
	It admits a natural dg-enhancement by taking the Drinfeld quotient 
	\begin{align}\label{MF:dg}
		\MF_{\star}(\xX, w)_{\rm{dg}} \cneq 
		\underline{\MF}_{\star}(\xX, w)_{\rm{dg}}/\mathrm{Acy}^{\rm{abs}}_{\star, \rm{dg}}, \
		\star \in \{\qcoh, \coh\}
	\end{align}
	where $\mathrm{Acy}^{\rm{abs}}_{\star, \rm{dg}}$
	is the full dg-subcategory 
	of $\underline{\MF}_{\star}(\xX, w)_{\rm{dg}}$ consisting of 
	absolutely acyclic objects in the homotopy category. 
	
	Let 
	$\zZ \subset \xX$ be
	a closed substack. 
	We define 
	\begin{align}\label{MF:supp}
		\MF_{\star}(\xX, w)_{\zZ} \cneq 
		\Ker(\MF_{\star}(\xX, w) \stackrel{j^{\ast}}{\to} \MF_{\star}(\xX
		\setminus \zZ, w|_{\xX \setminus \zZ})). 
	\end{align}
Here $j \colon \xX \setminus \zZ \hookrightarrow \xX$ is the open 
immersion. 
	Then we have the equivalence (cf.~\cite[Theorem~1.10]{MR3366002})
	\begin{align}\label{equiv:supporZ}
		\MF_{\star}(\xX, w)/\MF_{\star}(\xX, w)_{\zZ}
		\stackrel{\sim}{\to} \MF_{\star}(\xX \setminus \zZ, w|_{\xX \setminus \zZ}). 
	\end{align}
	Let $\Crit(w) \subset \xX$ be the critical locus. We also have the equivalence
	(cf.~\cite[Corollary~5.3]{MR3112502})
	\begin{align}\notag
		\MF_{\star}(\xX, w)_{\Crit(w)} 
		\stackrel{\sim}{\to} \MF_{\star}(\xX, w). 
	\end{align}
In particular by setting 
$\zZ'=\zZ \cap \Crit(w)$, we have the equivalence 
\begin{align}\label{equiv:supp3}
	\MF^{\C}_{\coh}(\xX \setminus \zZ', w|_{\xX \setminus \zZ'}) \stackrel{\sim}{\to}
		\MF^{\C}_{\coh}(\xX \setminus \zZ, w|_{\xX \setminus \zZ}). 
	\end{align}
	
	We will use the following two special versions of factorization categories. 
	Let $\yY$ be a smooth stack and set $\xX=\yY \times B\mu_2$, 
	and $\lL$ to be the line bundle on $\xX$ induced by the 
	weight one $\mu_2$-character. 
	Then $\lL^{\otimes 2} \cong \oO_{\xX}$, so any 
	regular function $w \colon \yY \to \mathbb{C}$
	determines a global section $w \in \Gamma(\xX, \lL^{\otimes 2})$. 
	We set
	\begin{align*}
		\MF_{\star}^{\mathbb{Z}/2}(\yY, w) \cneq 
		\MF_{\star}(\yY \times B\mu_2, w), \ 
		\star \in \{\qcoh, \coh\}. 
	\end{align*}
	The above triangulated categories are $\mathbb{Z}/2$-periodic, i.e. 
	$[2] \cong\id$. 
	When $\yY$ is an affine scheme, the 
	above triangulated category is equivalent to 
	Orlov's triangulated category of matrix factorizations~\cite{Orsin}. 
	
	Let $\C$ acts on a smooth stack $\yY$ and set 
	$\xX=[\yY/\C]$, and $\lL$ to be the line bundle on $\xX$
	induced by the weight one $\C$-character. 
	For a global section $w \in \Gamma(\xX, \lL^{\otimes 2})$, 
	we set
	\begin{align*}
		\MF_{\star}^{\C}(\yY, w) \cneq \MF_{\star}([\yY/\C], w), \ 
		\star \in \{\qcoh, \coh\}. 
	\end{align*}
	For example, 
	we will use the above 
	construction for $\yY=[Y/G]$ for a
	noetherian scheme $Y$ with an action of an algebraic group $G$
	and an action of $\C$ which commutes with the $G$-action, 
	and $w \colon Y \to \mathbb{C}$ is a $G$-invariant 
	function with $\C$-weight two. 
	Let 
	$\lambda \colon \C \to G$ be a one parameter subgroup 
	contained in the center of $G$ and the induced $\C$-action 
	 on $Y$ trivial. 
	Then we have the decomposition for $\star \in \{\coh, \qcoh\}$
	\begin{align*}
		\MF_{\star}^{\C}([Y/G], w)=
		\bigoplus_{j \in \mathbb{Z}}
		\MF_{\star}^{\C}([Y/G], w)_{\lambda \mathchar`- \wt= j}
	\end{align*}\
	where $\MF_{\star}^{\C}([Y/G], w)_{\lambda \mathchar`- \wt= j}$ is the weight 
	$j$-part with respect to $\lambda$. 
	We define the subcategory 
	\begin{align*}
		\MF_{\coh}^{\C}([Y/G], w)_{\lambda \mathchar`- \rm{above}}
		\subset \MF_{\qcoh}^{\C}([Y/G], w)
	\end{align*}
	to be consisting of objects whose $\lambda$-weights are bounded 
	above, and each $\lambda$-weight $j$ part is an object in 
	$\MF_{\coh}^{\C}([Y/G], w)_{\lambda \mathchar`- \wt= j}$. 
		
	For a closed substack $\zZ \subset \yY$, 
	we define the subcategories 
	\begin{align*}
		\MF_{\star}^{\mathbb{Z}/2}(\yY, w)_{\zZ} \subset \MF_{\star}^{\mathbb{Z}/2}(\yY, w), \ 
		\MF_{\star}^{\C}(\yY, w)_{\zZ} \subset \MF_{\star}^{\C}(\yY, w)
	\end{align*}
	in the similar way as (\ref{MF:supp}). Here we assume that $\zZ$ is $\C$-invariant in the latter case. 
	The dg-categories 
	$\MF^{\ast}_{\star}(\yY, w)_{\rm{dg}}$, 
	$\MF^{\ast}_{\star}(\yY, w)_{\zZ, \rm{dg}}$ are 
	also defined in the similar way as (\ref{MF:dg}). 
	\subsection{Kempf-Ness stratification}\label{subsec:exam:KN}
	Here review Kempf-Ness stratifications associated with GIT quotients of 
	reductive algebraic groups, 
	following the convention of~\cite[Section~2.1]{MR3327537}. 
		Let $G$ be a reductive algebraic group with maximal torus $T$, 
		which acts on a smooth affine scheme $Y$. 
		We denote by 
		$M$ the character lattice of $T$ and $N$ the cocharacter lattice
		of $T$. 
			For a one parameter subgroup $\lambda \colon \C \to G$, 
		let $Y^{\lambda \ge 0}$, $Y^{\lambda=0}$ be defined by 
		\begin{align*}
			Y^{\lambda \ge 0} &\cneq \{y \in Y: 
			\lim_{t\to 0} \lambda(t)(y) \mbox{ exists }\}, \\
			Y^{\lambda = 0} &\cneq \{y \in Y: 
			\lambda(t)(y)=y \mbox{ for all } t \in \C\}. 
		\end{align*}
		The Levi subgroup and the parabolic subgroup
		\begin{align*}
			G^{\lambda=0} \subset G^{\lambda \ge 0} \subset G
			\end{align*}
		are also similarly defined by the conjugate $G$-action on $G$, i.e. 
		$g\cdot (-)=g(-)g^{-1}$. The $G$-action on 
		$Y$ restricts to the $G^{\lambda \ge 0}$-action 
		on $Y^{\lambda \ge 0}$, 
		and the $G^{\lambda=0}$-action on $Y^{\lambda=0}$. 
		We note that
		$\lambda$ factors through 
		$\lambda \colon \C \to G^{\lambda=0}$, and it
		acts on $Y^{\lambda=0}$ trivially. 
	
		For an element 
		$l \in 
		\Pic([Y/G])_{\mathbb{R}}$, 
		we have the open subset of $l$-semistable points
		\begin{align}\notag
			Y^{l\sss} \subset Y
				\end{align}
			characterized 
			by the set of points $y \in Y$ such that 
			for any one parameter subgroup $\lambda \colon \mathbb{C}^{\ast} \to G$
			such that 
			the limit $z=\lim_{t\to 0}\lambda(t)(y)$
			exists in $Y$, we 
			have $\wt(l|_{z})\ge 0$. 
			Let $\lvert \ast \rvert$ is the Weyl-invariant norm 
		on $N_{\mathbb{R}}$. 
		The above subset of $l$-semistable 
		points fits into the \textit{Kempf-Ness (KN) stratification} 
		\begin{align}\label{KN:strata}
			Y=S_{1} \sqcup S_{2} \sqcup \cdots \sqcup S_N \sqcup 
			Y^{l\sss}.  
		\end{align}
		Here for each $1\le i\le N$ there exists a 
		one parameter subgroup $\lambda_{i} \colon \mathbb{C}^{\ast} \to
		T \subset G$, an open and closed subset 
		$Z_{i}$ of
		$(Y \setminus \cup_{i'<i} S_{i'})^{\lambda_{i}=0}$
		such that 
		\begin{align*}
			S_{i}=G \cdot Y_{i}, \ 
			Y_{i}\cneq \{ y \in Y^{\lambda_{i} \ge 0}: 
			\lim_{t \to 0}\lambda_{i}(t)(y) \in Z_{i}\}. 
		\end{align*}
		Moreover by setting the slope to be
		\begin{align}\notag
			\mu_{i} \cneq -\frac{
				\wt(l|_{Z_{i}})}{\lvert \lambda_{i} \rvert} \in \mathbb{R}
		\end{align}
		we have 
		the inequalities
		$\mu_1>\mu_2>\cdots>0$.
		The above stratification gives an example of a 
		$\Theta$-stratification in Definition~\ref{defi:Thetastrata} (see Example~\ref{exam:KN}).

	\subsection{Semiorthogonal decomposition via KN stratification}
	\label{subsec:crit:window}
	In the setting of Subsection~\ref{subsec:exam:KN},  
	let us consider a KN-stratification (\ref{KN:strata})
	with one parameter subgroup 
	$\lambda_{i} \colon \C \to G$ for each $i$. 
	Suppose that there is an auxiliary
	$\C$-action on $Y$ which commutes with the $G$-action, 
	and $w \colon Y \to \C$ is a 
	$G$-invariant function with auxiliary $\C$-weight two. 
	We have the following diagram (see~\cite[Definition~2.2]{MR3327537})
	\begin{align}\label{dia:YZ}
		\xymatrix{
			[Y_{i}/G^{\lambda_{i}\ge 0}] \ar[r]^-{\cong} \ar[d] & [S_{i}/G] \ar[dl]_{p_{i}} \ar@<-0.3ex>@{^{(}->}[r]^-{q_{i}}
			& \left[\left(Y\setminus \cup_{i'<i} S_{i'}\right)/G \right] \ar[d]_-{w} \\
			[Z_{i}/G^{\lambda_{i}=0}] \ar@<-0.3ex>@{^{(}->}[rru]_-{\tau_{i}} \ar[rr]_{w|_{Z_{i}}} & & \mathbb{C}. 
		}
	\end{align}
	Here the left vertical arrow is given by
	taking the $t\to 0$ limit of the action of $\lambda_{i}(t)$
	for $t \in \C$, and $\tau_{i}, q_{i}$ are induced by the 
	embedding $Z_{i} \hookrightarrow Y$, $S_{i} \hookrightarrow Y$ respectively. 
	Let $\eta_{i} \in \mathbb{Z}$ be defined by 
	\begin{align}\notag
		\eta_{i} \cneq \wt_{\lambda_{i}}(\det(N_{S_{i}/Y}^{\vee}|_{Z_{i}})).
	\end{align}
	Note that the auxiliary $\C$-action preserves $S_i$ and $Z_i$
	by its commutativity with the $G$-action. 
	We will use the following version of window theorem. 
	\begin{thm}\label{thm:window}\emph{(\cite{MR3327537, MR3895631})}
		For each $i$, we take $m_{i} \in \mathbb{R}$. 
		
		(i) For each $j \in \mathbb{Z}$, the composition 
		\begin{align*}
			q_{i \ast} p_{i}^{\ast} \colon
			\MF^{\C}_{\coh}([Z_{i}/G^{\lambda_i=0}], w|_{Z_{i}})_{\lambda_{i} \mathchar`- \wt= j}
			\to \MF^{\C}_{\coh}([S_{i}/G], w|_{S_{i}})
			\to \MF^{\C}_{\coh}([(Y \setminus \cup_{i'<i}S_{i'})/G], w)
		\end{align*}
		is fully-faithful, whose essential image is denoted by 
		$\dD_{i, j}$. 
		
		(ii) 
		There exist semiorthogonal decomposition 
		\begin{align}\notag
			\MF_{\coh}^{\C}([(Y \setminus \cup_{i'<i}S_{i'})/G], w)
			=\langle \ldots, \dD_{i, \lceil m_{i} \rceil -2}, 
			\dD_{i, \lceil m_{i} \rceil -1},  
			\wW_{i, m_{i}}, \dD_{i, \lceil m_{i} \rceil}, 
			\dD_{i, \lceil m_{i} \rceil +1}, \ldots
			\rangle. 
		\end{align}	
		Here $\wW_{i, m_i}$ consists of 
		factorizations 
		$(\pP, d_{\pP})$
		satisfying that 
		\begin{align}\label{condition:P}
			\tau_{i}^{\ast}(\pP, d_{\pP}) \in 
			\bigoplus_{j \in [m_{i}, m_{i}+\eta_{i})}
			\mathrm{MF}_{\coh}^{\C}([Z_{i}/G^{\lambda_i=0}], w|_{Z_{i}})_{\lambda_{i} \mathchar`- \wt= j}. 
		\end{align}
		
		(iii) The composition functor
		\begin{align*}
			\wW_{i, m_{i}} \hookrightarrow 
			\MF_{\coh}^{\C}([(Y \setminus \cup_{i'<i}S_{i'})/G], w)
			\twoheadrightarrow 
			\MF_{\coh}^{\C}([(Y \setminus \cup_{i' \le i}S_{i'})/G], w)
		\end{align*}
		is an equivalence 
	\end{thm}	
	As a consequence of Theorem~\ref{thm:window}, we have the following 
	window theorem.
	Let
	\begin{align*}
		\wW_{m_{\bullet}}^l([Y/G], w) \subset 
		\MF_{\coh}^{\C}([Y/G], w)
	\end{align*}
	be the subcategory of objects $(\pP, d_{\pP})$
	satisfying the condition (\ref{condition:P}) for all $i$. 
	Then the composition functor 
	\begin{align*}
		\wW_{m_{\bullet}}^l([Y/G], w) \hookrightarrow 
		\MF_{\coh}^{\C}([Y/G], w)
		\twoheadrightarrow
		\MF_{\coh}^{\C}([Y^{l\sss}/G], w)
	\end{align*}
	is an equivalence. 
	Note that 
	the above window subcategory depends on a choice of $m_{\bullet}$. 

	\subsection{Application for KN stratification of critical locus}
	We will apply Theorem~\ref{thm:window} for a KN stratification 
	of $\Crit(w)$
	\begin{align*}
		\Crit(w)= S_1' \sqcup S_2' \sqcup \cdots \sqcup S_N' \sqcup \Crit(w)^{l\sss}
	\end{align*}
	in the following way. 
	After discarding 
	KN strata $S_{i} \subset Y$ 
	with $\Crit(w) \cap S_{i}=\emptyset$, 
	the above stratification is obtained by
	restricting a KN stratification (\ref{KN:strata}) 
	for $Y$ to $\Crit(w)$. 
	Let $\lambda_{i} \colon \C \to G$ be a 
	one parameter subgroup for $S_{i}'$
	with center $Z_{i}' \subset S_{i}'$. 
	We define 
	$\overline{Z}_{i} \subset Y$ to 
	be the union of connected components 
	of the $\lambda_{i}$-fixed part of $Y$
	which contains $Z_{i}'$, 
	and $\overline{Y}_{i} \subset Y$ is the set of 
	points $y \in Y$ with 
	$\lim_{t\to 0}\lambda_{i}(t)y \in \overline{Z}_{i}$. 
	Similarly to (\ref{dia:YZ}), 
	we have the diagram
	\begin{align}\notag
		\xymatrix{
			[\overline{Y}_{i}/G^{\lambda_{i}\ge 0}] \ar@/^15pt/[rr]^-{\overline{q}_{i}} 
			\ar[d]_-{\overline{p}_{i}} \ar@<-0.3ex>@{^{(}->}[r] & 
			[Y^{\lambda_{i}\ge 0}/G^{\lambda_{i}\ge 0}] \ar[d] 
			\ar[r] & [Y/G] \\
			[\overline{Z}_{i}/G^{\lambda_{i}= 0}] \ar@<-0.3ex>@{^{(}->}[r] & [Y^{\lambda_{i}=0}/G^{\lambda_{i}=0}]. \ar[ru] & 	
		}
	\end{align}
	Here the left horizontal arrows are 
	open and closed immersions. 
	By noting the equivalence (\ref{equiv:supp3})
	together with 
		$\Crit(w) \cap \overline{Z}_{i}=
	\Crit(w|_{\overline{Z}_{i}})$
	and $Z_i=\overline{Z}_i\setminus \cup_{i'<i} S_{i'}$, 
	the result of 
	Theorem~\ref{thm:window}
	implies the semiorthogonal decomposition
	\begin{align}\notag
		\MF_{\coh}^{\C}([(Y \setminus \cup_{i'<i}S'_{i'})/G], w)
		=\langle \ldots, \dD_{i, \lceil m_{i} \rceil-2}, \dD_{i, \lceil m_{i} \rceil-1},  
		\wW_{i, m_{i}}, \dD_{i, \lceil m_{i} \rceil}, \dD_{i, \lceil m_{i} \rceil +1}, \ldots
		\rangle
	\end{align}	
	with equivalences
	\begin{align}\label{equiv:bar}
		\overline{q}_{i \ast} \overline{p}_{i}^{\ast} \colon
		\MF^{\C}_{\coh}([(\overline{Z}_{i} \setminus \cup_{i'<i}S'_{i'})/G^{\lambda_i=0}], w|_{\overline{Z}_{i}})_{\lambda_{i} \mathchar`- \wt= j}
		\stackrel{\sim}{\to}\dD_{i, j}. 
	\end{align}
	The subcategory $\wW_{i, m_{i}}$ 
	consists of factorizations $(\pP, d_{\pP})$ such that 
	\begin{align*}
	(\pP, d_{\pP})|_{[(\overline{Z}_{i} \setminus 
			\cup_{i'<i}S_{i'}')/G^{\lambda_i=0}]} \in 
		\bigoplus_{j \in [m_{i}, m_{i}+\overline{\eta}_{i})}
		\mathrm{MF}_{\coh}^{\C}([(\overline{Z}_{i} \setminus 
		\cup_{i'<i}S_{i'}')/G^{\lambda_i=0}], w|_{\overline{Z}_{i}})_{\lambda_{i} \mathchar`- \wt= j}
	\end{align*}
	where $\overline{\eta}_{i}=\wt_{\lambda_{i}}\det (\mathbb{L}_{\overline{q}_{i}})^{\vee}|_{\overline{Z}_{i}}$, 
	and the composition functor
	\begin{align*}
		\wW_{i, m_{i}} \hookrightarrow 
		\MF_{\coh}^{\C}([(Y \setminus \cup_{i'<i}S'_{i'})/G], w)
		\twoheadrightarrow 
		\MF_{\coh}^{\C}([(Y \setminus \cup_{i' \le i}S'_{i'})/G], w)
	\end{align*}
	is an equivalence.  
	
	\section{Categorical DT theory for $(-1)$-shifted cotangent derived stacks: review}
	The $\C$-equivariant DT category for $(-1)$-shifted cotangent 
	derived stacks is defined in~\cite{TocatDT} based on 
	Koszul duality equivalence between derived categories of coherent 
	sheaves on quasi-smooth affine derived schemes and 
	triangulated categories of 
	$\C$-equivariant factorizations. 
	The $\mathbb{Z}/2$-periodic version of Koszul duality equivalence 
	is also proved in~\cite{TocatDT2}, 
	which is used to give $\mathbb{Z}/2$-periodic 
	version of DT categories. 
	In this section, we review Koszul duality, 
	its functorial properties, and 
	the constructions of DT categories.

	\subsection{Koszul duality equivalence}
	Let $Y$ be a smooth affine scheme with a section $s$ of a vector bundle 
	$V \to Y$. 
	We denote by $\fU$ the derived zero locus of $s$
	\begin{align}\label{frak:U}
		\fU=\Spec \rR(V \to Y, s), 
		\end{align}
	where $\rR(V \to Y, s)$ is the Koszul complex
	\begin{align*}
		\rR(V \to Y, s) \cneq 
	\left(  
	\cdots \to \bigwedge^2 V^{\vee} \stackrel{s}{\to} V^{\vee} \stackrel{s}{\to}
	\oO_Y \to 0 \right).
	\end{align*}
	Let $V^{\vee} \to Y$ be the total space of the dual vector bundle of $V$. 
There is an associated function on $V^{\vee}$, given by 
\begin{align}\label{func:w}
	w \colon V^{\vee} \to \mathbb{C}, \ 
	w(x, v)=\langle s(x), v \rangle, \ 
	x \in Y, \ v \in V^{\vee}|_{x}. 
\end{align}
It is well-known that the critical locus of the above 
function is the classical truncation of the $(-1)$-shifted cotangent 
over $\fU$ (see~\cite{MR3607000}), 
\begin{align}\label{Omega:U}
	t_0(\Omega_{\fU}[-1])=\Crit(w) \subset V^{\vee}. 
\end{align}
Let $\C$ acts on the fibers of $V^{\vee} \to Y$ by weight two, so that $w$ is 
of weight two. 
Let $G$ be an affine algebraic group which acts on $Y$ 
such that $(V, s)$ is $G$-equivariant. 
We have the following Koszul duality equivalence 
which relates derived category of coherent sheaves on $[\fU/G]$ with
the triangulated category of 
$\C$-equivariant factorizations of $w$ on $[V^{\vee}/G]$
(here we refer to Subsection~\ref{subsec:qsmooth} for $\Dbc(-)$ for derived stacks):

\begin{thm}\emph{(cf.~\cite{MR3071664, MR2982435, MR3631231}, \cite[Theorem~2.3.3, Lemma~2.3.10]{TocatDT})}
	\label{thm:knoer}
	There is an equivalence of triangulated categories
	\begin{align}\label{equiv:Psi}
		\Phi \colon \Dbc([\fU/G])
		 \stackrel{\sim}{\to} \MF_{\coh}^{\C}([V^{\vee}/G], w),
	\end{align}
	which extends to the equivalence 
	\begin{align}\notag
		\Phi^{\ind} \colon \Ind \Dbc([\fU/G]) \stackrel{\sim}{\to} \MF_{\qcoh}^{\C}([V^{\vee}/G], w).
	\end{align}
\end{thm}
The equivalence (\ref{equiv:Psi}) is constructed in the following way. 
Let $\kK_s$ be the following $\C$-equivariant
factorization of $w$
\begin{align}\notag
	\kK_s \cneq 
	\left(
	\oO_{V^{\vee}} \otimes_{\oO_Y} \oO_{\mathfrak{U}}, 
	d_{\kK_s} \right). 
\end{align}
Here 
the $\C$-action is given by the grading  
\begin{align*}
	\oO_{V^{\vee}} \otimes_{\oO_Y}\oO_{\mathfrak{U}}=
	S_{\oO_Y}(V[-2]) \otimes_{\oO_Y} S_{\oO_Y}(V^{\vee}[1]),
\end{align*}
and the weight one map $d_{\kK_s}$ is given by 
\begin{align}\notag
	d_{\kK_s}=
	1 \otimes d_{\oO_{\mathfrak{U}}}+\eta \colon \oO_{V^{\vee}} \otimes_{\oO_Y}\oO_{\mathfrak{U}} \to
	\oO_{V^{\vee}} \otimes_{\oO_Y}\oO_{\mathfrak{U}} \langle 1 \rangle,
\end{align}
where 
$\eta \in V\otimes_{\oO_Y} V^{\vee} \subset \oO_{V^{\vee}} \otimes_{\oO_Y}
\oO_{\fU}$
corresponds to 
$\id \in \Hom(V, V)$, and $\langle 1 \rangle$ indicates the shift of $\C$-weight by one. 
We also equip the diagonal $G$-equivariant structure on $\kK_s$. 
Then $\Phi$ is given by 
\begin{align}\notag
	\Phi(-)=\kK_s \otimes_{\oO_{\fU}} (-), \ 
	\Dbc([\fU/G]) \to \MF_{\coh}^{\C}([V^{\vee}/G], w). 
\end{align}

\subsection{Singular supports of (ind) coherent sheaves}
The theory of singular supports of 
coherent sheaves on $\mathfrak{U}$
is developed in~\cite{MR3300415}
following the earlier work~\cite{MR2489634}. 
Here we recall its definition. 
Let $\mathrm{HH}^{\ast}(\mathfrak{U})$ be the 
Hochschild cohomology
\begin{align*}
	\mathrm{HH}^{\ast}(\mathfrak{U})
	\cneq \Hom_{\mathfrak{U} \times \mathfrak{U}}^{\ast}
	(\Delta_{\ast}\oO_{\mathfrak{U}}, \Delta_{\ast}\oO_{\mathfrak{U}}).
\end{align*}
Here $\Delta \colon \mathfrak{U} \to \mathfrak{U} \times \mathfrak{U}$ is the diagonal. 
Then it is shown in~\cite[Section~4]{MR3300415}
that there 
exists a
canonical map 
$\hH^1(\mathbb{T}_{\mathfrak{U}}) \to \mathrm{HH}^2(\mathfrak{U})$, 
so the map of 
graded rings
\begin{align}\notag
	\oO_{\Crit(w)}=
	S_{\oO_Y}(\hH^1(\mathbb{T}_{\mathfrak{U}})) \to \mathrm{HH}^{2\ast}(\mathfrak{U}) 
	\to \mathrm{Nat}_{D^b_{\rm{coh}}(\mathfrak{U})}(\id, \id[2\ast]). 
\end{align}
Here $\mathrm{Nat}_{D^b_{\rm{coh}}(\mathfrak{U})}(\id, \id[2\ast])$
is the group of natural transformations from $\id$ to $\id[2\ast]$
on $D^b_{\rm{coh}}(\mathfrak{U})$, and the right arrow is defined 
by taking Fourier-Mukai transforms associated with 
morphisms $\Delta_{\ast}\oO_{\mathfrak{U}} \to \Delta_{\ast}\oO_{\mathfrak{U}}[2\ast]$. 
The above maps induce the map for 
each $\fF \in D^b_{\rm{coh}}(\mathfrak{U})$, 
\begin{align}\notag
	\oO_{\mathrm{Crit}(w)}
	\to \Hom^{2\ast}(\fF, \fF).
\end{align}
The above map 
defines the 
$\mathbb{C}^{\ast}$-equivariant 
$\oO_{\mathrm{Crit}(w)}$-module 
structure on $\Hom^{2\ast}(\fF, \fF)$, 
which is finitely generated  by~\cite[Theorem~4.1.8]{MR3300415}. 
Below a closed subset $Z \subset \mathrm{Crit}(w)$ is called
\textit{conical} if it is invariant under the fiberwise 
$\mathbb{C}^{\ast}$-action on $\mathrm{Crit}(w)$. 
For $\fF \in D^b_{\rm{coh}}(\mathfrak{U})$, its 
\textit{singular
support} is a conical closed subset
\begin{align*}
	\mathrm{Supp}^{\rm{sg}}(\fF) \subset \mathrm{Crit}(w)
\end{align*}
defined to be the support of $\Hom^{2\ast}(\fF, \fF)$
as $\oO_{\mathrm{Crit}(w)}$-module. 

For a conical closed subset $Z \subset \Crit(w)$
invariant under the $G$-action, 
let $\zZ=[Z/G]$ and set
\begin{align*}
	\cC_{\zZ} \subset \Dbc([\fU/G]), \ 
	\Ind \cC_{\zZ} \subset \Ind \Dbc([\fU/G])
\end{align*}
be the triangulated subcategories consisting of objects
whose singular supports
are contained in $\zZ$, and its ind-completion respectively
(see (\ref{ssupp:stack}) for singular supports for derived stacks). 

\begin{prop}\label{prop:ssupp}\emph{(\cite[Proposition~2.3.9]{TocatDT})}
	The equivalences in Theorem~\ref{thm:knoer} 
	restrict to the equivalences
	\begin{align}\notag
		\Phi \colon \cC_{\zZ} \stackrel{\sim}{\to}
		 \MF^{\C}_{\coh}([V^{\vee}/G], w)_{\zZ}, \ 
		\Phi^{\ind} \colon \Ind \cC_{\zZ} \stackrel{\sim}{\to}
		\MF_{\qcoh}^{\C}([V^{\vee}/G], w)_{\zZ}. 
	\end{align}
	In particular by (\ref{equiv:supporZ}), the equivalences in Theorem~\ref{thm:knoer}
	descend to the equivalences 
	\begin{align}\label{Psi:supp2}
	&	\Phi \colon \Dbc([\fU/G])/\cC_{\zZ} \stackrel{\sim}{\to}
		\MF^{\C}_{\coh}([V^{\vee}/G] \setminus \zZ, w), \\ 
	&	\notag	\Phi^{\ind} \colon 	\Ind\Dbc(\fU)/\Ind\cC_{\zZ} \stackrel{\sim}{\to}
			\MF^{\C}_{\qcoh}([V^{\vee}/G] \setminus \zZ, w). 
	\end{align}
\end{prop}

\subsection{Some functorial properties of Koszul duality equivalence}\label{subsec:funct}
We review some functorial properties of Koszul duality equivalence in Theorem~\ref{thm:knoer}, 
based on~\cite[Section~2.4]{TocatDT}. 
For $i=1, 2$, let $Y_i$ be smooth affine schemes, 
$V_i \to Y_i$ be vector bundles with sections $s_i$. 
Suppose that affine algebraic groups $G_i$ act on $Y_i$
such that $(V_i, s_i)$ are $G_i$-equivariant. 
Let us consider a commutative diagram 
	\begin{align}\label{dia:Y}
		\xymatrix{
			V_1 \ar[r]^-{g} \ar[d] & V_2 \ar[d] \\
			Y_1 \ar[r]^-{f} \ar@/^10pt/[u]^-{s_1}
			& Y_2, \ar@/_10pt/[u]_-{s_2}&
		} 
	\end{align}
which is $G_i$-equivariant 
with respect to a group homomorphism $\phi \colon G_1 \to G_2$. 
The top morphism $g$ is a composition 
		$V_1 \stackrel{g'}{\to} f^{\ast}V_2 \stackrel{f}{\to}V_2$, 
		where $g'$ is a morphism of $G_1$-equivariant vector bundles on 
		$Y_1$. 
The diagram (\ref{dia:Y}) induces the diagram 
for smooth stacks $\yY_i=[Y_i/G_i], \ \vV_i=[V_i/G_i]$, 
\begin{align}\label{dia:Ystack}
	\xymatrix{
		\vV_1 \ar[r]^-{g} \ar[d] & \vV_2 \ar[d] \\
		\yY_1 \ar[r]^-{f} \ar@/^10pt/[u]^-{s_1}
		& \yY_2,  \ar@/_10pt/[u]_-{s_2}&
	}
\end{align}
which also induces the 
morphism of derived stacks 
\begin{align*}
	\mathbf{f} \colon 
	[\fU_1/G_1] \to [\fU_2/G_2], \ 
	\fU_i \cneq \Spec \rR(V_i \to Y_i, s_i). 
\end{align*}
Then we have the push-forward functor
(see~\cite[Section~3.6]{MR3037900})
\begin{align*}
	\mathbf{f}_{\ast}^{\rm{ind}} \colon 
	\Ind \Dbc([\fU_1/G_1]) \to \Ind \Dbc([\fU_2/G_2]). 
\end{align*}
Let $w_i \colon V_i^{\vee} \to \mathbb{C}$ be given as in (\ref{func:w})
defined from $(Y_i, V_i, s_i)$. 
The diagram (\ref{dia:Y}) induces the 
following diagram
\begin{align}\label{diagram:dual}
	\xymatrix{
		&  \mathbb{C} & \\
		V_1^{\vee} \ar[ur]^-{w_1} \ar[d]_-{p_1} & f^{\ast}V^{\vee}_2 \ar[l]^-{g} 
		\ar[r]_-{f} \ar@{}[dr]|\square
		\ar[d]_-{\overline{p}}
		\ar[u]^{\overline{w}}
		& V_2^{\vee} \ar[d]^-{p_2} \ar[ul]_-{w_2} \\
		Y_1 & Y_1 \ar@{=}[l] \ar[r]_-{f} & Y_2. 
	}
\end{align}
Here $\overline{w}$ is determined by 
\begin{align*}
	\overline{w}=f^{\ast}s_2 \in \Gamma(Y_1, f^{\ast}V_2) \subset
	\Gamma(Y_1, S(f^{\ast}V_2)). 
\end{align*}
The commutativity of (\ref{dia:Y}) implies that 
the diagram (\ref{diagram:dual}) is also commutative. 
One can check the following identity in 
$f^{\ast}V_2^{\vee}$ (see~\cite[Lemma~2.4.2]{TocatDT})
	\begin{align}\label{identity:crit}
		g^{-1}(\mathrm{Crit}(w_1)) \cap 
		f^{-1}(\mathrm{Crit}(w_2))= t_0(\Omega_{\mathfrak{U}_2}[-1]\times_{\fU_2}\fU_1). 
	\end{align}

We
define the following functor
\begin{align}\label{funct:fg}
	f_{\ast} \circ g^{\ast} \colon 
	\mathrm{MF}_{\qcoh}^{\C}([V_1^{\vee}/G_1], w_1)
	\stackrel{g^{\ast}}{\to}
	\mathrm{MF}_{\qcoh}^{\C}([f^{\ast}V_2^{\vee}/G_1], \overline{w})
	\stackrel{f_{\ast}}{\to} 	\mathrm{MF}_{\qcoh}^{\C}([V_2^{\vee}/G_2], w_2). 
\end{align}
Then the following diagram commutes (see~\cite[Lemma~2.2.4]{TocatDT}):
	\begin{align}\label{dia:com:ind}
		\xymatrix{
			\Ind \Dbc([\fU_1/G_1]) \ar[r]^-{\Phi_1^{\ind}} \ar[d]_-{\mathbf{f}^{\rm{ind}}_{\ast}} & 
			\mathrm{MF}_{\qcoh}^{\C}([V_1^{\vee}/G_1], w_1) \ar[d]^-{f_{\ast}\circ g^{\ast}}  \\
			\Ind \Dbc([\fU_2/G_2]) \ar[r]^-{\Phi_2^{\rm{ind}}}  & 
			\mathrm{MF}_{\qcoh}^{\C}([V_2^{\vee}/G_2], w_2). 	
		}
	\end{align}
	Here the horizontal arrows are equivalences in Theorem~\ref{thm:knoer}. 
	
	Suppose that the morphism $f \colon \yY_1 \to \yY_2$ in the diagram (\ref{dia:Ystack}) is a proper 
	morphism of smooth stacks, so 
	in particular $\bff \colon [\fU_1/G_1] \to [\fU_2/G_2]$ is proper. 
Then the diagram (\ref{dia:com:ind}) restricts to the commutative diagram 
	\begin{align}\label{dia:com:ind0.5}
		\xymatrix{
			\Dbc([\fU_1/G_1]) \ar[r]^-{\Phi_1} 
			\ar[d]_-{\mathbf{f}_{\ast}} & 
			\mathrm{MF}_{\coh}^{\C}([V_1^{\vee}/G_1], w_1) \ar[d]^-{f_{\ast}\circ g^{\ast}}  \\
			\Dbc([\fU_2/G_2]) \ar[r]^-{\Phi_2}  & 
			\mathrm{MF}_{\coh}^{\C}([V_2^{\vee}/G_2], w_2). 	
		}
	\end{align}
Moreover we have the continuous right adjoint 
of $\bff_{\ast}^{\rm{ind}}$ (see~\cite[Section~10.1]{MR3136100})
\begin{align*}
	\bff^{!} \colon \Ind \Dbc([\fU_2/G_2]) \to \Ind \Dbc([\fU_1/G_1]). 
\end{align*}
On the other hand, we also have the functor 
\begin{align}\notag
	g_{\ast} \circ f^{!} \colon 
	\mathrm{MF}_{\qcoh}^{\C}([V_2^{\vee}/G_2], w_2)
	\stackrel{f^{!}}{\to}
	\mathrm{MF}_{\qcoh}^{\C}([f^{\ast}V_2^{\vee}/G_1], \overline{w})
	\stackrel{g_{\ast}}{\to} 	\mathrm{MF}_{\qcoh}^{\C}([V_1^{\vee}/G_1], w_1)
\end{align}
which is a right adjoint of 
the functor (\ref{funct:fg}). 
	Then the following diagram commutes (see~\cite[Lemma~2.4.6]{TocatDT}):
	\begin{align}\label{dia:com:ind2}
		\xymatrix{
			\Ind \Dbc([\fU_2/G_2]) \ar[r]^-{\Phi_2^{\rm{ind}}} \ar[d]_-{\mathbf{f}^!} & 
			\mathrm{MF}_{\qcoh}^{\C}([V_2^{\vee}/G_2], w_2) \ar[d]^-{g_{\ast} \circ f^!}  \\
			\Ind \Dbc([\fU_1/G_1]) \ar[r]^-{\Phi_1^{\rm{ind}}}  & 
			\mathrm{MF}_{\qcoh}^{\C}([V_1^{\vee}/G_1], w_1). 	
		}
	\end{align}

Suppose that $\mathbf{f} \colon [\fU_1/G_1] \to [\fU_2/G_2]$ is 
not necessary proper, but it is 
quasi-smooth. 
Then we have the functor (see~\cite[Section~3.1]{MR3701352})
\begin{align*}
	\bff^{\ast} \colon \Dbc([\fU_2/G_2]) \to \Dbc([\fU_1/G_1]). 
\end{align*}
If furthermore $g \colon f^{\ast}V_2^{\vee} \to V_1^{\vee}$ in 
the diagram (\ref{diagram:dual}) is proper, 
then we have the functor 
\begin{align*}
	g_{!} \circ f^{\ast} \colon 
	\mathrm{MF}_{\coh}^{\C}([V_2^{\vee}/G_2], w_2)
	\stackrel{f^{\ast}}{\to}
	\mathrm{MF}_{\coh}^{\C}([f^{\ast}V_2^{\vee}/G_1], \overline{w})
	\stackrel{g_{!}}{\to} 	\mathrm{MF}_{\coh}^{\C}([V_1^{\vee}/G_1], w_1)
\end{align*}
which is a left adjoint of the functor (\ref{funct:fg}). 
	Then the following diagram commutes (see~\cite[Lemma~2.4.7]{TocatDT}):
	\begin{align}\label{dia:com:ind1.5}
		\xymatrix{
			\Dbc([\fU_2/G_2]) \ar[r]^-{\Phi_2} \ar[d]_-{\mathbf{f}^{\ast}} &
			\mathrm{MF}_{\coh}^{\C}([V_2^{\vee}/G_2], w_2) \ar[d]^-{g_{!} \circ f^{\ast}}  \\
			\Dbc([\fU_1/G_1]) \ar[r]^-{\Phi_1}  & 
			\mathrm{MF}_{\coh}^{\C}([V_1^{\vee}/G_1], w_1). 	
		}
	\end{align}

	\subsection{Quasi-smooth derived stacks}\label{subsec:qsmooth}
Below, we denote by $\mathfrak{M}$ a
derived Artin stack over $\mathbb{C}$.
This means that 
$\mathfrak{M}$ is a contravariant
$\infty$-functor from 
the $\infty$-category of 
affine derived schemes over $\mathbb{C}$ to 
the $\infty$-category of 
simplicial sets 
\begin{align*}
	\mathfrak{M} \colon 
	dAff^{op} \to SSets
\end{align*}
satisfying some conditions (see~\cite[Section~3.2]{MR3285853} for details). 
Here $dAff^{\rm{op}}$ is defined to be the
$\infty$-category of 
commutative simplicial $\mathbb{C}$-algebras, 
which is equivalent to the $\infty$-category of 
commutative differential graded 
$\mathbb{C}$-algebras with non-positive degrees. 
The classical truncation of $\mathfrak{M}$ is denoted by 
\begin{align*}
	\mM \cneq t_0(\mathfrak{M}) \colon 
	Aff^{op} \hookrightarrow 
	dAff^{op} \to SSets
\end{align*}
where the first arrow is a natural functor 
from the category of affine schemes
to affine derived schemes. 

For an affine derived scheme $\mathfrak{U}=\Spec R$
for a cdga $R$, 
we set $D_{\qcoh}(\fU)_{\rm{dg}}$ to be the 
dg-category of dg $R$-modules 
localized by quasi-isomorphisms. 
The
dg-category of quasi-coherent sheaves on $\mathfrak{M}$
is defined
to be the limit in the $\infty$-category of dg-categories (see~\cite[Section~4.1]{MR3285853})
\begin{align}\notag
	D_{\rm{qcoh}}(\mathfrak{M})_{\rm{dg}} \cneq
	\lim_{\mathfrak{U} \to \mathfrak{M}} D_{\rm{qcoh}}(\mathfrak{U})_{\rm{dg}}.
\end{align}
Here the limit is taken for the 
$\infty$-category of smooth morphisms 
$\alpha \colon \fU \to \fM$ for affine 
derived schemes $\fU$ with 
1-morphisms given by 
smooth morphisms $f \colon \fU \to \fU'$ commuting with maps to $\fM$, 
and pull-back $f^{\ast} \colon D_{\qcoh}(\fU')_{\rm{dg}}
\to D_{\qcoh}(\fU)_{\rm{dg}}$ is assigned for each $f$. 
The homotopy category of $D_{\rm{qcoh}}(\mathfrak{M})_{\rm{dg}}$
is denoted by 
$D_{\rm{qcoh}}(\mathfrak{M})$, which is a 
triangulated category. We have the dg and
triangulated subcategories
\begin{align*}
	D^b_{\rm{coh}}(\mathfrak{M})_{\rm{dg}} \subset 
	D_{\rm{qcoh}}(\mathfrak{M})_{\rm{dg}}, \ 
	D^b_{\rm{coh}}(\mathfrak{M}) \subset D_{\rm{qcoh}}(\mathfrak{M})
\end{align*}
consisting of objects which have bounded coherent 
cohomologies.

A morphism of derived stacks $f \colon \fM \to \fN$ is 
called \textit{quasi-smooth} if 
$\mathbb{L}_f$ is perfect 
such that for any point $x \to \mM$
the restriction $\mathbb{L}_f|_{x}$ is 
of cohomological 
amplitude $[-1, 1]$.
Here 
$\mathbb{L}_f$ is the $f$-relative cotangent complex. 
A derived stack 
$\mathfrak{M}$ over $\mathbb{C}$
is called \textit{quasi-smooth}
if $\fM \to \Spec \mathbb{C}$ is quasi-smooth. 
By~\cite[Theorem~2.8]{MR3352237}, 
the quasi-smoothness of $\mathfrak{M}$ is equivalent to 
that $\fM$ 
is a 1-stack, 
and 
any point of $\mathfrak{M}$ lies 
in the image of a $0$-representable 
smooth morphism 
$\alpha \colon \mathfrak{U} \to \mathfrak{M}$, 
where $\mathfrak{U}$ is an affine derived scheme 
obtained as a derived zero locus as 
in (\ref{frak:U}). 
In this case, we have 
\begin{align*}
	\Dbc(\fM)_{\rm{dg}}=\lim_{\mathfrak{U} \stackrel{\alpha}{\to} \mathfrak{M}} 
	\Dbc(\fU)_{\rm{dg}}
\end{align*}
where the limit is taken for 
the $\infty$-category $\iI$ of smooth morphisms
$\alpha \colon \fU \to \fM$
where $\fU$ is equivalent to an 
affine derived scheme of the form (\ref{frak:U}). 
In this paper when we write 
$\lim_{\mathfrak{U} \stackrel{\alpha}{\to} \mathfrak{M}}(-)$
for a quasi-smooth $\fM$, 
the limit is always taken for the $\infty$-category $\iI$
as above. 
For a quasi-smooth derived stack $\fM$, 
we denote by $\Ind \Dbc(\fM)_{\rm{dg}}$ the dg-category of its ind-coherent 
sheaves (see~\cite[Section~10]{MR3136100}, and also~\cite[Section~11.3, Proposition~11.4.3]{MR3136100}
for its equivalence with $\ast$-pull back version)
\begin{align}\notag
	\Ind \Dbc(\fM)_{\rm{dg}} \cneq \lim_{\fU \stackrel{\alpha}{\to}\fM}
	\Ind \Dbc(\fU)_{\rm{dg}},
\end{align}
where the limit is taken for the $\infty$-category $\iI$ 
as above. Its homotopy category is denoted by 
$\Ind \Dbc(\fM)$. 

Following~\cite[Definition~1.1.8]{MR3037900}, 
a derived stack $\mathfrak{M}$ is called 
\textit{QCA (quasi-compact and with affine automorphism groups)}
if the following conditions hold:
\begin{enumerate}
	\item $\mathfrak{M}$ is quasi-compact;
	\item The automorphism groups of its geometric points are affine;
	\item The classical inertia stack $I_{\mM} \cneq \Delta
	\times_{\mM \times \mM} \Delta$
	is of finite presentation over $\mM$. 
\end{enumerate}
The QCA condition will be useful 
since in this case
	$\Ind \Dbc(\fM)_{\rm{dg}}$ is compactly generated 
	with 
	compact objects 
	$\Dbc(\fM)_{\rm{dg}}$
	(see~\cite[Theorem~3.3.5]{MR3037900}). 

\subsection{Good moduli spaces for Artin stacks}\label{subsec:gmoduli}
	In general for a classical Artin stack $\mM$, 
its \textit{good moduli space}
is an algebraic space $M$ 
together with a quasi-compact morphism, 
\begin{align*}
	\pi_{\mM} \colon \mM \to M
\end{align*} 
satisfying the following conditions
(cf.~\cite[Section~1.2]{MR3237451}):
\begin{enumerate}
	\item The push-forward
	$\pi_{\mM\ast} \colon \mathrm{QCoh}(\mM) \to \mathrm{QCoh}(M)$
	is exact. 
	\item The induced morphism $\oO_M \to \pi_{\mM\ast}\oO_{\mM}$ is an 
	isomorphism. 
\end{enumerate}
The good moduli space morphism $\pi_{\mM}$ is universally closed.
Moreover 
for each closed point $y \in M$, 
there exists a unique closed point $x \in \pi_{\mM}^{-1}(y)$, 
and its automorphism group $\Aut(x)$ is reductive 
(see~\cite[Theorem~4.16, Proposition~12.14]{MR3237451}). 

Let $\fM$ be a quasi-smooth 
derived stack such that $\mM=t_0(\fM)$
admits a good moduli space $\pi_{\mM} \colon \mM \to M$
as above. 
For a closed point $y \in M$, 
let $\widehat{\mM}_y$ be the 
\textit{formal fiber} defined by 
\begin{align*}
	\widehat{\mM}_y \cneq \mM \times_M \Spec \widehat{\oO}_{M, y}. 
	\end{align*}
By a standard deformation theory argument, 
there exists a derived stack $\widehat{\fM}_y$
(unique up to equivalence) 
together with the following 
Cartesian diagrams (see~\cite[Lemma~5.2.5]{TocatDT})
\begin{align}\label{dia:fnbd}
	\xymatrix{
		\widehat{\fM}_y \ar[d]_-{\widehat{\iota}_y} \diasquare
		& \ar@<-0.3ex>@{_{(}->}[l] \widehat{\mM}_y
		\ar[r] \ar[d] \diasquare & \Spec
		\widehat{\oO}_{M, y} \ar[d] \\
		\fM  & \ar@<-0.3ex>@{_{(}->}[l] \mM \ar[r]^-{\pi_{\mM}} & M. 
	}
\end{align}
Below we use the same symbol $y \in \mM$ to 
denote the unique closed orbit in the fiber of 
$\pi_{\mM} \colon \mM \to M$ at $y$. 
Let $G_y \cneq \Aut(y)$, which is a reductive algebraic group. 
We denote by $\widehat{\hH}^0(\mathbb{T}_{\fM}|_{y})$ the formal 
fiber along 
$\hH^0(\mathbb{T}_{\fM}|_{y}) \to \hH^0(\mathbb{T}_{\fM}|_{y})\ssslash G_y$
at the origin, i.e. 
\begin{align*}
	\widehat{\hH}^0(\mathbb{T}_{\fM}|_{y}) \cneq 
	\hH^0(\mathbb{T}_{\fM}|_{y}) \times_{\hH^0(\mathbb{T}_{\fM}|_{y})\ssslash G_y} 
	\Spec \widehat{\oO}_{\hH^0(\mathbb{T}_{\fM}|_{y})\ssslash G_y, 0}. 
\end{align*}
\begin{defi}\emph{(\cite[Definition~5.2.3, Lemma~5.2.5]{TocatDT})}\label{defi:formalneigh}
	The derived stack $\fM$ with a good moduli space $\pi_{\mM} \colon \mM \to M$
	satisfies the formal neighborhood theorem if 
	for any closed point $y \in M$, 
	there is a $G_y$-equivariant formal morphism 
	(called Kuranishi map)
	\begin{align}\notag
		\kappa_y \colon \widehat{\hH}^0(\mathbb{T}_{\fM}|_{y})
		\to \hH^1(\mathbb{T}_{\fM}|_{y})
	\end{align}
	such that $\kappa_y(0)=0$, and by setting $\widehat{\fU}_{y}$
	to be the derived zero locus of $\kappa_y$, 
	we have an equivalence 
	\begin{align*}
		[\widehat{\fU}_y/G_y]
		\sim \widehat{\fM}_y
		\end{align*}
	which sends $0$ to $y$ and identities on stabilizer 
	groups. 
\end{defi}

\begin{rmk}\label{rmk:formal}
	It is proved in~\cite[Section~7.4]{TocatDT} that 
	if $\fM$ is a quasi-smooth derived stack 
	obtained as moduli spaces of 
	stable sheaves or stable pairs, then 
	it satisfies formal neighborhood theorem. 
	In this case, a Kuranishi map is 
	obtained from an $A_{\infty}$-structure of the 
	derived category of coherent sheaves on surfaces. 	
	\end{rmk}

\subsection{$(-1)$-shifted cotangent derived stacks}
Let $\mathfrak{M}$ be a quasi-smooth derived stack. 
We denote by 
$\Omega_{\mathfrak{M}}[-1]$ the $(-1)$-shifted
cotangent derived stack of $\mathfrak{M}$
\begin{align*}
	p \colon 
	\Omega_{\mathfrak{M}}[-1] \cneq 
	\Spec_{\mathfrak{M}}
	S(\mathbb{T}_{\mathfrak{M}}[1]) \to \mathfrak{M}. 
\end{align*}
Here $\mathbb{T}_{\mathfrak{M}} \in D^b_{\rm{coh}}(\mathfrak{M})$
is the tangent complex of $\mathfrak{M}$, which 
is dual to the cotangent complex 
$\mathbb{L}_{\mathfrak{M}}$ of $\mathfrak{M}$. 
The derived stack $\Omega_{\mathfrak{M}}[-1]$ 
admits a natural $(-1)$-shifted symplectic 
structure~\cite{MR3090262, Calaque}, 
which induces a
d-critical structure~\cite{MR3399099}
on its 
classical truncation $\nN$ 
\begin{align}\notag
	p_0 \colon \nN \cneq t_0(\Omega_{\mathfrak{M}}[-1])
	\to \mM.
\end{align}

Let $\fM_1$, $\fM_2$ be
quasi-smooth derived stacks 
with truncations $\mM_i=t_0(\fM_i)$. 
Let $f \colon \fM_1 \to \fM_2$ be a morphism. 
Then 
the morphism 
$f^{\ast}\mathbb{L}_{\fM_2} \to \mathbb{L}_{\fM_1}$
induces the diagram
\begin{align}\label{diagram:induced}
	\xymatrix{
		t_0(\Omega_{\fM_1}[-1]) \ar[d] & 
		t_0(\Omega_{\fM_2}[-1]\times_{\fM_2}\fM_1)
		\ar[d] \ar[l]_-{f^{\diamondsuit}} 
		\ar[r]^-{f^{\spadesuit}} 
		\ar@{}[rd]|\square
		& 
		t_0(\Omega_{\fM_2}[-1]) \ar[d] \\
		\mM_1 & \mM_1 \ar@{=}[l] \ar[r]_-{f} & \mM_2. 
	}
\end{align}
The morphism $f$ is quasi-smooth if and only
if $f^{\diamondsuit}$ is a closed immersion, 
$f$ is smooth if and only if $f^{\diamondsuit}$ is an isomorphism
(see~\cite[Lemma~3.1.2]{TocatDT}). 

Let us take a  conical closed substack
\begin{align*}
	\zZ \subset \nN=t_0(\Omega_{\mathfrak{M}}[-1]). 
\end{align*}
Here $\zZ$ is called \textit{conical}
if it is invariant
 under the fiberwise $\mathbb{C}^{\ast}$-action on $\nN \to \mM$. 
Let $\alpha \colon \fU \to \fM$ be a smooth morphism
such that $\fU$ is of the form (\ref{frak:U}). 
We have the associated conical closed subscheme
\begin{align}\notag
	\alpha^{\ast}\zZ \cneq 
	\alpha^{\diamondsuit} (\alpha^{\spadesuit})^{-1}(\zZ) \subset t_0(\Omega_{\mathfrak{U}}[-1])
	=\mathrm{Crit}(w).
\end{align}
Here 
$w$ is given as in (\ref{func:w}). 
As in~\cite{MR3300415}, we define 
\begin{align}\label{ssupp:stack}
	\cC_{\zZ, \rm{dg}} \cneq 
	\lim_{\fU \stackrel{\alpha}{\to} \fM}
	\cC_{\alpha^{\ast}\zZ, \rm{dg}}\subset D^b_{\rm{coh}}(\mathfrak{M})_{\rm{dg}}, \ 
	\Ind \cC_{\zZ, \rm{dg}} \cneq \lim_{\fU \stackrel{\alpha}{\to} \fM}
	\Ind \cC_{\alpha^{\ast}\zZ, \rm{dg}}\subset \Ind D^b_{\rm{coh}}(\mathfrak{M})_{\rm{dg}}, 
\end{align}
whose homotopy categories are denoted by $\cC_{\zZ}$, $\Ind \cC_{\zZ}$
respectively.

\subsection{DT categories for $(-1)$-shifted cotangents}
For a quasi-smooth and QCA derived stack $\fM$, 
let us take an open substack $\nN^{\rm{ss}}$ and its complement 
$\zZ$, 
\begin{align*}
	\nN^{\rm{ss}} \subset \nN, \ 
	\zZ \cneq \nN\setminus \nN^{\rm{ss}}. 
\end{align*}
In the case that $\nN^{\rm{ss}}$ is $\C$-invariant so that 
$\zZ$ is a conical closed substack, 
the $\C$-equivariant dg or triangulated DT categories were defined in~\cite{TocatDT}
as Drinfeld or Verdier quotient:
\begin{defi}\emph{(\cite[Definition~3.2.2]{TocatDT})}\label{defi:DTcat}
	We define $\C$-equivariant DT categories 
	for $\nN^{\rm{ss}}$ as 
\begin{align*}
	\dDT^{\C}(\nN^{\rm{ss}})_{\rm{dg}} \cneq 
	\Dbc(\fM)_{\rm{dg}}/\cC_{\zZ, \rm{dg}}, \ 
	\dDT^{\C}(\nN^{\rm{ss}}) \cneq \Dbc(\fM)/\cC_{\zZ}. 
\end{align*}
\end{defi}
The above definition was based on the Koszul 
duality equivalence in Theorem~\ref{thm:knoer}, 
which gives an interpretation of the above categories as
gluing of dg-categories of $\C$-equivariant
factorizations. 

Let $\wW \subset \mM$ be a closed 
substack, and take the open 
derived substack
$\mathfrak{M}_{\circ} \subset \mathfrak{M}$
whose truncation is $\mM \setminus \wW$. 
We have the following conical closed substack
\begin{align*}
	\zZ_{\circ} \cneq \zZ \setminus p_0^{-1}(\wW)
	\subset \nN_{\circ} \cneq t_0(\Omega_{\mathfrak{M}_{\circ}}[-1]).
\end{align*}
Note that $\nN_{\circ}=\nN \setminus p_0^{-1}(\wW)$, and 
we have the following open immersion
\begin{align}\label{open:N}
	\nN_{\circ}^{\rm{ss}} \cneq 
	\nN_{\circ} \setminus \zZ_{\circ} \hookrightarrow
	\nN \setminus \zZ =\nN^{\rm{ss}}. 
\end{align}
The following lemma is useful to replace the quotient 
categories to give equivalent DT categories: 
\begin{lem}\emph{(\cite[Lemma~3.2.9]{TocatDT})}\label{lem:replace0}
	Suppose that $p_0^{-1}(\wW) \subset \zZ$, or equivalently 
	the open immersion (\ref{open:N}) is an isomorphism. 
	Then 
	the restriction functor
	gives an equivalence 
	\begin{align*}
		\mathcal{DT}^{\mathbb{C}^{\ast}}
		(\nN^{\rm{ss}}) \stackrel{\sim}{\to}
		\mathcal{DT}^{\mathbb{C}^{\ast}}
		(\nN_{\circ}^{\rm{ss}}). 
	\end{align*}
\end{lem}
The ind-completions of DT categories are described as follows: 
\begin{prop}\emph{(\cite[Proposition~3.2.7, Theorem~7.2.2]{TocatDT})}\label{prop:DTcat}
	Suppose that $\Ind \cC_{\zZ, \dg}$ is compactly generated. 
	Then we have an equivalence
	\begin{align}\label{equiv:indDT}
		\Ind \Dbc(\fM)_{\rm{dg}}
	/\Ind \cC_{\zZ, \rm{dg}} \stackrel{\sim}{\to}
	\Ind\dDT^{\C}(\nN^{\rm{ss}})_{\rm{dg}}. 
		\end{align}
	\end{prop}

	If $\mM$ admits a good moduli space, we have the 
	following compact generation of the subcategory of 
	fixed singular supports: 
\begin{prop}\emph{(\cite[Theorem~7.2.2]{TocatDT})}\notag
For a quasi-smooth and QCA derived stack $\fM$, 
suppose that $\mM=t_0(\fM)$ admits a good moduli space.
Then $\Ind \cC_{\zZ, \rm{dg}}$ is compactly generated. 
In particular, the equivalence (\ref{equiv:indDT}) holds. 
\end{prop}
	
Let $\fM_{\epsilon} \cneq \fM \times \Spec \mathbb{C}[\epsilon]$
for $\deg(\epsilon)=-1$ 
and set
\begin{align*}
	\zZ_{\epsilon} \cneq \mathbb{C}^{\ast}(\zZ \times \{1\})
	\sqcup (\nN \times \{1\}) \subset 
	\nN \times \mathbb{A}^1=t_0(\Omega_{\fM_{\epsilon}}[-1]). 
	\end{align*}
Here $\C$ acts on fibers of $\nN \times \mathbb{A}^1 \to \mM$
by weight two. 
Then $\zZ_{\epsilon}$ is a conical closed substack. 
The $\mathbb{Z}/2$-periodic DT categories in
Definition~\ref{defi:DTcat} as follows: 
\begin{defi}\emph{(\cite[Definition~4.2]{TocatDT2})}\label{defi:DTcat2}
	We define $\mathbb{Z}/2$-periodic DT categories 
	for $\nN^{\rm{ss}}$ as 
	\begin{align*}
		\dDT^{\mathbb{Z}/2}(\nN^{\rm{ss}})_{\rm{dg}} \cneq 
		\Dbc(\fM_{\epsilon})_{\rm{dg}}/\cC_{\zZ_{\epsilon}, \rm{dg}}, \ 
		\dDT^{\mathbb{Z}/2}(\nN^{\rm{ss}}) \cneq \Dbc(\fM_{\epsilon})/\cC_{\zZ_{\epsilon}}. 
	\end{align*}
\end{defi}
The above definition is also based on the Koszul duality equivalence. 
If $\fM=\fU$ as in (\ref{frak:U}), 
then there is an equivalence (see~\cite[Proposition~3.13]{TocatDT2})
\begin{align*}
	\Dbc(\fU_{\epsilon})/\cC_{\zZ_{\epsilon}} \stackrel{\sim}{\to}
	\MF_{\coh}^{\mathbb{Z}/2}(V^{\vee}\setminus \zZ, w). 
	\end{align*}
Moreover it is proved in~\cite{TocatDT2} that, 
	up to idempotent completion, 
	the $\mathbb{Z}/2$-periodic DT category is
	recovered from the $\C$-equivariant DT category 
	$\dDT^{\C}(\nN^{\rm{ss}})$ defined in~\cite{TocatDT}. 
	\begin{thm}\emph{(\cite[Theorem~4.9]{TocatDT2})}\label{thm:intro:compare}
		Suppose that $\zZ \subset \nN$ is
		a conical closed substack 
		such that $\Ind \cC_{\zZ} \subset \Ind \Dbc(\fM)$
		is compactly generated. 
		Then there is an equivalence 
		\begin{align}\notag
			\overline{\mathcal{DT}}^{\mathbb{Z}/2}(\nN^{\rm{ss}})_{\rm{dg}}
			\simeq \dR \underline{\hH om}(\mathbb{C}[u^{\pm 1}], 
			\Ind \mathcal{DT}^{\C}(\nN^{\rm{ss}})_{\rm{dg}})^{\rm{cp}}. 	
		\end{align}
		Here $\deg(u)=2$, $\mathbb{C}[u^{\pm 1}]$ is regarded as a 
		dg-category with one object, 
		and $\dR \underline{\hH om}(-, -)$
		is an inner Hom of dg-categories (see~\cite[Corollary~6.4]{Todg}). 
	\end{thm}
By the above theorem, 
a fully-faithful functor (equivalence) of 
$\C$-equivariant DT categories 
induce a fully-faithful functor (equivalence)
of idempotent completions of $\mathbb{Z}/2$-periodic 
DT categories (see the argument of~\cite[Theorem~5.4]{TocatDT2}). 
Therefore in what follows, we focus on the 
$\C$-equivariant DT categories.

	\section{The stacks of filtered and graded objects: review}
		Halpern-Leistner~\cite{Halpinstab} developed the
	moduli theory of maps from $\Theta=[\mathbb{A}^{1}/\C]$ to a
	fixed Artin stack, 
	whose moduli stack is called the stack of
	filtered objects. 
	It is used to define the notion of 
	$\Theta$-stratifications,
	which generalizes Harder-Narasimhan stratifications 
	for moduli stacks of coherent sheaves and Kempf-Ness 
	stratifications for GIT quotient stacks. 
	The above moduli theory is extended to 
	the case of maps to derived Artin stacks by Halpern-Leistner-Preygel~\cite{Halpstack}.
	In this section, we review the theory of 
	stacks of filtered objects, $\Theta$-stratifications, 
	and some of their properties. 
\subsection{Quasi-coherent sheaves on theta stacks}\label{sec:filtst}
	Let $\Theta$ be the stack defined by 
	\begin{align*}
		\Theta \cneq [\mathbb{A}^1/\C]. 
		\end{align*}
	Here $\C$ acts on $\mathbb{A}^1$ by weight one. 
	Here we review some basic properties of quasi-coherent sheaves 
	on the stacks $\Theta \times T$ and $B\C \times T$ for a derived stack $T$.
	By the Rees construction, we have the following lemma: 
	\begin{lem}\emph{(\cite[Section~1.1]{HalpK32})}\label{lem:Rees}
		(i)
		The $\infty$-category $D_{\qcoh}(\Theta\times T)$
	is equivalent to the $\infty$-category of diagrams 
	\begin{align}\label{Rees}
		\eE_{\bullet}=
		\cdots \to \eE_{1} \to \eE_0 \to \eE_{-1} \to \cdots 
	\end{align}
	where $\eE_i \in D_{\qcoh}(T)$. 
	
	(ii) The $\infty$-category $D_{\qcoh}(B\C \times T)$ decomposes into 
	\begin{align*}
		D_{\qcoh}(B\C \times T)=\bigoplus_{i\in \mathbb{Z}} D_{\qcoh}(T)_{\wt=i}
	\end{align*}
	where $D_{\qcoh}(T)_{\wt=i}$ is a copy of $D_{\qcoh}(T)$ corresponding to the
	$\C$-weight $i$ part. 
	\end{lem}
	An object $\eE \in D_{\qcoh}(\Theta \times T)$
	determines a diagram (\ref{Rees}) 
	by 
	\begin{align*}
	\cdots \to	\pi_{T\ast}(\oO(i+1) \boxtimes \eE) \to \pi_{T\ast}(\oO(i) \boxtimes \eE) \to
		\pi_{T\ast}(\oO(i-1) \boxtimes \eE) \to \cdots,
		\end{align*}
	where $\pi_T \colon \Theta \times T \to T$ is the projection, 
	$\oO(i)$ is a line bundle on $\Theta$
	determined by the $\C$-character of weight $i$. 
	Conversely a diagram (\ref{Rees})
	determines an object of $D_{\qcoh}(\Theta \times T)$
	by 
	\begin{align*}
		\eE_{\bullet} \mapsto \bigoplus_{i \in \mathbb{Z}}\eE_i,
		\end{align*}
	where the action of $t \in \oO_{\mathbb{A}^1}=\mathbb{C}[t]$
	is given by $\eE_{i+1} \to \eE_i$. 
	
	We denote by 
	\begin{align}\label{mor:i}
		i_1 \colon \Spec \mathbb{C} \to \Theta, \ 
		i_0 \colon B\C \to \Theta
				\end{align}
			the morphisms 
	corresponding to $1 \in \mathbb{A}^1$, $0 \in \mathbb{A}^1$
	respectively. 
	Under the above Rees construction, the pull-back functors 
	\begin{align*}
		i_1^{\ast} \colon D_{\qcoh}(\Theta \times T) \to 
		D_{\qcoh}(T), \ 
		i_0^{\ast} \colon D_{\qcoh}(\Theta \times T) \to D_{\qcoh}(B\C \times T)
		\end{align*}
	are described as 
	\begin{align}\label{i_0ast}
		i_1^{\ast}(\eE_{\bullet})=\colim_{i\to -\infty}\eE_i, \ 
		i_0^{\ast}(\eE_{\bullet})=\bigoplus_{i\in \mathbb{Z}} \Cone(\eE_{i+1} \to \eE_i). 
		\end{align}
	Here $\Cone(\eE_{i+1} \to \eE_i)$ lies in $D_{\qcoh}(T)_{\wt=i}$. 
	
	Let $\pi_T \colon \Theta \times T \to T$
	and 
	$\pi_T' \colon B\C \times T \to T$ be the projections. 
	We have the pull-back functors 
	\begin{align*}
		\pi_T^{\ast} \colon D_{\qcoh}(T) \to D_{\qcoh}(\Theta \times T), \ 
		{\pi_T'}^{\ast} \colon D_{\qcoh}(T) \to D_{\qcoh}(B\C \times T). 
		\end{align*}
	Their left adjoint functors
	\begin{align*}
		\pi_{T+} \colon D_{\qcoh}(\Theta \times T) \to 
		D_{\qcoh}(T), \ 
		\pi'_{T+} \colon D_{\qcoh}(B\C \times T) \to 
		D_{\qcoh}(T)
		\end{align*}
	are given by (see~\cite[Lemma~1.3.1]{HalpK32})
	\begin{align}\label{pi+}
		\pi_{T+}(\eE_{\bullet})=
		\Cone\left( \eE_1 \to \colim_{i \to -\infty}\eE_i \right), \ 
		\pi'_{T+}\left(\bigoplus_{i\in \mathbb{Z}} \eE_i \right)=\eE_0. 
		\end{align}
		\subsection{The stacks of filtered objects}
	Let $\fM$ be a derived Artin stack with affine stabilizers. 
	The $\infty$-functor
	\begin{align*}
		dAff^{op} \to Sset, \ 
		T \mapsto \mathrm{Map}(\Theta \times T, \fM)
		\end{align*}
	is a derived Artin stack~\cite{Halpstack}, denoted by $\mathrm{Filt}(\fM)$
	and called the \textit{stack of filtered objects}.  
	Similarly the $\infty$-functor 
	\begin{align*}
		dAff^{op} \to Sset, \ 
		T \mapsto \mathrm{Map}(B\C \times T, \fM)
	\end{align*}
	is a derived Artin stack, denoted by $\mathrm{Grad}(\fM)$
	and called the \textit{stack of graded objects}. 
	There exist natural maps 
	\begin{align}\label{dia:Filt}
		\xymatrix{
		\mathrm{Filt}(\fM) \ar[r]^-{\ev_1} \ar[d]^-{\ev_0} & \fM \\
		\mathrm{Grad}(\fM) \ar[ur]_-{\tau}. \ar@/^18pt/[u]^-{\sigma} & 
	}
		\end{align}
	Here $\ev_0, \ev_1, \sigma, \tau$ are induced by morphisms
	$i_0$, $i_1$ in (\ref{mor:i}), the projection 
	$\Theta \to B\C$ and the natural morphism 
	$\Spec \mathbb{C} \to B\C$ respectively. 
	Note that we have $\tau=\ev_1 \circ \sigma$. 
		Moreover we have 
	\begin{align}\label{Filt:t0}
		t_0(\Filt(\fM))=\Filt(t_0(\fM)), \ 
		t_0(\Grad(\fM))=\Grad(t_0(\fM)). 
		\end{align}
	
		We note that there is a natural action of $B\C$ on $\Grad(\fM)$
	acting on the source of the maps
	$B\C \times T \to \fM$. 
	Consequently there is a natural morphism into the inertia stack
	\begin{align*}
		(\C)_{\Grad(\fM)} \to I_{\Grad(\fM)}. 
	\end{align*}
	We have the decomposition of the derived category  
	$\Dbc(\Grad(\fM))$ into weight space
	categories
	with respect to the above $B\C$-action 
	\begin{align}\label{grad:wt}
		\Dbc(\Grad(\fM))=\bigoplus_{j \in \mathbb{Z}} \Dbc(\Grad(\fM))_{\wt=j}. 
	\end{align}
	
	\begin{exam}\label{exam:frakU}
		Let $\fU$ be an affine derived scheme of the form (\ref{frak:U}).
		Suppose that a reductive algebraic group $G$
		acts on $Y$ such that $(V, s)$ is $G$-equivariant. 
		For a one parameter subgroup 
		$\lambda \colon \C \to G$, 
		we have the following commutative diagram 
		(see the notation in Subsection~\ref{subsec:exam:KN})
\begin{align*}
	\xymatrix{
		V^{\lambda=0}  \ar[d] \ar@/^10pt/[r] & \ar[l]
		V^{\lambda \ge 0} \ar@<-0.3ex>@{^{(}->}[r] \ar[d] & V \ar[d] \\
Y^{\lambda=0} \ar@/^10pt/[u]^-{s^{\lambda=0}}  \ar@/_10pt/[r]&	Y^{\lambda \ge 0} \ar@<-0.3ex>@{^{(}->}[r] \ar[l]
\ar@/^10pt/[u]^-{s^{\lambda \ge 0}}
	& Y, \ar@/_10pt/[u]_-{s}&
}
	\end{align*}	
Here the 
the horizontal arrows from left to right are natural inclusions, 
and the 
horizontal arrows from left to right
 are given by taking the limits 
for the $\lambda$-actions. 	
By setting $\fU^{\lambda \ge 0}$, $\fU^{\lambda=0}$ to be the 
derived zero loci of $s^{\lambda \ge 0}$, $s^{\lambda=0}$ respectively, 
we have the following diagram 
	\begin{align}\notag
	\xymatrix{
	[\fU^{\lambda \ge 0}/G^{\lambda \ge 0}]
	 \ar[r]^-{\ev_1} \ar[d]^-{\ev_0} & [\fU/G] \\
		[\fU^{\lambda=0}/G^{\lambda=0}]\ar[ur]_-{\tau}. 
		\ar@/^18pt/[u]^-{\sigma} & 
	}
\end{align}
It is proved in~\cite[Lemma~1.6.1]{HalpK32}
that $[\fU^{\lambda \ge 0}/G^{\lambda \ge 0}]$
is an open and closed substack of $\Filt([\fU/G])$, 
$[\fU^{\lambda=0}/G^{\lambda=0}]$
is its center, and 
any connected component of $\Filt([\fU/G])$ is a 
component of $[\fU^{\lambda \ge 0}/G^{\lambda \ge 0}]$
for a unique conjugacy class of $\lambda$. 
The decomposition (\ref{grad:wt}) at the components 
$[\fU^{\lambda=0}/G^{\lambda=0}]$ is 
\begin{align*}
	\Dbc([\fU^{\lambda=0}/G^{\lambda=0}])=\bigoplus_{j \in \mathbb{Z}}
	\Dbc([\fU^{\lambda=0}/G^{\lambda=0}])_{\wt=j}
	\end{align*}
where each component is $\lambda$-weight $j$-part 
with respect to $\lambda \colon \C \to G^{\lambda=0}$. 
	\end{exam}

The cotangent complexes of $\mathbb{L}_{\Filt(\fM)}$, $\mathbb{L}_{\Grad(\fM)}$ are 
described in the following way. 
Let us consider the following diagrams 
\begin{align}\label{dia:filgr}
	\xymatrix{
\Theta \times \Filt(\fM) \ar[r]^-{\ev_F} \ar[d]_-{\pi_F} & \fM \\
\Filt(\fM)
}, \
	\xymatrix{
	B\C \times \Grad(\fM) \ar[r]^-{\ev_G} \ar[d]_-{\pi_G'} & \fM \\
	\Grad(\fM)
}
	\end{align}
where horizontal arrows are universal maps and the vertical arrows are projections. 	
Then we have (see~\cite[Prop.~5.1.10]{Halpstack})
\begin{align}\label{L:filt}
	\mathbb{L}_{\Filt(\fM)}=
	\pi_{F+}(\ev_F^{\ast}\mathbb{L}_{\fM}), \ 
	\mathbb{L}_{\Grad(\fM)}=
	\pi_{G+}'(\ev_G^{\ast}\mathbb{L}_{\fM}).
	\end{align}
Here $\pi_{F+}$, $\pi_{G+}'$ are given by (\ref{pi+}). 
In the case that $\fM$ is quasi-smooth, we have the following: 
\begin{lem}\emph{(\cite[Lemma~2.2.4]{HalpK32})}\label{lem:qsmooth}
	If $\fM$ is quasi-smooth, then $\Filt(\fM)$ and $\Grad(\fM)$ are also 
	quasi-smooth and the morphism $\ev_0$ in the diagram (\ref{dia:Filt})
	 is a quasi-smooth morphism. 
	\end{lem}
As for the morphism $\ev_1$, the following $\Theta$-reductive 
condition is introduced in~\cite{Halpinstab}:
\begin{defi}\emph{(\cite[Definition~4.16]{Halpinstab})}
	A derived stack $\fM$ is called $\Theta$-reductive if the morphism 
	$\ev_1$ satisfies the valuative crieterion of properness. 
	\end{defi}
For example if $\mM=t_0(\fM)$ admits a good moduli 
space, then $\fM$ is $\Theta$-reductive (see~\cite[Theorem~A]{AHLH}). 
If $\fM$ is quasi-smooth and $\Theta$-reductive, 
then the diagram (\ref{dia:Filt})
is a diagram of quasi-smooth derived stacks such that $\ev_0$ is quasi-smooth 
and $\ev_1$ is proper. 
	
	\subsection{Theta stratifications}
	Let $\mathfrak{S}\subset \mathrm{Filt}(\fM)$ be an open and 
	closed substack. The \textit{center} of $\mathfrak{S}$ is defined to be
	the open and closed substack 
	$\mathfrak{Z} \cneq \sigma^{-1}(\mathfrak{S}) \subset \mathrm{Grad}(\fM)$.
	The diagram (\ref{dia:Filt}) restricts to the diagrams 
		\begin{align}\label{dia:Filt2}
		\xymatrix{
		\mathfrak{S} \ar[r]^-{\ev_1} \ar[d]^-{\ev_0} & \fM \\
		\mathfrak{Z}\ar[ur]_-{\tau}, \ar@/^18pt/[u]^-{\sigma} & 
		} \
		\xymatrix{
		\sS \ar[r]^-{\ev_1} \ar[d]^-{\ev_0} & \mM \\
		\zZ\ar[ur]_-{\tau}. \ar@/^18pt/[u]^-{\sigma} & 
	}
	\end{align}
Here the right diagram is obtained from the left by taking 
classical truncations. 
From (\ref{Filt:t0}), note that $\sS \subset \Filt(\mM)$ is open and closed, 
and $\zZ$ is the center of $\sS$. 
Following~\cite{Halpinstab, HalpK32}, 
the substack $\ffS \subset \Filt(\fM)$ is called 
a \textit{$\Theta$-stratum}
if the map
\begin{align*}
	\ev_1 \colon \ffS \hookrightarrow \Filt(\fM) \to \fM
	\end{align*}
	in the diagram (\ref{dia:Filt2}) 
	is a closed immersion. 
The last condition is also equivalent to that 
$\ev_1 \colon \sS \to \mM$ is a closed immersion
of classical stacks. 

	\begin{defi}\emph{(\cite[Definition~2.2]{Halpinstab})}
		\label{defi:Thetastrata}
	A \textit{$\Theta$-stratification} of $\fM$ 
	indexed by a totally ordered set $I$ with minimal element 
	$0 \in I$ consists of 
	\begin{enumerate}
		\item a collection of open substacks $\fM_{\le c} \subset \fM$
		for $c \in I$ such that $\fM_{\le c} \subset \fM_{\le c'}$
		when $c<c'$, 
		\item a $\Theta$-stratum $\mathfrak{S}_{c} \subset \mathrm{Filt}(\fM_{\le c})$
		for all $c\in I$ 
		with $\fM_{\le c} \setminus \ev_1(\mathfrak{S}_{c})=\fM_{<c}$, 
		\item for every $x \in \fM$, there is a minimal $c \in I$
		such that $x \in \fM_{\le c}$. 
		\end{enumerate}
	The semistable locus is defined to be 
	$\fM^{\rm{ss}} \cneq \fM_{\le 0} \subset \fM$. 
	\end{defi}
	We will often regard $\mathfrak{S}_{c}$ as a locally 
	closed substack of $\fM$
	by $\mathfrak{S}_c \stackrel{\ev_1}{\hookrightarrow} \fM_{\le c}
	\subset \fM$, and write the $\Theta$-stratification as
	\begin{align*}
		\fM=\fM^{\rm{ss}} \bigcup_{c>0} \mathfrak{S}_{c}. 
		\end{align*}
	
	Note that 
	\begin{align*}
		H^{\ast}(\Theta, \mathbb{R}) \cong \mathbb{R}[q], \ 
		q=c_1(\oO(1))
		\end{align*}
	where $\oO(1)$ is the line bundle on $\Theta$
	induced by the 
	weight one $\C$-character. 
	Then for any $\mathbb{R}$-line bundle $\lL$ on $\fM$
	with $l=c_1(\lL)$, 
	we have $q^{-1}f^{\ast}l \in \mathbb{R}$. 
	\begin{defi}
	A point $p \in \fM(\mathbb{C})$ is called 
	$l$\textit{-semistable} 
	if for any $f \colon \Theta \to \fM$ 
	with $f(1) \cong p$, 
	we have 
	$q^{-1}f^{\ast}l \ge 0$. 	
	\end{defi}
A map $f \colon B\C \to \fM$ is called non-degenerate
if the induced morphism $\C \to \Aut(f(0))$ is non-trivial, 
and $f \colon \Theta \to \fM$ is called non-degenerate if 
$f(0) \colon B\C \to \fM$ is non-degenerate. 
For an element $b \in H^4(\fM, \mathbb{R})$, 
we have the following positivity condition:
\begin{defi}\emph{(\cite[Definition~3.85]{Halpinstab})}\label{defi:bpos}
	An element $b \in H^4(\fM, \mathbb{R})$ is called 
	\textit{positive definite} if 
	for any non-degenerate
	$f \colon B\C \to \fM$,
	we have 
	$q^{-2} f^{\ast} b>0$. 
	\end{defi}
For $(l, b)$ such that $b$ is positive definite and 
a non-degenerate map 
$f \colon \Theta \to \fM$, 
 we set
	\begin{align}\label{mu:lb}
		\mu_{l, b}(f) \cneq 
		-\frac{q^{-1}f^{\ast}l}{\sqrt{q^{-2} f^{\ast} b}} \in \mathbb{R}. 
		\end{align}
	\begin{defi}\emph{(\cite[Definition~3.1.2]{HalpK32})}\label{defi:num}
	The numerical invariant 
	$\mu_{l, b}$ is said to define the $\Theta$-stratification 
	if there is a $\Theta$-stratification of $\fM$
	whose strata $\mathfrak{S} \subset \mathrm{Filt}(\fM)$
	are ordered by the values of $\mu$, and 
	each $f \in \mathfrak{S}$ with $f(1) \cong  p$
	is the unique maximizer of $\mu$
	(up to composition with a ramified covering 
	$(-)^n \colon \Theta \to \Theta$)
	satisfying the condition $f(1) \cong p$. 
	The $l$-semistable locus is denoted by $\fM^{l\sss} \subset \fM$. 
	\end{defi}

In the case that $\mM$ admits a good moduli space, 
the following result is proved in~\cite{HalpK32}. 
\begin{thm}\emph{(\cite[Theorem~3.1.3]{HalpK32})}\label{thm:theta}
	If $\mM$ admits a good moduli space, 
	then 
	$\mu_{l, b}$ given by (\ref{mu:lb}) defines the $\Theta$-stratification 
	\begin{align*}
		\mM= \sS_1 \sqcup \cdots \sqcup \sS_N \sqcup \mM^{l\sss},
	\end{align*}
	such that $\mM^{l\sss}$ also admits a good moduli space. 
\end{thm}

\begin{exam}\label{exam:KN}
	In the situation of Subsection~\ref{subsec:exam:KN}, 
	let us take 
	\begin{align*}
		l \in H^2(BG, \mathbb{R})=M_{\mathbb{R}}, \ b \in H^4(BG, \mathbb{R})
	\end{align*}
	such that $b$ is positive definite. 
	We regard them as elements of 
	$H^2([Y/G])$, $H^4([Y/G])$ by the pull-back 
	of the projection 
	$[Y/G] \to BG$. 
	A choice of $b$ corresponds to a Weyl-invariant 
	norm $\lvert \ast \rvert$ on $N_{\mathbb{R}}$. 
	By taking the quotient stacks of the 
	stratification (\ref{KN:strata}), we have the 
	$\Theta$-stratification of the quotient stack $\yY=[Y/G]$
	with respect to $\mu_{l, b}$
	\begin{align}\notag
		\yY=\sS_{1} \sqcup \sS_2 \sqcup \cdots \sqcup \sS_N \sqcup 
		\yY^{l\sss}. 
	\end{align} 
\end{exam}
Later we will use the following lemma: 
\begin{lem}\label{lem:fit:M}
	Let $\pi_{\mM} \colon \mM \to M$ be a good moduli space. 
	
	(i) 
We have the commutative diagram 
	\begin{align}\label{filt:gmoduli}
		\xymatrix{
		\Filt(\mM) \ar[r]^-{\ev_1} \ar[d]_-{\ev_0} & \mM \ar[d]^-{\pi_{\mM}}\\
		\Grad(\mM) \ar[r]_-{\pi_{\mM} \circ \tau}  & M. 
	}
\end{align}

(ii) For a morphism $M' \to M$ between algebraic spaces, we set 
$\mM'=\mM \times_M M'$
so that $\pi_{\mM'} \colon \mM' \to M'$ is a good moduli space morphism. 
Then the pull-back of the diagram (\ref{filt:gmoduli}) by 
$M' \to M$ is isomorphic to 
the diagram 
\begin{align}\notag
	\xymatrix{
		\Filt(\mM') \ar[r]^-{\ev_1} \ar[d]_-{\ev_0} & \mM' \ar[d]^-{\pi_{\mM'}}\\
		\Grad(\mM') \ar[r]_-{\pi_{\mM'} \circ \tau}  & M'. 
	}
\end{align}
	\end{lem}
\begin{proof}
	(i) 
	Let $g$ be the map
	$g \colon \Spec \mathbb{C} \to B\C \to \Theta$
	where the first map is the natural one and 
	the second one is $i_0$. 
	Then for any classical scheme $T$
	and a map $f \colon \Theta \times T  \to \mM$, 
	we have the commutative diagram 
	\begin{align*}
		\xymatrix{
T	\ar@<0.5ex>[r]^-{(i_1, \id)} \ar@<-0.5ex>[r]_-{(g, \id)} \ar[d] & \Theta \times T 
	\ar[r]^-{f} \ar[d] & \mM \ar[d]^-{\pi_{\mM}} \\
	T \ar@{=}[r] & T \ar[r] & M. 
	}
		\end{align*}
	Here the middle 
	vertical arrow is the projection which is nothing 
	but the good moduli space morphism for $\Theta \times T$, 
	and the bottom arrows are induced morphisms on good moduli spaces.
	The commutative diagram (\ref{filt:gmoduli}) follows from the above 
	commutative diagram. 
	
	(ii) The claim follows from 
	the following Cartesian diagrams (see~\cite[Corollary~1.30.1]{Halpinstab})
	\begin{align}\notag
		\xymatrix{
			\Filt(\mM') \ar[r] \ar[d]_-{\pi_{\mM'} \circ \ev_1} 
			& \Filt(\mM) \ar[d]^-{\pi_{\mM} \circ \ev_1}\\
			M' \ar[r]  & M, 
		} \
		\xymatrix{
			\Grad(\mM') \ar[r] \ar[d]_-{\pi_{\mM'} \circ \tau} 
			& \Grad(\mM) \ar[d]^-{\pi_{\mM} \circ \tau}\\
			M' \ar[r]  & M.
		}
	\end{align}	
	Here top horizontal arrows are induced by $\mM' \to \mM$.
	\end{proof}

	\section{The stacks of filtered and graded objects for $(-1)$-shifted cotangents}
	In this section, we describe the stack of filtered objects
	for the $(-1)$-shifted cotangent $\Omega_{\fM}[-1]$ in terms of 
	$(-2)$-shifted conormal stacks
	over $\Filt(\fM)$. 
	This is a generalization of the result proved in~\cite[Proposition~3.1]{THtype} that the
	moduli stacks
	of exact sequences of local surfaces are isomorphic to 
	$(-2)$-shifted conormal stacks over the moduli stacks of exact
	sequences of surfaces. 
	\subsection{Description via $(-2)$-shifted conormals}
	For a derived stack $\fM$, 
	let $\Omega_{\fM}[-1]$ be the $(-1)$-shifted cotangent stack of $\fM$. 
	We denote the diagram (\ref{dia:Filt}) for $\Omega_{\fM}[-1]$ as 
		\begin{align}\notag
		\xymatrix{
			\mathrm{Filt}(\Omega_{\fM}[-1]) \ar[r]^-{\ev_1^{\Omega}} \ar[d]^-{\ev_0^{\Omega}} & \Omega_{\fM}[-1] \\
			\mathrm{Grad}(\Omega_{\fM}[-1])\ar[ur]_-{\tau^{\Omega}}. \ar@/^18pt/[u]^-{\sigma^{\Omega}} & 
		}
	\end{align}
		On the other hand, 
	we have the following morphism from the diagram (\ref{dia:Filt})
	\begin{align}\label{mor:ev}
		(\ev_0, \ev_1) \colon \Filt(\fM) \to \Grad(\fM) \times \fM. 
		\end{align}
	We denote by 
	$\Omega_{(\ev_0, \ev_1)}[-2]$ the $(-2)$-shifted conormal stack 
	for the above morphism 
	\begin{align}\label{Omegaev}
		\Omega_{(\ev_0, \ev_1)}[-2]
		\cneq \Spec S_{\oO_{\fM}}(\mathbb{T}_{(\ev_0, \ev_1)}[2]) \to \Filt(\fM). 
		\end{align}
	Note that in the diagram (\ref{dia:Filt}), 
	we have the distinguished triangles on $\Filt(\fM)$
	\begin{align}\label{dist:L}
		&\ev_1^{\ast}\mathbb{L}_{\fM} \to \mathbb{L}_{\ev_0} \to \mathbb{L}_{(\ev_0, \ev_1)}
		\to \ev_1^{\ast}\mathbb{L}_{\fM}[1], \\
	\notag	&\ev_0^{\ast}\mathbb{L}_{\Grad(\fM)} \to \mathbb{L}_{\ev_1} \to \mathbb{L}_{(\ev_0, \ev_1)}
		\to \ev_0^{\ast}\mathbb{L}_{\Grad(\fM)}[1].
		\end{align}
	The last arrows induce the morphisms 
		\begin{align}\label{mor:Oev01}
		\xymatrix{
\Omega_{(\ev_0, \ev_1)}[-2] \ar[d]_-{h_0} \ar[r]^-{h_1} & \Omega_{\fM}[-1] \\
\Omega_{\Grad(\fM)}[-1]. & 	
}
		\end{align}
We have the following relationship between 
	the stack of filtered objects for $(-1)$-shifted cotangent stack and 
	$(-2)$-shifted conormal stack of the morphism (\ref{mor:ev}). 
	\begin{prop}\label{prop:compare}
		There exists an equivalence of derived stacks
		\begin{align}\label{equiv:Omega}
			\Filt(\Omega_{\fM}[-1]) \stackrel{\sim}{\to} \Omega_{(\ev_0, \ev_1)}[-2]
			\end{align}
		such that the following diagram is commutative 
		\begin{align}\label{dia:filt}
			\xymatrix{
				\Filt(\fM) \ar@{=}[d] & \Filt(\Omega_{\fM}[-1]) \ar[l] \ar[d]_-{\sim} \ar[r]^-{\ev_1^{\Omega}} & \Omega_{\fM}[-1] \ar@{=}[d] \\
				\Filt(\fM) & \Omega_{(\ev_0, \ev_1)}[-2] \ar[l] \ar[r]^-{h_1} & \Omega_{\fM}[-1]. 
				}
			\end{align}
		Here 
		the top left arrow is induced by the projection $\Omega_{\fM}[-1] \to \fM$, 
		the bottom left arrow is the natural morphism (\ref{Omegaev}). 
		\end{prop}
	\begin{proof}
		We first describe the 
		relative cotangent complex $\mathbb{L}_{(\ev_0, \ev_1)}$ for the morphism (\ref{mor:ev}). 
		Let $\ev_F \colon \Theta \times \Filt(\fM) \to \fM$
		 be the universal morphism as in the diagram (\ref{dia:filgr}).
		By the Rees construction in Lemma~\ref{lem:Rees}, the object 
		$\ev^{\ast}_F\mathbb{L}_{\fM} \in D_{\qcoh}(\Theta \times \Filt(\fM))$ 
		is described by a diagram 
		\begin{align}\label{ev:Rees}
			\ev_F^{\ast}\mathbb{L}_{\fM}
			= \cdots \to (\ev_F^{\ast}\mathbb{L}_{\fM})_{1} \to 
			(\ev_F^{\ast}\mathbb{L}_{\fM})_{0} \to
			(\ev_F^{\ast}\mathbb{L}_{\fM})_{-1} \to \cdots
			\end{align}
		for $(\ev_F^{\ast}\mathbb{L}_{\fM})_{i} \in D_{\qcoh}(\Filt(\fM))$. 
		By (\ref{pi+}) and (\ref{L:filt}), we have 
		\begin{align}\label{Lfilt2}
			\mathbb{L}_{\Filt(\fM)}=\Cone\left(	(\ev_F^{\ast}\mathbb{L}_{\fM})_{1}
			\to \colim_{i \to -\infty}	(\ev_F^{\ast}\mathbb{L}_{\fM})_{i} \right). 
			\end{align}
		From the commutative diagram
		\begin{align*}
			\xymatrix{
			\Filt(\fM) \ar[d]_-{i_1} \ar[rd]^-{\ev_1} & \\
			\Theta \times \Filt(\fM) \ar[r]_-{\ev_F} & \fM
		}
			\end{align*}
		and (\ref{i_0ast}), we have 
		\begin{align}\label{ev1ast}
			\ev_1^{\ast}\mathbb{L}_{\fM}=
			i_1^{\ast}\ev_F^{\ast}\mathbb{L}_{\fM}=
			\colim_{i\to -\infty}(\ev_F^{\ast}\mathbb{L}_{\fM})_i. 
			\end{align}
		Also from the commutative diagram 
		\begin{align*}
			\xymatrix{
	B\C \times \Filt(\fM) \ar[r]^-{(\id, \ev_0)} \ar[d]_{(i_0, \id)} 
	& B\C \times \Grad(\fM) \ar[d]^-{\ev_G} \\
	\Theta \times \Filt(\fM) \ar[r]_-{\ev_F} & \fM		
	}
			\end{align*}
		and (\ref{i_0ast}), (\ref{pi+}),
		we have  
		\begin{align}\label{ev0ast}
			\ev_0^{\ast}\mathbb{L}_{\Grad(\fM)}
			& \cong ((\id, \ev_0)^{\ast}\ev_G^{\ast}\mathbb{L}_{\fM})_{\wt=0} \\
		\notag	& \cong  ((i_0, \id)^{\ast}\ev_F^{\ast}\mathbb{L}_{\fM})_{\wt=0}  \\
		\notag & \cong \Cone\left( 	(\ev_F^{\ast}\mathbb{L}_{\fM})_{1}
		\to 	(\ev_F^{\ast}\mathbb{L}_{\fM})_{0} \right). 
			\end{align}
		We have the commutative diagram from (\ref{Lfilt2}), (\ref{ev1ast}), (\ref{ev0ast}), where 
		horizontal arrows are distinguished triangles
		\begin{align*}
			\xymatrix{
		\ev_0^{\ast}\mathbb{L}_{\Grad(\fM)} \ar[r] \ar[d]_-{\sim} & \mathbb{L}_{\Filt(\fM)} \ar[r] 
		\ar[d]_-{\sim} &
		\mathbb{L}_{\ev_0} \ar@{.>}[d] \\ 
		\Cone\left( 	(\ev_F^{\ast}\mathbb{L}_{\fM})_{1}
		\to 	(\ev_F^{\ast}\mathbb{L}_{\fM})_{0} \right) \ar[r] & 
		\Cone\left(	(\ev_F^{\ast}\mathbb{L}_{\fM})_{1}
		\to \ev_1^{\ast}\mathbb{L}_{\fM} \right) \ar[r] & 
		\Cone\left(	(\ev_F^{\ast}\mathbb{L}_{\fM})_{0}
		\to  \ev_1^{\ast}\mathbb{L}_{\fM}  \right).
		}
			\end{align*}
		Therefore there is an isomorphism 
		\begin{align*}
			\mathbb{L}_{\ev_0} \cong \Cone\left(	(\ev_F^{\ast}\mathbb{L}_{\fM})_{0}
			\to  \ev_1^{\ast}\mathbb{L}_{\fM}  \right)
			\end{align*}
		which makes the above diagram commutative. 
		We also have the commutative diagram, where 
		horizontal arrows are distinguished triangles 
		\begin{align}\label{dia:evmor}
			\xymatrix{
			\ev_1^{\ast}\mathbb{L}_{\fM} \ar[r] \ar@{=}[d] & \mathbb{L}_{\ev_0} \ar[r] 
		\ar[d]_-{\sim} &
		\mathbb{L}_{(\ev_0, \ev_1)} \ar@{.>}[d] \ar[r] & \ev_1^{\ast}\mathbb{L}_{\fM}[1] \ar@{=}[d] \\ 
	\ev_1^{\ast}\mathbb{L}_{\fM} \ar[r] & 
		\Cone\left(	(\ev_F^{\ast}\mathbb{L}_{\fM})_{0}
		\to \ev_1^{\ast}\mathbb{L}_{\fM} \right) \ar[r] & 
			(\ev_F^{\ast}\mathbb{L}_{\fM})_{0}[1]\ar[r] & \ev_1^{\ast}\mathbb{L}_{\fM}[1]
		}
			\end{align}
	Therefore there is an isomorphism
	\begin{align}\label{isom:Lev}
		\mathbb{L}_{(\ev_0, \ev_1)} \cong (\ev_F^{\ast}\mathbb{L}_{\fM})_0[1]
		\end{align}		
	which makes the above diagram commutative. 	
			
			Let $T$ be a derived stack, and consider a diagram 
			\begin{align}\label{dia:lift0}
				\xymatrix{
			& \Filt(\Omega_{\fM}[-1]) \ar[d]  \\
			T \ar@{.>}[ur] \ar[r]_-{f} & \Filt(\fM). 	
			}
				\end{align}
			By the definition of the stack of filtered objects, 
			the above diagram corresponds to a diagram 
			\begin{align}\label{dia:lift}
			\xymatrix{
				& \Omega_{\fM}[-1] \ar[d]  \\
			\Theta \times	T \ar@{.>}[ur] \ar[r]_-{g} & \fM. 	
			}
		\end{align}
	Here $f$, $g$ fit into the commutative diagram 
	\begin{align*}
		\xymatrix{
	\Theta \times T \ar[rd]^-{g} \ar[d]_{(\id, f)} & \\
	\Theta \times \Filt(\fM) \ar[r]_-{\ev_F} & \fM. 	
	}
		\end{align*}	
	By the definition of $\Omega_{\fM}[-1]$ and the above commutative diagram, 
	giving a dotted arrow in (\ref{dia:lift}) is equivalent to giving a morphism 
	\begin{align}\label{mor:lift}
		\oO_{\Theta \times T} \to g^{\ast}\mathbb{L}_{\fM}[-1]=(\id, f)^{\ast}\ev_F^{\ast}
		\mathbb{L}_{\fM}[-1]. 
		\end{align}
	Under the Rees construction, the above morphism 
	corresponds to a commutative diagram 
	\begin{align*}
		\xymatrix{
\cdots \ar[r] &	0 \ar[r] \ar[d] & \oO_T \ar[r]^-{\id} \ar[d] & 	\oO_T \ar[r]^-{\id} \ar[d] & \cdots \\
\cdots \ar[r] &	f^{\ast}(\ev_F^{\ast}\mathbb{L}_{\fM})_1[-1] \ar[r] & 
	f^{\ast}(\ev_F^{\ast}\mathbb{L}_{\fM})_0[-1] \ar[r] & 
	f^{\ast}(\ev_F^{\ast}\mathbb{L}_{\fM})_{-1}[-1] \ar[r] & \cdots
	}
		\end{align*}
			Together with the isomorphism (\ref{isom:Lev}), 
			we see that giving a morphism (\ref{mor:lift}) is equivalent to giving 
			a morphism 
			\begin{align}\label{OTf}
				\oO_T \to 	f^{\ast}(\ev_F^{\ast}\mathbb{L}_{\fM})_0[-1]
				\cong f^{\ast}
				\mathbb{L}_{(\ev_0, \ev_1)}[-2]. 
				\end{align}
			It follows that giving a dotted arrow in (\ref{dia:lift0}) is equivalent 
			to giving a dotted arrow in the diagram
				\begin{align}\notag
				\xymatrix{
					& \Omega_{(\ev_0, \ev_1)}[-2] \ar[d]  \\
					T \ar@{.>}[ur] \ar[r]_-{f} & \Filt(\fM).	
				}
			\end{align}
		Therefore we obtain an equivalence (\ref{equiv:Omega}) together with the left 
		commutative diagram in (\ref{dia:filt}). 
		The right commutative diagram in (\ref{dia:filt}) follows from the
		right square in the diagram (\ref{dia:evmor}). 
		\end{proof}
	
	In the proof of the above proposition, we described 
	the cotangent complex of $\Filt(\fM)$ as in (\ref{Lfilt2}). 
	Now suppose that $\fM$ is quasi-smooth and QCA, so in particular 
	$\Filt(\fM)$ is also quasi-smooth by Lemma~\ref{lem:qsmooth}. 
	Then $\mathbb{L}_{\Filt(\fM)}$ is 
	perfect, so we have its determinant line bundle 
	$\det \mathbb{L}_{\Filt(\fM)}$. 
	Later we will use its description. 
We use the following commutative diagram 
\begin{align}\label{dia:FGev}
	\xymatrix{
\Filt(\fM) \diasquare
 \ar[d]_-{\ev_0} & \Theta \times \Filt(\fM) \ar[l]_-{\pi_F} \ar[r]^-{(\mathrm{pr}, \id)} 
\ar[d]_-{(\id, \ev_0)} & B\C \times \Filt(\fM) \ar[d]_-{(i_0, \id)} &  \\
\Grad(\fM) & \Theta \times \Grad(\fM) \ar[l]^-{\pi_G}	\ar[r]_-{(\id, \sigma)} & 
\Theta \times \Filt(\fM) \ar[r]_-{\ev_F} & \fM. 
}
	\end{align}
\begin{lem}\label{lem:detFilt}
	There is an isomorphism of line bundles 
	\begin{align*}
		\det \mathbb{L}_{\Filt(\fM)} \cong 
		\ev_0^{\ast} \det(\pi_{G+}(\mathrm{id}, \sigma)^{\ast}\ev_F^{\ast}\mathbb{L}_{\fM}). 
		\end{align*}
	\end{lem}
\begin{proof}
	Let $(\ev_F^{\ast}\mathbb{L}_{\fM})_i$ be as in (\ref{ev:Rees}), and set 
	\begin{align}\label{gr:LM}
		\mathrm{gr}_i(\ev_F^{\ast}\mathbb{L}_{\fM}) \cneq 
		\Cone\left((\ev_F^{\ast}\mathbb{L}_{\fM})_{i+1} \to (\ev_F^{\ast}\mathbb{L}_{\fM})_i \right). 
		\end{align}
	We also set $\mathrm{gr}_{\ge j}(\ev_F^{\ast}\mathbb{L}_{\fM})$
	to be
	the direct sum of (\ref{gr:LM}) for all $i\ge j$, 
	and $\mathrm{gr}_{\bullet}(\ev_F^{\ast}\mathbb{L}_{\fM})$ to be
	the direct sum of (\ref{gr:LM}) for all $i$. 
	Then we have 
	\begin{align*}
		(i_0, \id)^{\ast}\ev_F^{\ast}\mathbb{L}_{\fM} \cong \mathrm{gr}_{\bullet}(\ev_F^{\ast}\mathbb{L}_{\fM})
		\end{align*}
	and they are perfect, so we have 
	$\mathrm{gr}_i(\ev_F^{\ast}\mathbb{L}_{\fM})=0$ for $\lvert i \rvert \gg 0$
	and each $\mathrm{gr}_i(\ev_F^{\ast}\mathbb{L}_{\fM})$ is also perfect. 
	It follows that each 
	$(\ev_F^{\ast}\mathbb{L}_{\fM})_i$ is also perfect. 	
	From the diagram (\ref{dia:FGev}), we have the isomorphisms
	\begin{align*}
		\ev_0^{\ast} \det(\pi_{G+}(\mathrm{id}, \sigma)^{\ast}\ev_F^{\ast}\mathbb{L}_{\fM})
		& \cong \det(\pi_{F+}(\id, \ev_0)^{\ast}(\id, \sigma)^{\ast}\ev_F^{\ast}\mathbb{L}_{\fM}) \\
		& \cong \det(\pi_{F+}(\mathrm{pr}, \id)^{\ast}(i_0, \id)^{\ast}\ev_F^{\ast}\mathbb{L}_{\fM}) \\
		& \cong \det \Cone\left( \mathrm{gr}_{\ge 1}(\ev_F^{\ast}\mathbb{L}_{\fM}) \to 
		\mathrm{gr}_{\bullet}(\ev_F^{\ast}\mathbb{L}_{\fM}) \right) \\
			&\cong \det \Cone\left(	(\ev_F^{\ast}\mathbb{L}_{\fM})_{1}
		\to \colim_{i \to -\infty}	(\ev_F^{\ast}\mathbb{L}_{\fM})_{i} \right) \\
		& \cong \det \mathbb{L}_{\Filt(\fM)}. 
		\end{align*}
	\end{proof}
	
	\subsection{The stacks of graded objects on $(-1)$-shifted cotangents}
		In a way similar to Proposition~\ref{prop:compare}, we have the following lemma: 
	\begin{lem}\label{lem:Grad}
		There exists an equivalence of derived stacks
	\begin{align}\label{equiv:OmegaG}
		\Grad(\Omega_{\fM}[-1]) \stackrel{\sim}{\to} \Omega_{\Grad(\fM)}[-1]
	\end{align}
	such that the following diagram is commutative 
	\begin{align}\label{dia:filtG}
		\xymatrix{
			\Grad(\fM) \ar@{=}[d] & \Grad(\Omega_{\fM}[-1]) \ar[l] \ar[d]_-{\sim} & 
			\Filt(\Omega_{\fM}[-1]) \ar[l]_-{\ev_0^{\Omega}} \ar[d]_-{\sim} \\
			\Grad(\fM) & \Omega_{\Grad(\fM)}[-1] \ar[l]  & \ar[l]_-{h_0} \Omega_{(\ev_0, \ev_1)}[-2]. 
		}
	\end{align}
	Here 
	the top left arrow is induced by the projection $\Omega_{\fM}[-1] \to \fM$, 
	the bottom left arrow is the natural projection, and the right vertical arrow 
	is given in Proposition~\ref{prop:compare}. 
			\end{lem}
	\begin{proof}
			Let $T$ be a derived stack, and consider a diagram 
		\begin{align}\notag
			\xymatrix{
				& \Grad(\Omega_{\fM}[-1]) \ar[d]  \\
				T \ar@{.>}[ur] \ar[r]_-{f} & \Grad(\fM). 	
			}
		\end{align}
		By the definition of the stack of graded objects, 
		the above diagram corresponds to a diagram 
		\begin{align}\label{dia:liftG}
			\xymatrix{
				& \Omega_{\fM}[-1] \ar[d]  \\
				B\C \times	T \ar@{.>}[ur] \ar[r]_-{g} & \fM. 	
			}
		\end{align}
		Here $f$, $g$ fit into the commutative diagram 
		\begin{align*}
			\xymatrix{
				B\C \times T \ar[rd]^-{g} \ar[d]_{(\id, f)} & \\
				B\C \times \Grad(\fM) \ar[r]_-{\ev_G} & \fM. 	
			}
		\end{align*}	
		By the definition of $\Omega_{\fM}[-1]$ and the above commutative diagram, 
		giving a dotted arrow in (\ref{dia:liftG}) is equivalent to giving a morphism 
		\begin{align}\notag
			\oO_{B\C \times T} \to g^{\ast}\mathbb{L}_{\fM}[-1]=(\id, f)^{\ast}\ev_G^{\ast}
			\mathbb{L}_{\fM}[-1]. 
		\end{align}
	By (\ref{pi+}) and (\ref{L:filt}), giving the above morphism is equivalent to 
	giving a morphism 
	\begin{align*}
		\oO_T \to f^{\ast}(\ev_G^{\ast}\mathbb{L}_{\fM})_{\wt=0}[-1]
		=f^{\ast}\mathbb{L}_{\Grad(\fM)}[-1]. 
		\end{align*}
	Therefore giving a dotted arrow in (\ref{dia:liftG}) is equivalent to 
	giving a dotted arrow in the diagram
		\begin{align}\notag
		\xymatrix{
			&  \Omega_{\Grad(\fM)}[-1] \ar[d]  \\
			T \ar@{.>}[ur] \ar[r]_-{f} & \Grad(\fM),	
		}
	\end{align}
	which implies the equivalence (\ref{equiv:OmegaG}) together with the left commutative diagram in (\ref{dia:filtG}). 
	
	Finally we check that the right diagram in (\ref{dia:filtG})
	is commutative. 
	Let us consider a diagram (\ref{dia:lift0})
	and its composition with 
	$\ev_0^{\Omega}$
		\begin{align}\notag
		\xymatrix{
			& \Filt(\Omega_{\fM}[-1]) \ar[d] \ar[r]^-{\ev_0^{\Omega}} &
			\Grad(\Omega_{\fM}[-1]) \\
			T \ar@{.>}[ur] \ar[r]_-{f} & \Filt(\fM).  & 	
		}
	\end{align}
The composition of the above dotted arrow 
with $\ev_0^{\Omega}$ corresponds to the composition of the
dotted arrow in the diagram 
(\ref{dia:lift}) with 
$i_1 \times \id_T \colon B\C \times T \to \Theta \times T$
	\begin{align}\notag
		\xymatrix{
		&	& \Omega_{\fM}[-1] \ar[d]  \\
		B\C \times T \ar[r]_-{i_0 \times \id_T} &	\Theta \times	T \ar@{.>}[ur] \ar[r]_-{g} & \fM. 	
		}
	\end{align}
		It corresponds to the morphism given by the pull-back 
	of (\ref{mor:lift}) by $(i_0 \times \id_T)$
	\begin{align*}
		\oO_T \to ((i_0, \id_T)^{\ast}g^{\ast}\mathbb{L}_{\fM})_{\wt=0}[-1]
		=\Cone( f^{\ast}(\ev_F^{\ast}\mathbb{L}_{\fM})_1
		\to  f^{\ast}(\ev_F^{\ast}\mathbb{L}_{\fM})_0)[-1],
		\end{align*}
	which is the composition of (\ref{OTf}) 
	with the natural 
	morphism 
	\begin{align*}
		f^{\ast}(\ev_F^{\ast}\mathbb{L}_{\fM})_0[-1]
	\to \Cone( f^{\ast}(\ev_F^{\ast}\mathbb{L}_{\fM})_1
	\to  f^{\ast}(\ev_F^{\ast}\mathbb{L}_{\fM})_0)[-1].
	\end{align*}
	Then the right commutative diagram in (\ref{dia:filtG}) follows from the commutative diagram 
	\begin{align*}
		\xymatrix{
 (\ev_F^{\ast}\mathbb{L}_{\fM})_0[-1] \ar[r] \ar[d]_-{\sim} & \Cone( (\ev_F^{\ast}\mathbb{L}_{\fM})_1
 \to  (\ev_F^{\ast}\mathbb{L}_{\fM})_0)[-1] \ar[d]_-{\sim} \\
 \mathbb{L}_{(\ev_0, \ev_1)}[-2] \ar[r]  & \ev_0^{\ast}\mathbb{L}_{\Grad(\fM)}[-1]. 	
}
		\end{align*}
				Here the left vertical arrow is (\ref{isom:Lev})
		and the 
		the right vertical arrow is (\ref{ev0ast}), and the horizontal arrows are 
		natural morphisms. 
		\end{proof}
	\subsection{Compatibility with singular supports}
	We now show that, using 
	Proposition~\ref{prop:compare} and Lemma~\ref{lem:Grad}, 
	some natural functors induced by the stacks of filtered objects 
	preserve singular supports in some sense. 
	We first consider the 
	following commutative diagram 
		\begin{align}\label{dia:filtev}
		\xymatrix{
			\Filt(\Omega_{\fM}[-1]) \ar[r] \ar[d]_-{\ev_0^{\Omega}} & 
			\Filt(\fM)	\ar[d]^-{\ev_0}  \\
			\Grad(\Omega_{\fM}[-1]) \ar[r]  & 	\Grad(\fM). 		
		}
	\end{align}
Here the horizontal arrows are induced by 
the projection $\Omega_{\fM}[-1] \to \fM$. 
	\begin{lem}\label{lem:isompi}
		The diagram (\ref{dia:filtev}) induces the 
		isomorphisms of connected components, 
		\begin{align}\label{dia:pi}
			\xymatrix{
	\pi_0(\Filt(\Omega_{\fM}[-1])) \ar[r]^-{\cong} \ar[d]_-{\cong} & 
		\pi_0(\Filt(\fM))	\ar[d]^-{\cong}  \\
			\pi_0(\Grad(\Omega_{\fM}[-1])) \ar[r]^-{\cong}  & 	\pi_0(\Grad(\fM)). 		
	}
			\end{align}
		\end{lem}
	\begin{proof}
		By Proposition~\ref{prop:compare} and 
Lemma~\ref{lem:Grad},
there exist $\C$-actions on 
$\Filt(\Omega_{\fM}[-1])$, $\Grad(\Omega_{\fM}[-1])$ which 
specialize to $\Filt(\fM)$, $\Grad(\fM)$ respectively. 
Therefore the horizontal arrows in (\ref{dia:pi}) are isomorphisms. 
 		The vertical isomorphisms are proved in~\cite[Lemma~1.24]{Halpinstab}. 
		\end{proof}
	
	Given a quasi-compact open and closed substack 
	$\ffS \subset \Filt(\fM)$ with the associated
	diagram (\ref{dia:Filt2}), 
	by Lemma~\ref{lem:isompi}
	we have the associated diagrams 
	for $\Omega_{\fM}[-1]$
	and its classical truncation
		\begin{align}\notag
		\xymatrix{
			\mathfrak{S}^{\Omega} \ar[r]^-{\ev_1^{\Omega}} \ar[d]^-{\ev_0^{\Omega}} & 
			\Omega_{\fM}[-1], \\
			\mathfrak{Z}^{\Omega} \ar[ur]_-{\tau^{\Omega}} \ar@/^18pt/[u]^-{\sigma^{\Omega}} & 
		}  \
	\xymatrix{
		\sS^{\Omega} \ar[r]^-{\ev_1^{\Omega}} \ar[d]^-{\ev_0^{\Omega}} & 
		t_0(\Omega_{\fM}[-1]). \\
		\zZ^{\Omega} \ar[ur]_-{\tau^{\Omega}} \ar@/^18pt/[u]^-{\sigma^{\Omega}} & 
	}
	\end{align}
	Here $\ffS^{\Omega}$ is the open and closed substack of 
	$\Filt(\Omega_{\fM}[-1])$ corresponding to $\ffS$
	by the top isomorphism in (\ref{dia:pi}), 
	$\ffZ^{\Omega}$ is the center of 
	$\ffS^{\Omega}$, and the right diagram is induced by 
	the left one by taking the classical truncation. 
	By Proposition~\ref{prop:compare} and Lemma~\ref{lem:Grad}, we have the 
	commutative isomorphisms 
	\begin{align}\label{dia:Somega}
		\xymatrix{
			\zZ^{\Omega} \ar[d]^-{\cong} & \sS^{\Omega} \ar[l]_-{\ev_0^{\Omega}} \ar[r]^-{\ev_1^{\Omega}} \ar[d]^-{\cong} 
			& t_0(\Omega_{\fM}[-1]) \ar@{=}[d] \\
			t_0(\Omega_{\ffZ}[-1]) & t_0(\Omega_{(\ev_0, \ev_1)|_{\ffS}}[-2]) \ar[l]_-{h_0} \ar[r]^-{h_1} & t_0(\Omega_{\fM}[-1]). 
		}	
	\end{align}

Suppose that $\fM$ is quasi-smooth and $\Theta$-reductive, 
so that in the diagram (\ref{dia:Filt2})
the morphism $\ev_0$ is quasi-smooth and $\ev_1$ is proper. 
Then we have the functor 
\begin{align}\label{funct:Ups0}
	\Upsilon \colon \Dbc(\ffZ) \stackrel{\ev_0^{\ast}}{\to}
	\Dbc(\ffS) \stackrel{\ev_{1\ast}}{\to} \Dbc(\fM). 
	\end{align}
Let $T_0 \subset \zZ^{\Omega}$
and $T_1 \subset t_0(\Omega_{\fM}[-1])$ be conical 
closed substacks such that the following holds:
\begin{align}\label{hpull}
	(\ev_0^{\Omega})^{-1}(T_0) \subset 
	(\ev_1^{\Omega})^{-1}(T_1) \subset \sS^{\Omega}. 
	\end{align}
Below we identify 
$T_0$ as a closed substack
in $t_0(\Omega_{\fZ}[-1])$ by the left isomorphism in (\ref{dia:Somega}). 
We have the following lemma: 
\begin{lem}\emph{(\cite[Proposition~2.4]{THtype})}\label{lem:Tpush}
	Under the condition (\ref{hpull}),
	the functor $\Upsilon$ in (\ref{funct:Ups0}) preserves singular supports, 
	i.e. it restricts to the functor 
	\begin{align*}
		\Upsilon \colon \cC_{T_0} \to \cC_{T_1}.
		\end{align*}
	\end{lem}
\begin{proof}
	Using the notation of the diagram (\ref{diagram:induced}), 
	the diagram (\ref{dia:Filt2}) induces the following diagram 
	\begin{align}\label{diagram:fN}
		\xymatrix{
			& & t_0(\Omega_{(\ev_0, \ev_1)|_{\ffS}}[-2])
			\ar@{}[dd]|\square
			\ar[ld]_-{g_0} \ar[rd]^-{g_1} \ar@/_30pt/[lldd]_-{h_0} 
			\ar@/^30pt/[rrdd]^-{h_1} & & \\
			& \ev_0^{\ast}t_0(\Omega_{\ffZ}[-1]) \ar[ld]_-{\ev_0^{\spadesuit}} 
			\ar[rd]^-{\ev_0^{\diamondsuit}} & & 
			\ev_1^{\ast}t_0(\Omega_{\fM}[-1]) \ar[ld]_-{\ev_1^{\diamondsuit}}
			\ar[rd]^-{\ev_1^{\spadesuit}} & \\
			t_0(\Omega_{\ffZ}[-1]) & & t_0(\Omega_{\ffS}[-1]) & & t_0(\Omega_{\fM}[-1]).
		}
	\end{align} 
	Here the middle square is Cartesian by~\cite[(2.18)]{THtype}. 
	As $\ev_0$ is quasi-smooth and $\ev_1$ is proper, 
	the morphism $\ev_0^{\diamondsuit}$ is a closed immersion 
	and $\ev_1^{\spadesuit}$ is proper, so $h_1$ is also proper. 
	The morphisms $h_0, h_1$ are restrictions of the classical truncations of the
	morphisms in (\ref{mor:Oev01}). 
	By~\cite[Lemma~8.4.2]{MR3300415}, the functor $\ev_0^{\ast}$ in (\ref{funct:Ups0})
	restricts to the functor 
	\begin{align*}
		\ev_0^{\ast} \colon \cC_{T_0} \to \cC_{\ev_0^{\diamondsuit}(\ev_0^{\spadesuit})^{-1}(T_0)}. 
		\end{align*}
	By~\cite[Lemma~8.4.5]{MR3300415}, 
	the functor $\ev_{1\ast}$ in (\ref{funct:Ups0})
	restricts to the functor 
\begin{align*}
	\ev_{1\ast} \colon 
	\cC_{\ev_0^{\diamondsuit}(\ev_0^{\spadesuit})^{-1}(T_0)}
\to \cC_{\ev_1^{\spadesuit}(\ev_1^{\diamondsuit})^{-1}\ev_0^{\diamondsuit}(\ev_0^{\spadesuit})^{-1}(T_0)}.
	\end{align*}	
The lemma follows from 
\begin{align*}
	\ev_1^{\spadesuit}(\ev_1^{\diamondsuit})^{-1}\ev_0^{\diamondsuit}(\ev_0^{\spadesuit})^{-1}(T_0)
	=h_1(h_0)^{-1}(T_0) \subset T_1. 
	\end{align*}
	Here the last inclusion follows from (\ref{hpull}) and the diagram 
	(\ref{dia:Somega}). 
	\end{proof}

\subsection{Comparison of $\Theta$-strata}
Finally in this section, we observe a relationship between 
$\Theta$-strata for $\mM$ and $t_0(\Omega_{\fM}[-1])$. 
Let $\ffS \subset \Filt(\fM)$ be as in the previous subsection, and 
set $\sS=t_0(\ffS)$ as in the diagram (\ref{dia:Filt2}). 
We have the commutative diagram 
\begin{align}\label{dia:SM}
	\xymatrix{
\sS^{\Omega} \ar[r]^-{\ev_1^{\Omega}} \ar[d] & t_0(\Omega_{\fM}[-1]) \ar[d] \\
\sS \ar[r]^-{\ev_1} & \mM. 
}
	\end{align}
	Here the left vertical arrow is induced 
	by the projection 
	$\Omega_{\fM}[-1] \to \fM$. 
	\begin{lem}\label{lem:compareS}
		Suppose that $\fM$ is quasi-smooth and 
		the morphism $\ev_1 \colon \sS \to \mM$ is a closed 
		immersion. 
		The $\ev_1^{\Omega} \colon \sS^{\Omega} \to t_0(\Omega_{\fM}[-1])$
		is a closed immersion. 
		\end{lem}
	\begin{proof}
		From the distinguished triangle (\ref{dist:L}),
		we have the exact sequence 
		of coherent sheaves on $\sS$
		\begin{align*}
			\hH^{-2}(\mathbb{L}_{\ev_0}) \to 
			\hH^{-2}(\mathbb{L}_{(\ev_0, \ev_1)}) \to 
			\hH^{-1}(\ev_1^{\ast}\mathbb{L}_{\fM}). 
			\end{align*} 
		By Lemma~\ref{lem:qsmooth}, 
		we have $\hH^{-2}(\mathbb{L}_{\ev_0})=0$. 
		Together with the middle vertical isomorphism in (\ref{dia:Somega}),
		we have the closed immersion 
		\begin{align}\label{closed:S}
			\sS^{\Omega} \hookrightarrow 
			\sS \times_{\mM} t_0(\Omega_{\fM}[-1]). 
			\end{align} 
		Therefore the lemma holds. 
		\end{proof}
	By the above lemma, 
	a $\Theta$-strata $\sS \hookrightarrow \mM$
	gives rise to a $\Theta$-strata 
	$\sS^{\Omega} \hookrightarrow t_0(\Omega_{\fM}[-1])$. 
	On the other hand, the diagram (\ref{dia:SM}) is not necessary 
	Cartesian (see Example~\ref{exam:pull} below), 
	so a $\Theta$-stratification for $\mM$ 
	does not necessary pull-back to a $\Theta$-stratification 
	for $t_0(\Omega_{\fM}[-1])$. 
	\begin{exam}\label{exam:pull}
		Let $V_0, V_1$ be finite dimensional $\C$-representations, 
		and set 
		\begin{align*}
			\fU=\Spec \mathbb{C}[V_1^{\vee}[1] \oplus V_0^{\vee}], \ 
			\fM=[\fU/\C]. 
			\end{align*}
		Here the differential on 
		$\mathbb{C}[V_1^{\vee}[1] \oplus V_0^{\vee}]$ is zero. 
		Then 
		we have 
		\begin{align*}
			t_0(\Omega_{\fM}[-1])=[(V_0 \oplus V_1^{\vee})/\C]
			\end{align*}
		and 
		a diagram (\ref{dia:SM}) is of the form 
		\begin{align*}
			\xymatrix{
		[(V_0 \oplus V_1^{\vee})^{\lambda \ge 0}/\C] \ar@<-0.3ex>@{^{(}->}[r] \ar[d] & 
			[(V_0 \oplus V_1^{\vee})/\C] \ar[d] \\
			[V_0^{\lambda \ge 0}/\C] \ar@<-0.3ex>@{^{(}->}[r] & 
			[V_0/\C]	
		}
			\end{align*}
		for some one parameter subgroup $\lambda \colon \C \to \C$. 
		Each horizontal arrow is a closed immersion, 
		but the above diagram is Cartesian if and only if 
		$V_1^{\lambda >0}=0$, or equivalently 
		$\hH^{-1}(\mathbb{L}_{\fM}|_{0})=V_1^{\vee}$ has only non-negative 
		weights. 
		\end{exam}
	
	We have the following sufficient condition for the diagram (\ref{dia:SM}) to be a 
	Cartesian: 
	\begin{lem}\label{lem:weight}
		Suppose that 
		$\hH^{-1}(\mathbb{L}_{\fM}|_{\zZ})$ has only non-negative 
		weights, where 
		$(-)|_{\zZ}$ is the pull-back via 
		$\zZ \hookrightarrow \fZ \stackrel{\tau}{\to} \fM$. 
		Then the diagram (\ref{dia:SM}) is a Cartesian. 
			\end{lem}
		\begin{proof}
			From the proof of Lemma~\ref{lem:compareS}, we have the exact sequence 
			\begin{align}\label{exact:Z}
				0 \to \hH^{-2}(\mathbb{L}_{(\ev_0, \ev_1)}|_{\zZ}) \to 
				\hH^{-1}(\mathbb{L}_{\fM}|_{\zZ}) \to 
				\hH^{-1}(\mathbb{L}_{\ev_0}|_{\zZ}). 
				\end{align}
			Here we have pulled back objects on $\fS$ via 
			$\zZ \hookrightarrow \fZ \stackrel{\sigma}{\to} \fS$. 
			We also have the distinguished triangle 
			\begin{align*}
				\mathbb{L}_{\fZ}|_{\zZ} \to \mathbb{L}_{\fS}|_{\zZ} \to \mathbb{L}_{\ev_0}|_{\zZ}				\end{align*}
			By~\cite[Lemma~1.5.5]{HalpK32}, the object $\mathbb{L}_{\fS}|_{\zZ}$ has only 
			non-positive weights such that 
			$\mathbb{L}_{\fZ}|_{\zZ}$ is the direct summand of the weight zero part. 
			Therefore $\mathbb{L}_{\ev_0}|_{\zZ}$ has only 
			negative weights. 
			By the assumption that $\hH^{-1}(\mathbb{L}_{\fM}|_{\zZ})$ has only non-negative 
			weights, it follows that the last map in (\ref{exact:Z}) is a zero map. 
			Therefore the second map in (\ref{exact:Z}) is an isomorphism, 
			so the closed immersion (\ref{closed:S}) is an isomorphism on $\zZ \hookrightarrow \sS$. 
			Since any point in $\sS$ is speciazlized to a point in $\zZ$, 
			and the locus in $\sS$ where (\ref{closed:S}) is 
			an isomorphism is open, it follows that 
			the closed immersion (\ref{closed:S}) is an isomorphism. 
			\end{proof}
	
	\section{Semiorthogonal decompositions of DT categories}
	In this section, we give a proof of Theorem~\ref{intro:main}. 
	In~\cite[Section~6]{TocatDT}, we proved a similar 
	result for moduli spaces of PT stable pairs using 
	categorified Hall products~\cite{PoSa}. The argument in this 
	section is almost a repetition of~\cite[Section~6]{TocatDT}
	with some modifications. 
	\subsection{Theta stratifications of $(-1)$-shifted cotangents}\label{subsec:theta-1}
	Let $\fM$ be a quasi-smooth and QCA derived stack, 
	such that its classical truncation $\mM=t_0(\fM)$
	 admits a good moduli space
	\begin{align*}
		\pi_{\mM} \colon \mM \to M. 
		\end{align*}
Let $\nN$ be the classical truncation of the $(-1)$-shifted cotangent
\begin{align}\label{stack:N}
	p_0 \colon \nN=t_0(\Omega_{\fM}[-1]) \to \mM. 
	\end{align} 
Here we discuss $\Theta$-stratification on $\nN$. 
First we have the following lemma. 
\begin{lem}\label{lem:gmoduliN}
	The stack $\nN$ also admits a good moduli space $\nN \to N$. 	
	\end{lem}
\begin{proof}
	Let $f$ be the composition 
	\begin{align*}
		f \colon 
		\nN \stackrel{p_0}{\to} \mM \stackrel{\pi_{\mM}}{\to} M.
		\end{align*}
	We set $N \cneq \Spec_M f_{\ast}\oO_{\nN}$. 
		Since $p_0$ is an affine morphism, the 
	natural morphism $\nN \to N$ is checked to be a good moduli space morphism. 
	\end{proof}
Let us take 
$\lL \in \Pic(\mM)_{\mathbb{R}}$
with $l=c_1(\lL) \in H^2(\mM, \mathbb{R})$ and $b \in H^4(\mM, \mathbb{R})$
such that $b$ is positive definite. 
We also regard them as elements in $\Pic(\nN)_{\mathbb{R}}$, 
$H^2(\nN, \mathbb{R})$ and 
$H^4(\nN, \mathbb{R})$ by the 
pull-back via (\ref{stack:N}). 
By Theorem~\ref{thm:theta} and Lemma~\ref{lem:gmoduliN}, 
there is a $\Theta$-stratification with respect to 
$l$ and $b$
\begin{align}\label{N:theta}
	\nN=\sS_1^{\Omega} \sqcup \cdots \sqcup \sS_N^{\Omega} \sqcup 
	\nN^{l\sss}. 
	\end{align}
For each $i$, 
we denote by $\sS_{\le i}^{\Omega}$ the 
union of $\sS_j^{\Omega}$ for $j\le i$
which is a closed substack of $\nN$. 
We have the diagram 
\begin{align}\label{dia:Nst}
		\xymatrix{
		\sS_i^{\Omega} \ar@<-0.3ex>@{^{(}->}[r]^-{\ev_1^{\Omega}} \ar[d]^-{\ev_0^{\Omega}} & 
		\nN \setminus \sS_{\le i-1}^{\Omega} \ar@<-0.3ex>@{^{(}->}[r]  & \nN \\
		\zZ_i^{\Omega} \ar@/^18pt/[u]^-{\sigma^{\Omega}} \ar[ur]_-{\tau^{\Omega}} & &
	}
	\end{align}
where $\zZ_i^{\Omega}$ is the center of $\sS_i^{\Omega}$. 
Note that we have open immersions 
\begin{align*}
	\sS_i^{\Omega} \subset \Filt(\nN \setminus \sS_{\le i-1}^{\Omega})
	\subset \Filt(\nN), \ 
		\zZ_i^{\Omega} \subset \Grad(\nN \setminus \sS_{\le i-1}^{\Omega})
	\subset \Grad(\nN)
	\end{align*}
where the first inclusions are open and closed, 
and the second inclusions are induced by the open immersion 
$\nN\setminus \sS_{\le i-1}^{\Omega} \subset \nN$. 

We define 
\begin{align*}
	\overline{\sS}_i^{\Omega} \subset \Filt(\nN), \ 
	\overline{\zZ}_i^{\Omega} \subset \Grad(\nN)
	\end{align*}
to be the smallest open and closed substacks which contain
$\sS_i^{\Omega}$, $\zZ_i^{\Omega}$ respectively. 
Note that 
$\overline{\zZ}_i^{\Omega}$ is the center of $\overline{\sS}_i^{\Omega}$, 
and the diagram (\ref{dia:Nst}) extends to the diagram 
\begin{align}\label{dia:Nst2}
	\xymatrix{
		\overline{\sS}_i^{\Omega} \ar[r]^-{\ev_1^{\Omega}} \ar[d]^-{\ev_0^{\Omega}} & 
		\nN   \\
		\overline{\zZ}_i^{\Omega}.  \ar@/^18pt/[u]^-{\sigma^{\Omega}} \ar[ur]_-{\tau^{\Omega}} & 
	}
\end{align} 
We have the following lemma: 
\begin{lem}\label{lem:SZcart}
	The following diagrams are Cartesian 
	\begin{align}\label{lem:Cart}
				\xymatrix{
				\sS_i^{\Omega} \ar@<-0.3ex>@{^{(}->}[r] \ar@<-0.3ex>@{^{(}->}[d]_-{\ev_1^{\Omega}} \diasquare & \overline{\sS}_i^{\Omega} \ar[d]^-{\ev_1^{\Omega}} \\
				\nN \setminus \sS_{\le i-1}^{\Omega}  \ar@<-0.3ex>@{^{(}->}[r] & \nN,  	
			} \quad 
				\xymatrix{
			\sS_i^{\Omega} \ar@<-0.3ex>@{^{(}->}[r] \ar[d]_-{\ev_0^{\Omega}}  \diasquare & \overline{\sS}_i^{\Omega} \ar[d]^-{\ev_0^{\Omega}} \\
			\zZ_i^{\Omega}  \ar@<-0.3ex>@{^{(}->}[r] & \overline{\zZ}_i^{\Omega}. 	
		}
	\end{align}	
\end{lem}
\begin{proof}
	The left diagram is Cartesian by the unique maximizer condition in Definition~\ref{defi:num}. 
	As for the right diagram, let 
	us take a map $f \colon \Theta \to \nN$ in $\overline{\sS}_i^{\Omega}$
	and suppose that 
	$f(0) \colon B\C \to \nN$ lies $\zZ_i^{\Omega}$. 
	Then $f(0)$ factors as 
	\begin{align*}
		f(0) \colon B\C \to \nN \setminus \sS_{\le i-1}^{\Omega}
		\subset \nN. 
		\end{align*}
	Therefore $f(1) \in \nN \setminus \sS_{\le i-1}^{\Omega}$ as 
	$\nN \setminus \sS_{\le i-1}^{\Omega} \subset \nN$ is open. 
	By the left Cartesian diagram in (\ref{lem:Cart}), we conclude that 
	$f$ lies in $\sS_i^{\Omega}$. 
\end{proof}
We define open and closed substacks 
\begin{align*}
	\ffS_i \subset \Filt(\fM), \ 
	\ffZ_i \subset \Grad(\fM)
	\end{align*}
to be 
corresponding to 
$\overline{\sS}_i^{\Omega}$, $\overline{\zZ}_i^{\Omega}$
under the isomorphisms (\ref{dia:pi}). 
Then $\ffZ_i$ is the center of $\ffS_i$
and we have the following diagram as in (\ref{dia:Filt2})
	\begin{align}\label{dia:Filt4}
	\xymatrix{
		\mathfrak{S}_i \ar[r]^-{\ev_1} \ar[d]^-{\ev_0} & \fM \\
		\mathfrak{Z}_i \ar[ru]_-{\tau}. \ar@/^18pt/[u]^-{\sigma} & 
	}
\end{align}
Note that, although $\sS_i^{\Omega}$ forms a $\Theta$-stratification for $\nN$, 
the morphism $\ev_1$ in (\ref{dia:Filt4}) is not necessary a closed immersion so that 
$\fS_i$ may not form a $\Theta$-stratification of $\fM$. 
We set $T_i^{\Omega}$ to be  
\begin{align*}
	T_i^{\Omega} \cneq \overline{\zZ}_i^{\Omega} \setminus \zZ_i^{\Omega}
	\subset \overline{\zZ}_i^{\Omega} \stackrel{\cong}{\to}t_0(\Omega_{\ffZ_i}[-1]). 
	\end{align*}
Here the last isomorphism is given in Lemma~\ref{lem:Grad}. 
We regard $T_i^{\Omega}$ as a closed substack of $t_0(\Omega_{\ffZ_i}[-1])$
by the above isomorphism, which is conical. 
Following Definition~\ref{defi:DTcat}, the 
DT categories for $\zZ_i^{\Omega}$ and $\nN \setminus \sS_{\le i-1}^{\Omega}$ are defined by 
\begin{align*}
	\dDT^{\C}(\zZ_i^{\Omega}) \cneq \Dbc(\ffZ_i)/\cC_{T_i^{\Omega}}, \ 
	\dDT^{\C}(\nN \setminus \sS_{\le i-1}^{\Omega}) \cneq 
	\Dbc(\fM)/\cC_{\sS_{\le i-1}^{\Omega}}. 
	\end{align*}

Since $\fM$ is quasi-smooth and $\Theta$-reductive, 
in the diagram (\ref{dia:Filt4})
the morphism $\ev_0$ is quasi-smooth and $\ev_1$ is proper. 
Therefore for each $i$, we have the functor
\begin{align}\label{funct:ev}
	\Upsilon_i \colon \Dbc(\ffZ_i) \stackrel{\ev_0^{\ast}}{\to} \Dbc(\ffS_i) \stackrel{\ev_{1\ast}}{\to}
	\Dbc(\fM). 
	\end{align}
\begin{lem}\label{lem:descend}
	The functor (\ref{funct:ev}) descends to the functor 
	\begin{align}\label{funct:Upsi}
		\Upsilon_i \colon	\dDT^{\C}(\zZ_i^{\Omega}) 
		\to \dDT^{\C}(\nN \setminus \sS_{\le i-1}^{\Omega}). 
		\end{align}
	\end{lem}
\begin{proof}
	By the Cartesian squares (\ref{lem:Cart}) and Lemma~\ref{lem:SZcart}, 
	we have the identity in the diagram (\ref{dia:Nst2})
	\begin{align}\label{ev:TSomega}
		(\ev_0^{\Omega})^{-1}(T_i^{\Omega})=(\ev_1^{\Omega})^{-1}(\sS_{\le i-1}^{\Omega})
		=\overline{\sS}_i^{\Omega} \setminus \sS_i^{\Omega}
		\subset \overline{\sS}_i^{\Omega}.
		\end{align}
	Therefore by Lemma~\ref{lem:Tpush}, the functor 
 (\ref{funct:ev}) sends 
	$\cC_{T_i^{\Omega}}$ to $\cC_{\sS_{\le i-1}^{\Omega}}$. 
	By taking the quotients, we obtain the functor (\ref{funct:Upsi})
	as a descendant of (\ref{funct:ev}). 
	\end{proof}

The weight decomposition (\ref{grad:wt})
respects the singular supports, 
so we have the decomposition of the 
DT category for $\zZ_i^{\Omega}$ into 
 weight space categories (see~\cite[Subsection~3.2.3]{TocatDT})
\begin{align*}
	\dDT^{\C}(\zZ_i^{\Omega})=\bigoplus_{j\in \mathbb{Z}}
	\dDT^{\C}(\zZ_i^{\Omega})_{\wt=j}. 
	\end{align*}
We denote by 
\begin{align*}
	\xymatrix{
			\dDT^{\C}(\zZ_i^{\Omega})
\ar@<0.5ex>[r]^-{\mathrm{pr}_j}  & \ar@<0.5ex>[l]^-{i_j}
\dDT^{\C}(\zZ_i^{\Omega})_{\wt=j}
}
\end{align*}
the projection onto the weight space category, 
inclusion from the weight space category, respectively. 
We define $\Upsilon_{i, j}$ to be the composition 
\begin{align}\label{Ups:ij}
	\Upsilon_{i, j} \colon 
	\dDT^{\C}(\zZ_i^{\Omega})_{\wt=j}
	\stackrel{i_j}{\hookrightarrow} 
	\dDT^{\C}(\zZ_i^{\Omega})
	\stackrel{\Upsilon_i}{\to}
	\dDT^{\C}(\nN \setminus \sS_{\le i-1}^{\Omega}). 
	\end{align}

\subsection{Descriptions along formal fibers}
Here we give a description of the diagram (\ref{dia:Filt4}), 
formally locally on the good moduli space 
$M$. Below 
we assume that the good moduli space morphism 
$\pi_{\mM} \colon \mM \to M$ satisfies the formal neighborhood 
theorem in Definition~\ref{defi:formalneigh}. 
We take a closed point $y \in M$ and 
use the same symbol $y \in \mM$ to denote the unique 
closed point in the fiber of $\pi_{\mM} \colon \mM \to M$ at $y$, 
and $G_y \cneq \Aut(y)$. 
By the formal neighborhood theorem, 
there exists a $G_y$-equivariant Kuranishi map 
\begin{align}\label{kuranishi:2}
	\kappa_y \colon \widehat{\hH}^0(\mathbb{T}_{\fM}|_{y}) \to 
	\hH^1(\mathbb{T}_{\fM}|_{y})
\end{align}
as in Definition~\ref{defi:formalneigh}, 
and an equivalence 
\begin{align*}
	\widehat{\fM}_y \sim [\widehat{\fU}_y/G_y]
	\end{align*}
(see the notation in Subsection~\ref{subsec:gmoduli}). 

From Lemma~\ref{lem:fit:M}, we have the commutative diagram 
\begin{align}\notag
	\xymatrix{
\sS_i \ar[r]^-{\ev_1} \ar[d]_-{\ev_0} & \mM \ar[d]^-{\pi_{\mM}} \\
\zZ_i \ar[r]_-{\pi_{\mM}\circ \tau} & M. 	
}
		\end{align}
	Here $\sS_i$, $\zZ_i$ are classical truncations of 
	$\fS_i$, $\fZ_i$ in the diagram (\ref{dia:Filt4}). 
By taking the fiber products with 
$\Spec \widehat{\oO}_{M, y} \to M$, 
we obtain the diagram
\begin{align}\label{formal:y}
	\xymatrix{
		\widehat{\ffS}_{i, y}  \ar[rr]^-{\ev_1} \ar[dd]_-{\ev_0} &  \diasquare  & \widehat{\fM}_y \\ 
\diasquare	& \widehat{\sS}_{i, y} \ar@<-0.3ex>@{_{(}->}[lu] \ar[r]^-{\ev_1} \ar[d]_-{\ev_0} & \widehat{\mM}_y \ar@<-0.3ex>@{_{(}->}[u] \ar[d]^-{\pi_{\mM}} \\
\widehat{\ffZ}_{i, y}	& \widehat{\zZ}_{i, y} \ar[r] \ar@<-0.3ex>@{_{(}->}[l]
 & \Spec \widehat{\oO}_{M, y}.	
}
	\end{align}
Here $\widehat{\sS}_{i, y} \cneq \sS_i \times_M \Spec \widehat{\oO}_{M, y}$
is an open and closed substack 
of $\Filt(\widehat{\mM}_y)$ by Lemma~\ref{lem:fit:M} (ii), 
$\widehat{\zZ}_{i, y}$ is its center, 
and 
\begin{align*}
	\widehat{\ffS}_{i, y}
	\subset \Filt(\widehat{\fM}_y), \ 
		\widehat{\ffZ}_{i, y} \subset \Grad(\widehat{\fM}_y)
		\end{align*}
	 are the 
corresponding 
components under the isomorphisms (\ref{Filt:t0}). 
We have the following lemma: 
\begin{lem}\label{lem:Cart:MS}
	The following diagrams are Cartesians: 
\begin{align}\label{formal:y2}
		\xymatrix{
		\widehat{\ffZ}_{i, y} \ar[d]_-{\widehat{i}_y} \diasquare
		& \ar[l]_-{\ev_0} \widehat{\ffS}_{i, y}
		\ar[r]^-{\ev_1} \ar[d]_-{\widehat{j}_y} \diasquare & 
		\widehat{\fM}_y \ar[d]^-{\widehat{\iota}_y} \\
		\ffZ_i  & \ar[l]_-{\ev_0} \ffS_i \ar[r]^-{\ev_1} & \fM. 
	}
	\end{align}
\end{lem}
\begin{proof}
The diagrams (\ref{formal:y2}) are Cartesians on 
classical truncations by the constructions. 
It is enough to show that 
\begin{align*}
	\widehat{j}_y^{\ast}\mathbb{L}_{\fS_i} 
\stackrel{\cong}{\to} 
\mathbb{L}_{\widehat{\fS}_{i, y}}, \ 
	\widehat{i}_y^{\ast}\mathbb{L}_{\fZ_i} 
\stackrel{\cong}{\to} 
\mathbb{L}_{\widehat{\fZ}_{i, y}}. 
\end{align*}
Since we have the isomorphism 
$\widehat{\iota}_{y}^{\ast}\mathbb{L}_{\fM}
\stackrel{\cong}{\to}
\mathbb{L}_{\widehat{\fM}_y}$, 
the above isomorphisms easily follow 
from the descriptions of cotangent complexes 
of stacks filtered objects and graded objects in terms of 
$\mathbb{L}_{\fM}$ as in (\ref{L:filt}). 
	\end{proof}

For a one parameter subgroup $\lambda \colon \C \to G_y$, 
from Example~\ref{exam:frakU} and Lemma~\ref{lem:fit:M} (ii)
we have the following diagram 
\begin{align}\notag
	\xymatrix{
		\left[\hH^0(\mathbb{T}_{\fM}|_{y})^{\lambda \ge 0}/G_y^{\lambda \ge 0}
		\right] 
		\ar[r] \ar[d] & 
		\left[\hH^0(\mathbb{T}_{\fM}|_{y})/G_y
		\right] \ar[d] \\
		\left[\hH^0(\mathbb{T}_{\fM}|_{y})^{\lambda = 0}/G_y^{\lambda = 0}
	\right] \ar[r] \ar[ru] \ar@/^18pt/[u]	& 
\hH^0(\mathbb{T}_{\fM}|_{y})\ssslash G_y.	
	}	
	\end{align}
By taking the formal fibers at $0 \in \hH^0(\mathbb{T}_{\fM}|_{y})\ssslash G_y$, 
we obtain the commutative diagram 
\begin{align}\notag
	\xymatrix{
		\left[\widehat{\hH}^0(\mathbb{T}_{\fM}|_{y})^{\lambda \ge 0}/G_y^{\lambda \ge 0}
		\right] 
		\ar[r] \ar[d] & 
		\left[\widehat{\hH}^0(\mathbb{T}_{\fM}|_{y})/G_y
		\right] \ar[d] \\
		\left[\widehat{\hH}^0(\mathbb{T}_{\fM}|_{y})^{\lambda = 0}/G_y^{\lambda = 0}
		\right] \ar[r]	\ar@/^18pt/[u] \ar[ur] & 
		\widehat{\hH}^0(\mathbb{T}_{\fM}|_{y})\ssslash G_y.	
	}	
\end{align}
Since $\kappa_y$ in (\ref{kuranishi:2}) is $G_y$-equivariant 
it restricts to the following maps 
\begin{align*}
	&\kappa_y^{\lambda \ge 0} \colon \widehat{\hH}^0(\mathbb{T}_{\fM}|_{y})^{\lambda \ge 0} 
	\to \hH^1(\mathbb{T}_{\fM}|_{y})^{\lambda \ge 0}, \\ 
	&\kappa_y^{\lambda=0} \colon \widehat{\hH}^0(\mathbb{T}_{\fM}|_{y})^{\lambda = 0} 
	\to \hH^1(\mathbb{T}_{\fM}|_{y})^{\lambda = 0},
	\end{align*}
which are $G_y^{\lambda \ge 0}$-equivariant, 
$G_y^{\lambda=0}$-equivariant, respectively. 
We define 
\begin{align*}
	\widehat{\fU}_y^{\lambda \ge 0} \hookrightarrow 
	\widehat{\hH}^0(\mathbb{T}_{\fM}|_{y})^{\lambda \ge 0}, \ 
	\widehat{\fU}_y^{\lambda = 0}
	 \hookrightarrow \widehat{\hH}^0(\mathbb{T}_{\fM}|_{y})^{\lambda = 0} 
	\end{align*}
to be the derived zero loci 
of $\kappa_y^{\lambda \ge 0}$, 
$\kappa_y^{\lambda=0}$ respectively. 

By Example~\ref{exam:frakU}, 
for each $1\le i\le N$ there is a 
one parameter subgroup 
\begin{align}\label{lambda:i}
	\lambda_i \colon \C \to G_y
	\end{align}
such that 
each component of the outer diagram in (\ref{formal:y}) 
\begin{align}\notag
	\xymatrix{
		\widehat{\fS}_{i, y} \ar[r]^-{\ev_1} \ar[d]^-{\ev_0} &
	\widehat{\fM}_{i, y} \\
	\widehat{\fZ}_{i, y} \ar@/^18pt/[u]^-{\sigma} \ar[ru]_-{\tau} & 
	}	
\end{align}
is equivalent 
to the diagram of the 
following form for $\lambda=\lambda_i$
\begin{align}\label{dia:lambda}
	\xymatrix{
	[\widehat{\fU}_y^{\lambda \ge 0}/G_y^{\lambda \ge 0}] \ar[r]^-{\ev_1} \ar[d]^-{\ev_0} &
	[\widehat{\fU}_y/G_y] \\
	[\widehat{\fU}_y^{\lambda=0}/G_y^{\lambda=0}]. \ar@/^18pt/[u]^-{\sigma} \ar[ru]_-{\tau} & 
}	
	\end{align}

\subsection{Descriptions of $(-1)$-shifted cotangents along formal fibers}
Here we give a formal local description of the
stack $\nN$ and its $\Theta$-stratification 
 (\ref{N:theta}). 
Let us consider the composition 
\begin{align*}
	\nN \stackrel{p_0}{\to} \mM \stackrel{\pi_{\mM}}{\to} M. 
	\end{align*}
We denote by $\widehat{\nN}_y$ the formal fiber of the above 
morphism at $y \in M$, i.e. 
\begin{align*}
	\widehat{\nN}_y \cneq \nN \times_M \Spec \widehat{\oO}_{M, y}. 
	\end{align*}
Let $\kappa_y$ be the Kuranishi map (\ref{kuranishi:2}). 
We define the following function, determined by $\kappa_y$ as in (\ref{func:w})
\begin{align*}
	\widehat{w}_y \colon 
	\left[ \left(\widehat{\hH}^0(\mathbb{T}_{\fM}|_{y}) \oplus 
	\hH^1(\mathbb{T}_{\fM}|_{y})^{\vee} \right) /G_y \right]
	\to \mathbb{C}, \ 
	\widehat{w}_y(-, -)=\langle \kappa_y(-), - \rangle. 
	\end{align*}
Then from the identity (\ref{Omega:U}), the stack 
$\widehat{\nN}_y$ is isomorphic to 
the critical locus of $\widehat{w}_y$
\begin{align}\notag
	\widehat{\nN}_y\cong
	\Crit(\widehat{w}_y) \subset \left[\left(\widehat{\hH}^0(\mathbb{T}_{\fM}|_{y}) \oplus 
	\hH^1(\mathbb{T}_{\fM}|_{y})^{\vee}\right) /G_y \right]. 
	\end{align}

We next give a formal local description 
of the $\Theta$-stratification (\ref{N:theta}) for $\nN$. 
We have the following commutative diagram 
\begin{align*}
	\xymatrix{
\sS_i^{\Omega}	\ar@<-0.3ex>@{^{(}->}[r] \ar[d]_-{\ev_0^{\Omega}} 
\diasquare & \overline{\sS}_i^{\Omega}
\ar[r]^-{\ev_1^{\Omega}}  \ar[d]_-{\ev_0^{\Omega}} & \nN \ar[r]^-{p_0} \ar[d]_-{\pi_{\nN}}
& \mM \ar[d]_-{\pi_{\mM}} \\
\zZ_i^{\Omega}	\ar@<-0.3ex>@{^{(}->}[r]  & \overline{\zZ}_i^{\Omega}
\ar[r]^-{\pi_{\nN} \circ \tau^{\Omega}} & N \ar[r] & M. 	
}
	\end{align*} 
Here the left Cartesian diagram follows from Lemma~\ref{lem:SZcart}, 
the middle square follows from Lemma~\ref{lem:fit:M}, and 
the right bottom arrow is the induced morphism on good moduli 
spaces. By taking the fiber products with 
$\Spec \widehat{\oO}_{M, y} \to M$, we obtain 
the commutative diagram 
\begin{align}\label{Somega}
	\xymatrix{
		\widehat{\sS}_{i, y}^{\Omega}	\ar@<-0.3ex>@{^{(}->}[r] \ar[d]_-{\ev_0^{\Omega}} 
		\diasquare & \overline{\sS}_{i, y}^{\Omega}
		\ar[r]^-{\ev_1^{\Omega}}  \ar[d]_-{\ev_0^{\Omega}} & 
		\widehat{\nN}_y \ar[r]^-{p_0} \ar[d]_-{\pi_{\nN}}
		& \widehat{\mM}_y \ar[d]_-{\pi_{\mM}} \\
		\widehat{\zZ}_{i, y}^{\Omega}	\ar@<-0.3ex>@{^{(}->}[r]  & \overline{\zZ}_{i, y}^{\Omega}
		\ar[r]^-{\pi_{\nN} \circ \tau^{\Omega}} & \widehat{N}_y \ar[r] & \Spec \widehat{\oO}_{M, y}. 	
	}
\end{align} 
By~\cite[Corollary~1.30.1]{Halpinstab}, 
the stacks $\widehat{\sS}_{i, y}^{\Omega}$, $\overline{\sS}^{\Omega}_{i, y}$
are open and closed substacks
\begin{align*}
	\widehat{\sS}_{i, y}^{\Omega} \subset \Filt(\widehat{\nN}_y \setminus \widehat{\sS}_{\le i-1, y}^{\Omega}), \ \overline{\sS}^{\Omega}_{i, y} \subset \Filt(\widehat{\nN}_y), 
	\end{align*}
and 
the $\Theta$-stratification (\ref{N:theta}) induces the $\Theta$-stratification 
on $\widehat{\nN}_y$
\begin{align}\notag
		\widehat{\nN}_y=
	\widehat{\sS}_{1, y}^{\Omega} \sqcup \cdots \sqcup \widehat{\sS}_{N, y}^{\Omega}\sqcup \widehat{\nN}_y^{l\sss}. 
	\end{align}
Here $\widehat{\nN}_y^{l\sss}$ is the semistable locus with 
respect to the $G_y$-character 
$\lL|_{y}$.

For a one parameter subgroup $\lambda \colon \C \to G_y$, 
let $\widehat{w}_y^{\lambda \ge 0}$, $\widehat{w}_y^{\lambda=0}$
be the following functions 
determined by $\kappa_y^{\lambda \ge 0}$, $\kappa_y^{\lambda=0}$
as in (\ref{func:w})
\begin{align*}
&\widehat{w}_y^{\lambda \ge 0} \colon 
\left[\left( \widehat{\hH}^0(\mathbb{T}_{\fM}|_{y})^{\lambda \ge 0} \oplus 
\left(\hH^1(\mathbb{T}_{\fM}|_{y})^{\lambda \ge 0}\right)^{\vee}\right) /G_y^{\lambda \ge 0} \right]
\to \mathbb{C}, \ 
\widehat{w}_y^{\lambda \ge 0}(-, -)=\langle 
\kappa_y^{\lambda \ge 0}(-), - \rangle, \\
&\widehat{w}_y^{\lambda = 0} \colon 
\left[\left(\widehat{\hH}^0(\mathbb{T}_{\fM}|_{y})^{\lambda=0} \oplus 
\left(\hH^1(\mathbb{T}_{\fM}|_{y})^{\lambda=0}\right)^{\vee}
\right) /G_y^{\lambda =0} \right]
\to \mathbb{C}, \ 
\widehat{w}_y^{\lambda =0}(-, -)=\langle \kappa_y^{\lambda=0}(-), - \rangle.	
	\end{align*}
We have the following commutative diagram 
\begin{align*}
	\xymatrix{
	 \left[\left( \widehat{\hH}^0(\mathbb{T}_{\fM}|_{y})^{\lambda \ge 0} \oplus 
	(\hH^1(\mathbb{T}_{\fM}|_{y})^{\vee})^{\lambda \ge 0}\right) /G_y^{\lambda \ge 0} \right]	\ar[r]^-{q_2} \ar@{=}[d] 
	&   \left[\left( \widehat{\hH}^0(\mathbb{T}_{\fM}|_{y}) \oplus 
		\hH^1(\mathbb{T}_{\fM}|_{y})^{\vee}\right) /G_y \right]
		\ar@/^130pt/[ddd]_-{\widehat{w}_y} \\		
		\left[\left( \widehat{\hH}^0(\mathbb{T}_{\fM}|_{y})^{\lambda \ge 0} \oplus 
		\left(\hH^1(\mathbb{T}_{\fM}|_{y})^{\lambda \le 0}\right)^{\vee}\right) /G_y^{\lambda \ge 0} \right]	\ar@/_110pt/[dd]_-{q_1} 
		\ar[r]_-{r_1} \ar[d]_-{r_2} \diasquare & 
		\left[\left( \widehat{\hH}^0(\mathbb{T}_{\fM}|_{y})^{\lambda \ge 0} \oplus 
		\hH^1(\mathbb{T}_{\fM}|_{y})^{\vee}\right) /G_y^{\lambda \ge 0} \right] \ar[u]_-{f_2} \ar[d]^-{g_2}  \\
		\left[\left( \widehat{\hH}^0(\mathbb{T}_{\fM}|_{y})^{\lambda \ge 0} \oplus 
		\left(\hH^1(\mathbb{T}_{\fM}|_{y})^{\lambda=0}\right)^{\vee}\right) /G_y^{\lambda \ge 0} \right]	
		\ar[r]_-{g_1} \ar[d]_-{f_1} & \left[\left( \widehat{\hH}^0(\mathbb{T}_{\fM}|_{y})^{\lambda \ge 0} \oplus 
		\left(\hH^1(\mathbb{T}_{\fM}|_{y})^{\lambda \ge 0}\right)^{\vee}\right) /G_y^{\lambda \ge 0} \right] \ar[d]^-{\widehat{w}_y^{\lambda \ge 0}} \\
		\left[\left( \widehat{\hH}^0(\mathbb{T}_{\fM}|_{y})^{\lambda=0} \oplus 
		\left(\hH^1(\mathbb{T}_{\fM}|_{y})^{\lambda=0}\right)^{\vee}\right) /
		G_y^{\lambda=0} \right]
		\ar[r]_-{\widehat{w}_y^{\lambda=0}} 
		& \mathbb{C}. 
	} 
	\end{align*}
Here $f_1$, $r_2$, $g_2$ are given by projections onto the 
corresponding weight spaces, and 
$g_1$, $r_1$, $f_2$, $q_2$ are given by inclusions 
from the corresponding weight spaces. 
In order to simplify the notation, we write the above diagram 
as 
\begin{align}\label{dia:comX}
	\xymatrix{	
	\xX_6	\ar@/_20pt/[dd]_-{q_1} 
	\ar[r]_-{r_1} \ar[d]_-{r_2} \ar@/^20pt/[rr]^-{q_2}\diasquare & 
	\xX_2 \ar[r]_-{f_2} \ar[d]^-{g_2} & \xX_1 \ar[dd]^-{w_1}  \\
	\xX_4	
	\ar[r]_-{g_1} \ar[d]_-{f_1} & \xX_3 
	\ar[dr]^-{w_3} & \\
\xX_5
	\ar[rr]_-{w_5}  & 
	& \mathbb{C},
} 
\end{align}
where each $(\xX_i, w_i)$ is 
\begin{align*}
	w_1 \colon 
\xX_1 &=\left[\left( \widehat{\hH}^0(\mathbb{T}_{\fM}|_{y}) \oplus 
\hH^1(\mathbb{T}_{\fM}|_{y})^{\vee}\right) /G_y \right]
\stackrel{\widehat{w}_y}{\to} \mathbb{C}, \\
w_2 \colon \xX_2 &=\left[\left( 
\widehat{\hH}^0(\mathbb{T}_{\fM}|_{y})^{\lambda \ge 0} \oplus 
\hH^1(\mathbb{T}_{\fM}|_{y})^{\vee}\right) /G_y^{\lambda \ge 0} \right]\stackrel{\widehat{w}_y f_2}{\to} \mathbb{C}, \\
w_3 \colon \xX_3 &=\left[\left( \widehat{\hH}^0(\mathbb{T}_{\fM}|_{y})^{\lambda \ge 0} \oplus 
\left(\hH^1(\mathbb{T}_{\fM}|_{y})^{\lambda \ge 0}\right)^{\vee}\right) /
G_y^{\lambda \ge 0} \right]\stackrel{\widehat{w}_y^{\lambda \ge 0}}{\to} \mathbb{C}, \\
w_4 \colon \xX_4 &=\left[\left( \widehat{\hH}^0(\mathbb{T}_{\fM}|_{y})^{\lambda \ge 0} \oplus 
\left(\hH^1(\mathbb{T}_{\fM}|_{y})^{\lambda=0}\right)^{\vee}\right) /G_y^{\lambda \ge 0} \right]\stackrel{\widehat{w}_y^{\lambda \ge 0}g_1}{\to} \mathbb{C}, \\
w_5 \colon \xX_5 &=\left[\left( \widehat{\hH}^0(\mathbb{T}_{\fM}|_{y})^{\lambda=0} \oplus 
\left(\hH^1(\mathbb{T}_{\fM}|_{y})^{\lambda=0}\right)^{\vee}\right) /
G_y^{\lambda=0} \right]\stackrel{\widehat{w}_y^{\lambda =0}}{\to} \mathbb{C}, \\
w_6 \colon \xX_6 &= \left[\left( \widehat{\hH}^0(\mathbb{T}_{\fM}|_{y})^{\lambda \ge 0} \oplus 
(\hH^1(\mathbb{T}_{\fM}|_{y})^{\vee})^{\lambda \ge 0}\right) /G_y^{\lambda \ge 0} \right]\stackrel{\widehat{w}_y q_2}{\to} \mathbb{C}.
\end{align*}
We note that, 
by the Koszul duality equivalence in Theorem~\ref{thm:knoer}, 
we have equivalences 
\begin{align}\label{Koszul:formal}
	&\widehat{\Phi}_y \colon 
	\Dbc([\widehat{\fU}_y/G_y]) \stackrel{\sim}{\to}
	\MF_{\coh}^{\C}(\xX_1, w_1), \\
	\notag	&\widehat{\Phi}_y^{\lambda \ge 0} \colon 
	\Dbc([\widehat{\fU}_y^{\lambda \ge 0}/
	G_y^{\lambda \ge 0}]) \stackrel{\sim}{\to}
	\MF_{\coh}^{\C}(\xX_3, w_3), \\
	\notag	&\widehat{\Phi}_y^{\lambda = 0} \colon 
	\Dbc([\widehat{\fU}_y^{\lambda = 0}/
	G_y^{\lambda = 0}]) \stackrel{\sim}{\to}
	\MF_{\coh}^{\C}(\xX_5, w_5). 
\end{align}
The diagram (\ref{dia:comX})
induces the diagram 
 	\begin{align}\label{diagram:Crity}
 		\xymatrix{
 		q_1^{-1}(\Crit(w_5))
 	\cap (g_2 r_1)^{-1}(\Crit(w_3))
 	\cap q_2^{-1}(\Crit(w_1)) \ar[r]^-{q_2} \ar[d]^-{q_1}	 & 
 \Crit(w_1) \\
 	\Crit(w_5)\ar@/^18pt/[u]^-{q_3} \ar[ur]_-{q_4}. & 
 	}
 	 \end{align} 
  Here $q_3$ is induced by the zero section of 
  $q_1 \colon \xX_6 \to \xX_5$, and $q_4 \cneq q_2 \circ q_3$. 
  
  Let $\lambda=\lambda_i$ be the one parameter 
  subgroup of $G_y$ corresponding to the $i$-th 
  $\Theta$-strata as in (\ref{lambda:i}).
  Then from the diagrams (\ref{dia:Somega}), (\ref{diagram:fN}) together with (\ref{identity:crit}), 
  one can show that (see the proof of~\cite[Proposition~3.2.4]{TocatDT})
  the diagram (\ref{diagram:Crity}) is isomorphic to the following diagram
  in (\ref{Somega})
\begin{align}\label{dia:SZNhat}
	\xymatrix{
\overline{\sS}_{i, y}^{\Omega} \ar[r]^-{\ev_1^{\Omega}}
\ar[d]^-{\ev_0^{\Omega}} & \widehat{\nN}_y \\
\overline{\zZ}_{i, y}^{\Omega}.  \ar@/^18pt/[u]^-{\sigma^{\Omega}} \ar[ru]_-{\tau^{\Omega}}
}
	\end{align}

\subsection{Adjoint functors of $\Upsilon_i$ and $\Upsilon_{i, j}$}
Let us take the ind-completion of the functor 
(\ref{funct:ev})
\begin{align}\label{Up:ind}
	\Upsilon_i^{\ind} \colon \Ind \dDT^{\C}(\zZ_i^{\Omega}) 
	\to \Ind \dDT^{\C}(\nN \setminus \sS_{\le i-1}^{\Omega}). 	
	\end{align}
We see that the above functor admits a right adjoint, explicitly described 
in terms of the diagram (\ref{dia:Filt4}) as follows: 
\begin{lem}\label{lem:Upadj}
	The functor (\ref{Up:ind}) admits a right adjoint functor of the form
	\begin{align}\label{Up:indr}
		\Upsilon_i^{R}=\ev_{0\ast}^{\ind}\ev_1^! \colon 
		\Ind \dDT^{\C}(\nN \setminus \sS_{\le i-1}^{\Omega})
		\to \Ind \dDT^{\C}(\zZ_i^{\Omega}). 	
	\end{align}
Here $\ev_0$, $\ev_1$ are given in the diagram (\ref{dia:Filt4}). 
	\end{lem}
\begin{proof}
	Let us take the ind-completion of the functor (\ref{funct:ev})
	\begin{align}\label{Up:ind2}
		\Upsilon_i^{\ind} \colon \Ind \Dbc(\ffZ_i) \stackrel{\ev_0^{\ast}}{\to} \Ind \Dbc(\ffS_i) 
		\stackrel{\ev_{1\ast}^{\ind}}{\to} \Ind \Dbc(\fM). 
		\end{align}
	By Proposition~\ref{prop:DTcat},
	the functor (\ref{Up:ind})
	is a descendant of the functor (\ref{Up:ind2}).
	It admits a right adjoint 
	\begin{align}\label{radj1}
		\ev_{0\ast}^{\ind}\ev_1^! \colon 
		\Ind \Dbc(\fM) \stackrel{\ev_1^!}{\to}
		\Ind \Dbc(\ffS_i) \stackrel{\ev_{0\ast}^{\ind}}{\to} \Ind \Dbc(\ffZ_i). 
		\end{align}  
	Here see~\cite[Section~3.6, 3.7.7]{MR3037900}, \cite[Proposition~3.16]{MR3701352}
	for the above adjoint functors for ind-coherent sheaves. 
	We show that the above functor (\ref{radj1}) restricts to the functor 
	\begin{align}\label{radj2}
		\ev_{0\ast}^{\ind} \ev_1^! \colon \Ind 
		\cC_{\sS_{\le i-1}^{\Omega}} \to 
		\Ind \cC_{T_i^{\Omega}}. 
		\end{align}
	Since singular supports have point-wise characterizations (see~\cite[Proposition~6.2.2]{MR3300415}), 
	it is enough to prove the above claim formally locally on 
	the good moduli space $M$, i.e. for any closed point $y \in M$
	and 
	the diagram (\ref{dia:lambda}), it is enough to show that
	the functor 
	\begin{align}\label{upind:loc}
		\ev_{0\ast}^{\ind} \ev_1^! \colon 
		\Ind \Dbc([\widehat{\fU}_y/G_y]) \stackrel{\ev_1^!}{\to}
		\Ind \Dbc([\widehat{\fU}_y^{\lambda \ge 0}/G_y^{\lambda \ge 0}]) 
		\stackrel{\ev_{0\ast}^{\ind}}{\to} \Ind \Dbc([\widehat{\fU}_y^{\lambda=0}/G_y^{\lambda=0}]) 
		\end{align}
	restricts to the functor 
	\begin{align}\label{upind:loc2}
		\ev_{0\ast}^{\ind} \ev_1^! \colon 
		\Ind \cC_{\widehat{\sS}_{\le i-1, y}^{\Omega}} \to 
		\Ind \cC_{\widehat{T}_{i, y}^{\Omega}}. 
	\end{align}
Here $\widehat{T}_{i, y}^{\Omega} \cneq \overline{\zZ}_{i, y}^{\Omega} \setminus 
\widehat{\zZ}_{i, y}^{\Omega}$ in the diagram (\ref{Somega}). 
By the identity (\ref{ev:TSomega}), 
we also have the following identity in the diagram (\ref{dia:SZNhat})
\begin{align}\label{ev0:Omega}
	(\ev_0^{\Omega})^{-1}(\widehat{T}_{i, y})
	=(\ev_1^{\Omega})^{-1}(\widehat{\sS}_{\le i-1, y}^{\Omega})  \subset \overline{\sS}_{i, y}^{\Omega}.
	\end{align}
By the above identity together with the fact that the Cartesian square in (\ref{dia:comX})
is a derived Cartesian and $f_2$ is proper, 
we can apply~\cite[Proposition~3.2.14]{TocatDT}
to conclude that the functor (\ref{upind:loc})
restricts to the functor (\ref{upind:loc2}). 

By Proposition~\ref{prop:DTcat} and Lemma~\ref{lem:adquot} below, the desired right adjoint 
(\ref{Up:indr})
is obtained by taking the quotients of both sides in (\ref{radj1}) 
by the subcategories in (\ref{radj2}). 	
	\end{proof}
We have used the following lemma, whose proof is straightforward. 
\begin{lem}\emph{(\cite[Lemma~1.1]{Orsin})}\label{lem:adquot}
	Let $\dD, \dD'$ be triangulated categories and 
	$\cC \subset \dD$, $\cC' \subset \dD'$ 
	be full triangulated subcategories. 
	Let $F \colon \dD \to \dD'$, $G \colon \dD' \to \dD$ be an adjoint
	pair of exact functors such that $F(\cC) \subset \cC'$, 
	$G(\cC') \subset \cC$. 
	Then they induce functors
	\begin{align*}
		F \colon \dD/\cC \to \dD'/\cC', \ 
		G \colon \dD'/\cC' \to \dD/\cC
		\end{align*}
	which are adjoints as well. 
	\end{lem}

Let $\Upsilon_{i, j}^R$ be the functor 
obtained by composing 
$\Upsilon_i^R$ in Lemma~\ref{lem:Upadj} with the projection onto the 
weight $j$ part
\begin{align}\label{Up:ijR}
	\Upsilon_{i, j}^R \colon 
\Ind \dDT^{\C}(\nN \setminus \sS_{\le i-1}^{\Omega}) \stackrel{\Upsilon_{i}^R}{\to}
\Ind \dDT^{\C}(\zZ_i^{\Omega}) \stackrel{\mathrm{pr}_j}{\to} \Ind \dDT^{\C}(\zZ_i^{\Omega})_{\wt=j}. 
	\end{align}
\begin{prop}\label{prop:radj}
	The functor (\ref{Up:ijR})
	restricts to the functor 
	\begin{align}\notag
		\Upsilon_{i, j}^R \colon 
		\dDT^{\C}(\nN \setminus \sS_{\le i-1}^{\Omega}) \to \dDT^{\C}(\zZ_i^{\Omega})_{\wt=j},
		\end{align}
	giving a right adjoint of the functor (\ref{Ups:ij}). 
	\end{prop}
\begin{proof}
	It is enough to show that the composition 
	\begin{align*}
		\Ind \Dbc(\fM) \stackrel{\ev_1^!}{\to} \Ind \Dbc(\ffS_i) \stackrel{\ev_{0\ast}^{\ind}}{\to}
		\Ind \Dbc(\ffZ_i) \stackrel{\mathrm{pr}_j}{\to} \Ind \Dbc(\ffZ_i)_{\wt=j}
		\end{align*}
		in the diagram (\ref{dia:Filt4}) sends $\Dbc(\fM)$ to $\Dbc(\ffZ_i)_{\wt=j}$. 
		For $\eE \in \Dbc(\fM)$, we have 
		\begin{align}\label{prj:ev}
			\mathrm{pr}_j \ev_{0\ast}^{\ind}\ev_1^!(\eE) \in 
			\Ind \Dbc(\ffZ_i)_{\wt=j}^+ 
			\subset \Ind \Dbc(\ffZ_i). 
			\end{align}
		Here
		the subscript $+$ indicates 
		bounded below subcategory with respect to the natural 
		t-structure on 
		ind-coherent sheaves, 
		and the fact that $\ev_{0\ast}^{\ind}$, $\ev_1^!$ preserve 
		these subcategories is proved in~\cite[Lemma~3.4.4]{MR3136100}. 
		By~\cite[Proposition~1.2.4]{MR3136100}, we have the equivalence 
		\begin{align*}
			\Ind \Dbc(\ffZ_i)_{\wt=j}^{+}			
			\stackrel{\sim}{\to}
			D_{\qcoh}(\ffZ_i)_{\wt=j}^{+}.
			\end{align*}
		Therefore it is enough to show that the object (\ref{prj:ev})
		is cohomological bounded above and has coherent cohomologies. 
		Since this is a local property, it is enough to check this
		formally locally on $M$, i.e. it is enough to show that for each 
		closed $y\in M$ we have 
		\begin{align*}
			\widehat{i}_y^{\ast}	\mathrm{pr}_j \ev_{0\ast}^{\ind}\ev_1^!(\eE)
			\in \Dbc(\widehat{\ffZ}_{i, y})_{\wt=j}. 
			\end{align*}
		Here $\widehat{i}_y$ is given in the diagram (\ref{formal:y2}). 
	By the base change 
		properties for functors of ind-coherent sheaves (see~\cite[Lemma~3.6.9, Proposition~7.1.6]{MR3136100}), 
		it is enough to show that the composition functor 
		from the diagram (\ref{dia:lambda})
		\begin{align}\notag
		\Ind \Dbc([\widehat{\fU}_y/G_y]) &\stackrel{\ev_1^!}{\to}
		\Ind \Dbc([\widehat{\fU}_y^{\lambda \ge 0}/G_y^{\lambda \ge 0}]) \\
	\notag	&\stackrel{\ev_{0\ast}^{\ind}}{\to} \Ind \Dbc([\widehat{\fU}_y^{\lambda=0}/G_y^{\lambda=0}])
		\stackrel{\mathrm{pr}_j}{\to} \Ind \Dbc([\widehat{\fU}_y^{\lambda=0}/G_y^{\lambda=0}])_{\wt=j}
			\end{align}
		restricts to the functor 
		\begin{align}\notag
			\mathrm{pr}_j \ev_{0\ast}^{\ind}\ev_1^! \colon 
			\Dbc([\widehat{\fU}_y/G_y]) \to \Dbc([\widehat{\fU}_y^{\lambda=0}/G_y^{\lambda=0}])_{\wt=j}.
			\end{align}
				Below we use the notation of the diagram (\ref{dia:comX}). 
		By the equivalences (\ref{Koszul:formal})
		and using (\ref{dia:com:ind}) and (\ref{dia:com:ind2}),
		we are reduced to showing that the composition
		of functors 
		\begin{align*}
			\MF_{\qcoh}^{\C}(\xX_1, w_1) &\stackrel{f_2^!}{\to} \MF_{\qcoh}^{\C}(\xX_2, w_2) \stackrel{g_{2\ast}}{\to}	
			\MF_{\qcoh}^{\C}(\xX_3, w_3) 
			\stackrel{g_1^{\ast}}{\to} \MF_{\qcoh}^{\C}(\xX_4, w_4) \\
			&\stackrel{f_{1\ast}}{\to} 
			\MF_{\qcoh}^{\C}(\xX_5, w_5) \stackrel{\mathrm{pr}_j}{\to}
			\MF_{\qcoh}^{\C}(\xX_5, w_5)_{\lambda \mathchar`- \wt= j}	
		\end{align*}
		restricts to the functor
		\begin{align*}
			\MF_{\coh}^{\C}(\xX_1, w_1) \to \MF_{\coh}^{\C}(\xX_5, w_5)_{\lambda \mathchar`- \wt= j}. 
		\end{align*}
		By the derived base change, the above composition functor is equivalent to the 
		following composition 
		\begin{align*}
			\MF_{\qcoh}^{\C}(\xX_1, w_1) &\stackrel{f_2^!}{\to} \MF_{\qcoh}^{\C}(\xX_2, w_2) \stackrel{r_1^{\ast}}{\to}	
			\MF_{\qcoh}^{\C}(\xX_6, w_6) \\
			&\stackrel{q_{1\ast}}{\to} 
			\MF_{\qcoh}^{\C}(\xX_5, w_5)
		\stackrel{\mathrm{pr}_j}{\to}
			\MF_{\qcoh}^{\C}(\xX_5, w_5)_{\lambda \mathchar`- \wt= j}. 	
		\end{align*}
		Since $f_2 \colon \xX_2 \to \xX_1$ 
		is a representable morphism of smooth stacks, 
		it is quasi-smooth and $f_2^!$ is given by 
		$f_2^!(-)=f_2^{\ast}(-) \otimes \omega_{f_2}$. 
		Therefore $r_1^{\ast}f_2^!$
		gives the functor 
		\begin{align*}
			r_1^{\ast} f_2^! \colon 
			\MF_{\coh}^{\C}(\xX_1, w) \to \MF_{\coh}^{\C}(\xX_6, w_6). 
		\end{align*}
		It is enough to show that the functor $\mathrm{pr}_j q_{1\ast}$
		gives the functor 
		\begin{align}\label{funct:pr}
			\mathrm{pr}_j q_{1\ast} \colon 
			\MF_{\coh}^{\C}(\xX_6, w_6) \to 
			\MF_{\coh}^{\C}(\xX_5, w_5)_{\lambda \mathchar`- \wt= j}. 
		\end{align}
		The morphism $q_1$ factors as 
		\begin{align*}
			q_1 \colon \xX_6 \stackrel{q_1'}{\to} \xX_7 \cneq 
			[A /
			G_y^{\lambda\ge 0}]
			\stackrel{q_1''}{\to} \xX_5
		\end{align*}
	where $A \cneq \widehat{\hH}^0(\mathbb{T}_{\fM}|_{y})^{\lambda=0} \oplus 
		\left(\hH^1(\mathbb{T}_{\fM}|_{y})^{\lambda=0}\right)^{\vee}$
		and $G_y^{\lambda \ge 0}$ acts on it 
		through the projection $G^{\lambda \ge 0} \twoheadrightarrow G^{\lambda=0}$.
		Since $\hH^0(\mathbb{T}_{\fM}|_{y})^{\lambda > 0} \oplus 
		(\hH^1(\mathbb{T}_{\fM}|_{y})^{\vee})^{\lambda > 0}$
		has positive $\lambda$-weights, 
		by~\cite[Lemma~2.2.3]{TocatDT}
				the push-forward $q_{1\ast}'$ restricts to the functor 
		\begin{align*}
			q_{1\ast}' \colon 
			\MF_{\coh}^{\C}(\xX_6, w_6) \to 
			\MF_{\coh}^{\C}(\xX_7, w_7)_{\lambda \mathchar`- \rm{above}},
		\end{align*}
		where $w_7=q_1''^{\ast}w_5$. 
		Let $G_y^{\lambda >0}$ be the kernel of the projection
		$G_y^{\lambda \ge 0} \twoheadrightarrow G_y^{\lambda=0}$. 
		Then $G_y^{\lambda>0}$ is unipotent, 
		so it admits a filtration of normal subgroups
		\begin{align*}
			0=G_0 \subset G_1 \subset \cdots \subset G_k =G_y^{\lambda>0}
			\end{align*}
		such that each subquotient 
		$G_i/G_{i-1}$ is isomorphic to the additive group 
		$\mathbb{A}^{m_i}$ for some $m_i$. 
		By setting $Q_i=G_y^{\lambda \ge 0}/G_i$, we have the factorizations 
		of $q_1''$
		\begin{align*}
		q_1'' \colon \xX_7=[A/Q_0] \to [A/Q_1] \to \cdots \to [A/Q_k]=\xX_5. 	
			\end{align*}
		Here each $Q_i$ acts on $A$ through the projection $Q_i \twoheadrightarrow G_y^{\lambda=0}$. 
		Since 
		$[A/Q_{i-1}] \to [A/Q_i]$ is a $\mathbb{A}^{m_i}$-gerbe, 
		the push-forward of a coherent sheaf along with
		the above morphism is quasi-isomorphic 
		to a bounded complex of coherent sheaves. 
		Therefore 
		the functor $q_{1\ast}''$ gives 
		\begin{align*}
			q_{1\ast}'' \colon 
			\MF_{\coh}^{\C}(\xX_7, w_7) \to 
			\MF_{\coh}^{\C}(\xX_5, w_5). 
		\end{align*}
		Therefore $q_{1\ast}=q_{1\ast}'' \circ q_{1\ast}'$ 
		restricts to the functor 
		\begin{align}\notag
			q_{1\ast} \colon 
			\MF_{\coh}^{\C}(\xX_6, w_6) \to \MF_{\coh}^{\C}(\xX_5, w_5)_{\lambda \mathchar`- \rm{above}},
		\end{align}
		which concludes that $\mathrm{pr}_j q_{1\ast}$ gives
		the functor 
		(\ref{funct:pr}).		 
	\end{proof}

We also have the left adjoints as follows: 
\begin{lem}\label{lem:ladjoint}
	The functor (\ref{Ups:ij}) admits a left adjoint 
	\begin{align}\label{Ups:L}
		\Upsilon_{i, j}^L \colon 
		\dDT^{\C}\left( \nN \setminus \sS_{\le i-1}^{\Omega}\right) \to 
		\dDT^{\C}\left(\zZ_{i}^{\Omega}\right)_{\wt=j}. 
	\end{align}	
\end{lem}
\begin{proof}
	In order to simplify the notation, we write
	$\fM_1=\fM$, $\fM_2=\ffS_i$ and $\fM_3=\ffZ_i$ in the diagram (\ref{dia:Filt4}). 
	We denote by $\mathbb{D}_i$ the Serre duality 
	equivalence for $\Dbc(\fM_i)$, 
	given by 
	\begin{align*}
		\mathbb{D}_i \cneq \hH om(-, \omega_{\fM_i}), \ \omega_{\fM_i} \cneq 
		\det \mathbb{L}_{\fM_i}[\rank \mathbb{L}_{\fM_i}]. 
	\end{align*}
	We have the following 
	adjoint pairs for the functors between 
	$\Ind \Dbc(\fM_1)$ and $\Ind \Dbc(\fM_3)_{\wt=j}$,
	\begin{align}\notag
		\mathrm{pr}_j 
		\mathbb{D}_{3}
		\ev_{0\ast}\ev_1^!
		\mathbb{D}_{1} \dashv 
		\mathbb{D}_1 \ev_{1\ast}\ev_0^{\ast} \mathbb{D}_3 i_j. 
	\end{align}
	Here the above functors are given by 
	\begin{align*}
		\xymatrix{
			\Ind \Dbc(\fM_1)	
			\ar@<0.5ex>[r]^-{\mathbb{D}_1} &
			\ar@<0.5ex>[l]^-{\mathbb{D}_1}
			\Ind \Dbc(\fM_1)^{\rm{op}} \ar@<0.5ex>[r]^-{\ev_1^!} & 
			\ar@<0.5ex>[l]^-{\ev_{1\ast}}
			\Ind\Dbc(\fM_2)^{\rm{op}}  \\ 
			\ar@<0.5ex>[r]^-{\ev_{0\ast}}&\ar@<0.5ex>[l]^-{\ev_0^{\ast}}
			\Ind\Dbc(\fM_3)^{\rm{op}} \ar@<0.5ex>[r]^-{\mathbb{D}_3} & 
			\ar@<0.5ex>[l]^-{\mathbb{D}_3}
		\Ind\Dbc(\fM_3) \ar@<0.5ex>[r]^-{\mathrm{pr}_j} & 
			\ar@<0.5ex>[l]^-{i_j}
			\Ind\Dbc(\fM_3)_{\wt=j}. 
		}
	\end{align*}
	Since $\ev_1$ is proper, 
	using~\cite[Corollary~9.5.9]{MR3136100}
	we have 
	\begin{align*}
		\mathbb{D}_1 \ev_{1\ast}\ev_0^{\ast} \mathbb{D}_3 i_j
		\cong 	
		\ev_{1\ast}\mathbb{D}_2\ev_0^{\ast} \mathbb{D}_3 i_j 
		\cong \ev_{1\ast}\ev_0^{!}i_j.
	\end{align*}
	By Lemma~\ref{lem:qsmooth}, 
	 $\ev_0$ is quasi-smooth. 
	 Also by Lemma~\ref{lem:detFilt}, 
	 the relative 
	dualizing complex 
	$\omega_{\ev_0}$ is of the form 
	$\ev_0^{\ast}\lL[k]$ for a line bundle 
	$\lL$ on $\fM_3$ and $k\in \mathbb{Z}$. 
	Therefore 
	we conclude that 
	\begin{align*}
		\mathbb{D}_1 \ev_{1\ast}\ev_0^{\ast} \mathbb{D}_3 i_j(-)
		\cong 	\ev_{1\ast}\ev_0^{\ast}(i_j(-)\otimes \lL)[k]. 
	\end{align*}
	The composition $\otimes \lL \circ i_j$ is isomorphic to 
	$i_{j+\wt(\lL)} \circ \otimes \lL$, where $\otimes \lL$ 
	is 
	the equivalence 
	\begin{align*}
		\otimes \lL \colon 
		\Ind\Dbc(\fM_3)_{\wt=j}
		\stackrel{\sim}{\to}
		\Ind\Dbc(\fM_3)_{\wt=j+\wt(\lL)}.
	\end{align*}
	Therefore we have the adjoint pair 
	for the functors between $\Ind\Dbc(\fM_1)$ and $\Ind\Dbc(\fM_3)_{\wt=j}$
	\begin{align}\label{ladjoint:ast}
		\otimes \lL \circ \mathrm{pr}_{j-\wt(\lL)}
		\mathbb{D}_{3}
		\ev_{0\ast}\ev_1^!
		\mathbb{D}_{1} \dashv 
		\ev_{1\ast}\ev_0^{\ast} i_j. 
	\end{align}
	Since the dualizing functors and tensor products with 
	line bundle preserve 
	singular supports and compact objects, 
	as in the proof of Lemma~\ref{lem:Upadj}
	the LHS of (\ref{ladjoint:ast}) induces the functor (\ref{Ups:L})
	giving a left adjoint of the functor (\ref{Ups:ij}). 
\end{proof}

\subsection{Semiorthogonal decompositions of DT categories}
Here we give a proof of Theorem~\ref{intro:main}. 
We first show that the functors $\Upsilon_{i, j}$ are fully-faithful. 
\begin{prop}\label{prop:ups:ff}
	The functor $\Upsilon_{i, j}$ in 
	(\ref{Ups:ij}) is fully-faithful. 	
\end{prop}
\begin{proof}
	By Proposition~\ref{prop:radj}, 
	it is enough to show that the natural transform
	\begin{align}\label{id:adjoint}
		\id \to \Upsilon_{i, j}^R \circ \Upsilon_{i, j}
	\end{align}
	is an isomorphism. 
	This is a local property, so it is enough to 
	check the isomorphism (\ref{id:adjoint})
	after pulling back via $\widehat{\iota}_y \colon 
	\widehat{\fM}_y \to \fM$
	at each closed point $y \in M$. 
	Let $\lambda_i \colon \C \to G_y$
	be the one parameter subgroup 
	corresponding to the $i$-th $\Theta$-strata 
	as in (\ref{lambda:i}), and set $\lambda=\lambda_i$. 
	By the base change, 
	we are reduced to showing that the functor 
	\begin{align*}
		\ev_{1\ast}\ev_0^{\ast} \colon 
		\Dbc([\widehat{\fU}_y^{\lambda=0}/G_y^{\lambda=0}])_{\wt=j}/\cC_{\widehat{T}_{i, y}^{\Omega}}
		\to \Dbc([\widehat{\fU}_y/G_y])/\cC_{\widehat{\sS}_{\le i-1, y}^{\Omega}}
	\end{align*}
	from the diagram (\ref{dia:lambda}) is fully-faithful. 
	We show that the above functor is fully-faithful 
	through the Koszul duality equivalences in (\ref{Koszul:formal}).
	
	Below we use the notation in the diagram (\ref{dia:comX}).  
	Let us consider the following composition functor
	\begin{align*}
		\MF_{\coh}^{\C}(\xX_5, w_5)_{\lambda \mathchar`- \wt= j} &\stackrel{f_1^{\ast}}{\to}
		\MF_{\coh}^{\C}(\xX_4, w_4) \stackrel{g_{1!}}{\to} \MF_{\coh}^{\C}(\xX_3, w_3) \\
		&\stackrel{g_2^{\ast}}{\to}
		\MF_{\coh}^{\C}(\xX_2, w_2) \stackrel{f_{2\ast}}{\to}\MF_{\coh}^{\C}(\xX_1, w_1).
	\end{align*}
	By (\ref{dia:com:ind0.5})
	and (\ref{dia:com:ind1.5}), 
	it is enough to show that the descendant of the 
	above composition functor 
	\begin{align}\label{ff:mf:loc}
		f_{2\ast}g_2^{\ast}	g_{1 !} f_1^{\ast} \colon 
		\MF_{\coh}^{\C}(\xX_5 \setminus 
		\widehat{T}_{i, y}^{\Omega}, w_5)_{\lambda \mathchar`- \wt= j}
		\to \MF_{\coh}^{\C}(\xX_1 \setminus 
		\widehat{\sS}_{\le i-1, y}^{\Omega}, w_1)
	\end{align}
	is fully-faithful. 
	Note that we have 
	\begin{align*}
		g_{1!}(-)=g_{1\ast}(- \otimes \det(\hH^1(\mathbb{T}_{\fM}|_{y})^{\lambda>0})^{\vee}[-\dim \hH^1(\mathbb{T}_{\fM}|_{y})^{\lambda>0}]). 
	\end{align*}
	Here 
	$\det(\hH^1(\mathbb{T}_{\fM}|_{y})^{\lambda>0})^{\vee}$
	is a $G_y^{\lambda \ge 0}$-character, 
	which descends to a 
	$G_y^{\lambda=0}$-character by the 
	projection $G_y^{\lambda \ge 0} \twoheadrightarrow G_y^{\lambda=0}$
	as $\C$ is commutative. 
	Therefore 
	it 
	is written as $f_1^{\ast}\det(\hH^1(\mathbb{T}_{\fM}|_{y})^{\lambda>0})^{\vee}$ where 
	we regard 
	$\det(\hH^1(\mathbb{T}_{\fM}|_{y})^{\lambda>0})^{\vee}$ as a
	$G_y^{\lambda=0}$-character. 
	It follows that we have 
	\begin{align*}
		f_{2\ast}g_2^{\ast}	g_{1 !} f_1^{\ast}(-)
		\cong q_{2\ast}q_1^{\ast}(- \otimes 
		\det(\hH^1(\mathbb{T}_{\fM}|_{y})^{\lambda>0})^{\vee}
		[- \dim \hH^1(\mathbb{T}_{\fM}|_{y})^{\lambda>0}]). 	
	\end{align*}
	Here the first functor is an equivalence 
	\begin{align}\label{otimes:O}
		\otimes \det(\hH^1(\mathbb{T}_{\fM}|_{y})^{\lambda>0})^{\vee} \colon 
		\MF_{\coh}^{\C}(\xX_5 & \setminus \widehat{T}_{i, y}^{\Omega}, w_5)_{\lambda \mathchar`- \wt= j} \\
		\notag
		&\stackrel{\sim}{\to}
		\MF_{\coh}^{\C}(\xX_5 \setminus \widehat{T}_{i, y}^{\Omega}, w_5)_{\lambda \mathchar`- \wt= j-\langle \lambda, \hH^1(\mathbb{T}_{\fM}|_{y})^{\lambda>0} \rangle}. 
	\end{align}
	By the identity (\ref{ev0:Omega}), we have 
	$		\widehat{T}_{i, y}^{\Omega}=(\tau^{\Omega})^{-1}
	(\widehat{\sS}_{\le i-1, y}^{\Omega})
	$
	in the diagram (\ref{dia:SZNhat}). 
	Therefore from the equivalence (\ref{equiv:bar}), 
	we conclude that 
	the functor 
	\begin{align}\label{funct:formal:loc}
		q_{2\ast}q^{\ast}_1 \colon 
		\MF_{\coh}^{\C}(\xX_5 \setminus \widehat{T}_{i, y}^{\Omega}, w_5)_{\lambda \mathchar`- \wt= j-\langle \lambda, \hH^1(\mathbb{T}_{\fM}|_{y})^{\lambda>0} \rangle}
		\to \MF_{\coh}^{\C}(\xX_1 \setminus \widehat{\sS}_{\le i-1, y}^{\Omega}, w_1)
	\end{align}
	is fully-faithful. Therefore the functor (\ref{ff:mf:loc}) is fully-faithful. 
\end{proof}

We denote by 
\begin{align}\notag
	\dD_{i, j} \subset \dDT^{\C}(\nN \setminus \sS_{\le i-1}^{\Omega})
\end{align}
the essential images of the functor 
$\Upsilon_{i, j}$ in (\ref{Ups:ij}), which is 
equivalent to $\dDT^{\C}(\zZ_i^{\Omega})_{\wt=j}$
by Proposition~\ref{prop:ups:ff}. 
We have the following semiorthogonality of these
subcategories: 
\begin{lem}\label{lem:orthognoal}
	For $j>j'$, we have 
	\begin{align*}
		\Hom(\dD_{i, j}, \dD_{i, j'})=0. 
	\end{align*}
\end{lem}
\begin{proof} 
	It is enough to show the vanishing 
	\begin{align*}
		\Upsilon_{i, j}^R \circ \Upsilon_{i, j}\cong 0, \ j>j'.
		\end{align*} 
	As in the proof of Proposition~\ref{prop:ups:ff}, it is enough to prove
	this formally locally on the good moduli space $M$. 
	Note that the LHS of (\ref{funct:formal:loc})
	have $\lambda$-weight $j-\langle \lambda, \hH^1(\mathbb{T}_{\fM}|_{y})^{\lambda>0} \rangle$. 
	Therefore by Theorem~\ref{thm:window}, 
	the essential images of the functors (\ref{funct:formal:loc}) are 
	semiorthogonal for $j>j'$, so the lemma follows. 
\end{proof}

For each $1\le i \le N$ and $m \in \mathbb{Z}$, we define 
\begin{align*}
	\wW_{i, m} \cneq  \bigcap_{j\ge m} \Ker(\Upsilon_{i, j}^R) \cap 
	\bigcap_{j< m} \Ker(\Upsilon_{i, j}^L)
	\subset \dDT^{\C}(\nN \setminus \sS_{\le i-1}^{\Omega}). 
	\end{align*}
The following is the main result in this section, which 
gives a window theorem for DT categories associated with 
$\Theta$-stratification on $(-1)$-shifted cotangents. 
\begin{thm}\label{thm:sod}
	There exists a semiorthogonal decomposition 
	\begin{align}\label{sod:DT}
		\dDT^{\C}(\nN \setminus \sS_{\le i-1}^{\Omega})=\langle \ldots, \dD_{i, m-2}, \dD_{i, m-1}, 
		\wW_{i, m}, \dD_{i, m}, \dD_{i, m+1}, \ldots   \rangle,  
		\end{align}
	such that the composition functor 
	\begin{align}\label{W:equiv}
			\wW_{i, m}
			\hookrightarrow \dDT^{\C}(\nN \setminus \sS_{\le i-1}^{\Omega})
			\twoheadrightarrow \dDT^{\C}(\nN \setminus \sS_{\le i}^{\Omega})
		\end{align}
	is an equivalence. 
	\end{thm}
\begin{proof}
	We first show the semiorthogonal decomposition (\ref{sod:DT}). 
	From Lemma~\ref{lem:orthognoal} and the definition of $\wW_{i, m}$, 
	the RHS of (\ref{sod:DT}) is semiorthogonal. It is enough 
	to show that the RHS generates the LHS. 
		For an object $\eE$ in the LHS, 
	the proof of Proposition~\ref{prop:radj} shows that 
	$\Upsilon_{i, j}^R(\eE)=0$ for $j \gg 0$ formally locally 
	at the good moduli space $p \in M$. 	
	Since $\mM \to M$ is universally closed, 
	there exists an open neighborhood of 
	$p \in M$ on which 
	$\Upsilon_{i, j}^R(\eE)=0$ for $j\gg 0$. 
	As $M$ is quasi-compact, 
	we conclude that $\Upsilon_{i, j}^R(\eE)=0$ for $j\gg 0$. 
	A similar argument shows that 
	$\Upsilon_{i, j}^L(\eE)=0$ for $j \ll 0$. 
	
	Suppose that $\eE$ is not an object in 
	$\wW_{i, m}$. 
	Then there is 
	$j_1 \ge m$ such that 
	$\Upsilon_{i, j}^R(\eE)=0$ for $j>j_1$ and 
	$\Upsilon_{i, j_1}^R(\eE) \neq 0$, 
	or 
	there is $j_1' < m$
	such that $\Upsilon_{i, j}^L(\eE)=0$ for $j<j_1'$ and 
	$\Upsilon_{i, j_1'}^L(\eE)\neq 0$. 
	Below we assume the former case. 
	The latter case is similarly discussed. 
	We have the distinguished triangle 
	$\Upsilon_{i, j_1}\Upsilon_{i, j_1}^R(\eE)
	\to \eE \to \eE_1$, 
	where 
	$\Upsilon_{i, j}^R(\eE_1)=0$ for $j\ge j_1$. 
	Repeating the above constructions for $\eE_1$, 
	we have a distinguished triangle
	\begin{align*}
		\eE_2  \to \eE \to \eE_3, \ 
		\eE_2 \in \langle \dD_{i, m}, \ldots,  \dD_{i, j_1} \rangle, \ 
		\eE_3 \in \bigcap_{j\ge m} \Ker(\Upsilon_{i, j}^R). 
	\end{align*}
	If $\Upsilon_{i, j}^L(\eE_3)=0$ for all $j< m$, then 
	$\eE_3 \in \wW_{i, m}$. 
	Otherwise there is $j_1' < m$ such that 
	$\Upsilon_{i, j}^L(\eE_3) \neq 0$ for $j< j_1'$ and 
	$\Upsilon_{i, j_1'}^L(\eE_3) \neq 0$. 
	Similarly to above, 
	we have the distinguished triangle 
$\eE_4 \to \eE_3 \to \Upsilon_{i, j_1'} \Upsilon_{i, j_1'}^L(\eE)$
	such that  
	$\Upsilon_{i, j}^L(\eE_4)=0$ for $j\le j_1'$. 
	We also have $\Upsilon_{i, j}^R(\eE_4)=0$ for $j\ge m$
	by applying $\Upsilon_{i, j}^R$ to the above triangle
	and using Lemma~\ref{lem:orthognoal}. 
	By repeating the above construction for $\eE_4$, 
	we obtain the distinguished triangle 
	\begin{align*}
		\eE_5 \to \eE_3 \to \eE_6, \ 
		\eE_6 \in \langle 
		\dD_{i, j_1'}, \ldots, \dD_{i, m-1}\rangle, \ 
		\eE_5 \in \wW_{i, m}. 
	\end{align*}
	Therefore we obtain the desired semiorthogonal decomposition (\ref{sod:DT}). 
	
Below we show that the composition functor (\ref{W:equiv}) is an equivalence. 
By Lemma~\ref{lem:Tpush}, 
any object in $\dD_{i, j}$ have singular 
supports contained in $\sS_{\le i}^{\Omega}$. 
Therefore 
the composition 
\begin{align*}
	\dD_{i, j} \hookrightarrow \dDT^{\C}(\nN \setminus 
	\sS_{\le i-1}^{\Omega}) \to \dDT^{\C}(\nN \setminus 
	\sS_{\le i}^{\Omega}) 
	\end{align*}
is zero. It follows that by the semiorthogonal decomposition (\ref{sod:DT}) 
the functor (\ref{W:equiv}) is essentially surjective. 
It remains to show that (\ref{W:equiv}) is fully-faithful. 
Let $\iota \colon U \to M$ be an \'{e}tale morphism 
for an affine scheme $U$, and 
take the following Cartesian diagrams
\begin{align*}
	\xymatrix{
		\mathfrak{M}_U \ar[d]_-{\iota_{\fM}} \diasquare & \ar@<0.3ex>@{_{(}->}[l] \mM_U \ar[r]\ar[d] \diasquare & U \ar[d]_-{\iota} \\
		\fM & \ar@<0.3ex>@{_{(}->}[l]  \mM \ar[r] & M
	}
\end{align*}
such that $\iota_{\fM}$ is \'{e}tale. 
Here the above diagram exists since the $\infty$-category of \'{e}tale morphisms
with target $\mM$ is equivalent to that with target $\fM$. 
We
have the fully-faithful functor (see~\cite[Lemma~7.2.3]{TocatDT})
\begin{align*}
	\dDT^{\C}\left(\nN \setminus \sS_{\le i-1}^{\Omega}\right)
	\hookrightarrow 
	\lim_{U \stackrel{\iota}{\to} M}
	\left(\Dbc(\fM_U)/\cC_{\iota_{\fM}^{\ast}\sS_{\le i-1}^{\Omega}}\right). 
\end{align*}
Let $\eE_1, \eE_2$ be an object in $\wW_{i, m}$. 
By the above fully-faithful functor, it is enough to show 
that the natural morphism 
\begin{align}\label{nat:func}
	\Hom_{\Dbc(\fM_U)/\cC_{\iota_{\fM}^{\ast}\sS_{\le i-1}^{\Omega}}}(\iota_{\fM}^{\ast}\eE_1, 
	\iota_{\fM}^{\ast}\eE_2)
	\to
		\Hom_{\Dbc(\fM_U)/\cC_{\iota_{\fM}^{\ast}\sS_{\le i}^{\Omega}}}
		(\iota_{\fM}^{\ast}\eE_1, 
	\iota_{\fM}^{\ast}\eE_2)
\end{align}
is an isomorphism. 
The above morphism is regarded as a morphism in $D_{\qcoh}(U)$, 
so it is enough to show the above isomorphism 
formally locally at any point in $U$. 
Similarly to Proposition~\ref{prop:radj}, we prove 
the corresponding claim for derived categories of factorizations via 
Koszul duality. 

For each $y \in M$, 
we use the notation of the diagram (\ref{dia:comX}). 
Let $\widehat{\dD}_{i, j}$ be
the essential image of the functor (\ref{ff:mf:loc}). 
Since the functor (\ref{ff:mf:loc}) is the composition of (\ref{otimes:O}) and (\ref{funct:formal:loc}), 
by Theorem~\ref{thm:window} we have the semiorthogonal decomposition 
\begin{align}\notag
	\MF^{\C}_{\coh}(\xX_1 \setminus \widehat{\sS}_{\le i-1, y}^{\Omega}, w_1)
	=\langle \ldots, \widehat{\dD}_{i, m-2}, \widehat{\dD}_{i, m-1}, 
	\widehat{\wW}_{i, m}, \widehat{\dD}_{i, m}, \widehat{\dD}_{i, m+1}, 
	\ldots \rangle 
\end{align}
such that the composition functor 
\begin{align*}
	\widehat{\wW}_{i, m} \hookrightarrow 
		\MF^{\C}_{\coh}(\xX_1 \setminus \widehat{\sS}_{\le i-1, y}^{\Omega}, w_1)
		\to 
			\MF^{\C}_{\coh}(\xX_1 \setminus \widehat{\sS}_{\le i, y}^{\Omega}, w_1)
\end{align*}
is an equivalence. Since $\Upsilon_{i, j}^R$ and $\Upsilon_{i, j}^L$ 
commute with 
base change over $M$, 
for any object $\eE \in \wW_{i, m}$
we have $\widehat{\Phi}_y(\widehat{\iota}_y^{\ast}\eE) \in 
\widehat{\wW}_{i, m}$
where $\widehat{\iota}_y \colon \widehat{\fM}_y \to \fM$
is the morphism in (\ref{dia:fnbd}) and 
$\widehat{\Phi}_y$ is the equivalence
\begin{align*}
	\widehat{\Phi}_y \colon 
	\Dbc([\widehat{\fU}_y/G_y])/\cC_{\widehat{\sS}_{\le i-1, y}^{\Omega}}
	\stackrel{\sim}{\to}
	\MF^{\C}_{\coh}(\xX_1 \setminus \widehat{\sS}_{\le i-1, y}^{\Omega}, w_1)
\end{align*}
 given by (\ref{Psi:supp2}). 
It follows that the natural morphism 
\begin{align*}
	\Hom_{\Dbc(\widehat{\fM}_y)/\cC_{\widehat{\sS}_{\le i-1, y}^{\Omega}}}
	(\widehat{\iota}_y^{\ast}\eE_1, \widehat{\iota}_y^{\ast}\eE_2)
	\to \Hom_{\Dbc(\widehat{\fM}_y)/\cC_{\widehat{\sS}_{\le i, y}^{\Omega}}}
	(\widehat{\iota}_y^{\ast}\eE_1, \widehat{\iota}_y^{\ast}\eE_2)
	\end{align*}
is an isomorphism. 
Using Lemma~\ref{lem:fmloc} below, we conclude
the formal local isomorphism of (\ref{nat:func}). 
\end{proof}

\begin{lem}\label{lem:fmloc}
	In the setting of the diagram (\ref{dia:fnbd}),
	suppose that $M$ is affine. 
	Let $\zZ \subset t_0(\Omega_{\fM}[-1])$ be a conical 
	closed substack. 
	Then for any $\eE_1, \eE_2 \in \Dbc(\fM)$, 
	we have the isomorphism 
	\begin{align*}
		\Hom_{\Dbc(\fM)/\cC_{\zZ}}(\eE_1, \eE_2)
		\otimes_{\oO_M}\widehat{\oO}_{M, y}
		\stackrel{\cong}{\to}
		\Hom_{\Dbc(\widehat{\fM}_y)/\cC_{\widehat{\iota}_y^{\ast}\zZ}}
		(\widehat{\iota}_y^{\ast}\eE_1, \widehat{\iota}_y^{\ast}\eE_2). 
		\end{align*} 
	\end{lem}
\begin{proof}
	Noting Proposition~\ref{prop:DTcat} and Lemma~\ref{lem:adquot}, 
	we have the isomorphisms
	\begin{align*}
	\Hom_{\Dbc(\widehat{\fM}_y)/\cC_{\widehat{\iota}_y^{\ast}\zZ}}
	(\widehat{\iota}_y^{\ast}\eE_1, \widehat{\iota}_y^{\ast}\eE_2)
	&\cong 
	\Hom_{\Ind (\Dbc(\fM)/\cC_{\zZ})}
	(\eE_1, \widehat{\iota}_{y\ast}^{\ind}\widehat{\iota}_y^{\ast}\eE_2) \\
	&\cong 		\Hom_{\Ind (\Dbc(\fM)/\cC_{\zZ})}
	(\eE_1, \eE_2 \otimes_{\oO_M}\widehat{\oO}_{M, y}). 
		\end{align*}
	Here the second isomorphism is the projection formula 
	for ind-coherent shaves (see~\cite[Proposition~3.3.7]{MR3701352}), 
	where $\widehat{\oO}_{M, y} \in D_{\qcoh}(M)$
	acts on $\Ind \Dbc(\fM)$ by the tensor product. 
	Since $\eE_2$ is a compact object
	in $\Ind(\Dbc(\fM)/\cC_{\zZ})$
	(see the proof of~\cite[Proposition~3.2.7]{TocatDT}), 
	we have the isomorphism 
	\begin{align*}
		\Hom_{\Dbc(\fM)/\cC_{\zZ}}(\eE_1, \eE_2)
		\otimes_{\oO_M}\widehat{\oO}_{M, y}
		\stackrel{\cong}{\to}
		\Hom_{\Ind (\Dbc(\fM)/\cC_{\zZ})}
		(\eE_1, \eE_2 \otimes_{\oO_M}\widehat{\oO}_{M, y}). 
		\end{align*} 
	Therefore the lemma holds. 
		\end{proof}
	
	For an interval $I \subset \mathbb{R} \cup \{-\infty, \infty\}$, 
	we set
	\begin{align*}
		\dD_{i, I} \cneq \langle 
		\dD_{i, j} \colon j \in I \rangle
		\subset \dDT^{\C}(\nN \setminus \sS_{\le i-1}^{\Omega}). 
		\end{align*}
	Note that
	for each $m \in \mathbb{R}$
	 the semiorthogonal decomposition in Theorem~\ref{thm:sod}
implies that 
\begin{align}\notag
	\dDT^{\C}(\nN \setminus \sS_{\le i-1}^{\Omega})
	=\langle \dD_{i, < m}, \wW_{i, m}, \dD_{i, \ge m} \rangle, 
	\end{align}
where $\wW_{i, m} \cneq \wW_{i, \lceil m \rceil}$. 
	As a corollary of Theorem~\ref{thm:sod}, we have the following: 
	\begin{cor}\label{cor:sod}
		For each choice of $m_i \in \mathbb{R}$ for $1\le i\le N$, 
		there exists a semiorthogonal decomposition 
		\begin{align}\label{sod:ss}
			\dDT^{\C}(\nN)=
			\langle \dD_{1, < m_1},  \ldots, 
			\dD_{N, < m_N}, \wW_{m_{\bullet}}^{l}, 
			\dD_{N, \ge m_N}, \ldots, \dD_{1, \ge m_1}\rangle			
			\end{align}
		such that the composition functor 
		\begin{align*}
			\wW_{m_{\bullet}}^{l}
			\hookrightarrow \dDT^{\C}(\nN) \twoheadrightarrow
			\dDT^{\C}(\nN^{l\sss})
			\end{align*}
		is an equivalence. 
		\end{cor}
	
	\begin{rmk}\label{rmk:sod}
		In~\cite[Theorem~2.3.1]{HalpK32}, 
		Halpern-Leistner proves the existence of 
		semiorthogonal decomposition of $\Dbc(\fM)$, 
		associated with the $\Theta$-stratification for $\fM$
		\begin{align}\label{theta:M}
			\fM=\fS_1 \sqcup \cdots \sqcup \fS_N \sqcup \fM^{l\sss}, 
			\end{align}
		under the additional assumption 
		that each $\hH^{-1}(\mathbb{L}_{\fM}|_{\zZ_i})$ has 
		only non-negative weights. 
		Here $\zZ_i$ is the center of $\sS_i=t_0(\fS_i)$. 
		By Lemma~\ref{lem:weight}, the above weight assumption implies that 
		the $\Theta$-stratification (\ref{theta:M}) for $\fM$
		pulls back to the $\Theta$-stratification (\ref{N:theta}) for $\nN$. 
		In this case we have
		$\dDT^{\C}(\nN^{l\sss})=\Dbc(\fM^{l\sss})$, 
		 $\dDT^{\C}(\zZ_i^{\Omega})=\Dbc(\fZ_i)$
		and the semiorthogonal decomposition in Corollary~\ref{cor:sod}
		coincides with the one proved in~\cite[Theorem~2.3.1]{HalpK32}. 
		On the other hand, the proof of Corollary~\ref{cor:sod}
		is applied without the above weight assumption. Indeed 
		the semistable locus $\nN^{l\sss}$ may be strictly bigger 
		than the pull-back of $\mM^{l\sss}$, 
		and we need categorical DT theory to formulate 
		the semiorthogonal decomposition in Corollary~\ref{cor:sod}
		without the above weight assumption. 
		\end{rmk}

	From the proof of Theorem~\ref{thm:sod}, one can characterize 
	the subcategory $\wW_{m_{\bullet}}^l$ in (\ref{sod:ss}) 
	in terms of formal fibers 
	along with the good moduli space morphism
	$\pi_{\mM} \colon \mM \to M$. 
	Let us take a closed point $y \in M$
	and use the notation in the diagram (\ref{dia:comX}). 
	For each $1\le i\le N$, by taking 
	the one parameter subgroup $\lambda_i$ as in 
(\ref{lambda:i}), we set
	\begin{align}\label{etai}
		\eta_i &\cneq \langle \lambda_i, \det\mathbb{L}_{q_2}^{\vee} \rangle \\
	\notag	&=\langle \lambda_i, -\hH^0(\mathbb{T}_{\fM}|_{y})^{\lambda_i <0}
		+\hH^1(\mathbb{T}_{\fM}|_{y})^{\lambda_i>0}+
		\hH^{-1}(\mathbb{T}_{\fM}|_{y})^{\lambda_i<0} \rangle. 
		\end{align}
Here $q_2$ is a morphism in (\ref{dia:comX}). 
We also fix $\overline{\lL} \in \Pic(\fM)_{\mathbb{R}}$
with $\overline{l}=c_1(\overline{\lL})$
satisfying 
\begin{align}\label{irrat}
	\wt(\tau^{\ast}\overline{\lL}) \notin \mathbb{Q}
	\end{align}
for the diagram (\ref{dia:Filt4}) for all $1\le i\le N$, e.g. 
$\overline{l}=\varepsilon \cdot l$ for $\varepsilon \notin \mathbb{Q}$. 
We define the subcategory 
\begin{align*}
	\widehat{\wW}_y^l \subset \MF_{\coh}^{\C}(\xX_1, w_1)
	\end{align*}
to be consisting of factorizations $(\pP, d_{\pP})$
such that, 
for the inclusion 
\begin{align}\label{inclu:j}
	\tau \colon \xX_5 \setminus \widehat{T}_{i, y}^{\Omega} \hookrightarrow 
	\xX_1 \setminus \widehat{\sS}_{\le i-1, y}^{\Omega} 
	\end{align}
from the $\lambda_i$-fixed part, 
we have 
\begin{align}\label{cond:jP}
	\tau^{\ast}(\pP, d_{\pP}) \in 
	\bigoplus_{j\in \left[ -\frac{1}{2}\eta_i+\langle \lambda_i, \delta \rangle, 
		\frac{1}{2}\eta_i+\langle \lambda_i, \delta \rangle\right]}
	\MF_{\coh}^{\C}(\xX_5 \setminus \widehat{T}_{i, y}^{\Omega}, w_5)_{\lambda_i 
	\mathchar`- \wt =j}
		\end{align}
for all $1\le i\le N$.
Here $\delta$ is 
given by 
\begin{align}\label{Gy:delta}
	\delta \cneq \frac{1}{2}\det(\hH^1(\mathbb{T}_{\fM}|_{y}))^{\vee}+
	\overline{\lL}|_{y}
	\in K(BG_y)_{\mathbb{R}}, 
	\end{align}
which is independent of $i$. 
We also set 
\begin{align}\label{def:mi}
	m_i \cneq \frac{1}{2}\wt(\sigma^{\ast}\det \mathbb{L}_{\ev_1}^{\vee})
+\wt(\tau^{\ast}\overline{\lL})
	\end{align}
in the notation of the diagram (\ref{dia:Filt4}). 
\begin{prop}\label{prop:W:loc}
	For a choice of $m_i$ in (\ref{def:mi}), 
	an object $\eE \in \dDT^{\C}(\nN)$ lies in 
	$\wW_{m_{\bullet}}^l$ if and only if for any closed 
	point $y \in M$, for the morphism 
	$\widehat{\iota}_y \colon \widehat{\fM}_y \to \fM$ in (\ref{dia:fnbd})
	and an equivalence $\widehat{\Phi}_y$ in (\ref{Koszul:formal}), 
	we have 
	$\widehat{\Phi}_y(\widehat{\iota}_y^{\ast}\eE) \in 
		\widehat{\wW}_{y}^l. 
	$
	\end{prop}
\begin{proof}
	By Theorem~\ref{thm:window} 
	and the argument of Theorem~\ref{thm:sod}, 
	an object $\eE \in \dDT^{\C}(\nN)$ lies in 
	$\wW_{m_{\bullet}}^l$ if and only if for any closed 
	point $y \in M$ and $1\le i \le N$,
	for the morphism (\ref{inclu:j}) 
	and one parameter subgroup 
	$\lambda_i$ as in (\ref{lambda:i}), 
	we have 
	\begin{align}\label{Phiy:sss}
		\tau^{\ast}\widehat{\Phi}_y(\widehat{\iota}_y^{\ast}\eE) \in 
	\bigoplus_{j \in [m_i-\langle \lambda_i, \hH^1(\mathbb{T}_{\fM}|_{y})^{\lambda_i >0}\rangle, m_i-\langle \lambda_i, \hH^1(\mathbb{T}_{\fM}|_{y})^{\lambda_i >0}\rangle+\eta_i)}
	\MF_{\coh}^{\C}(\xX_5 \setminus \widehat{T}_{i, y}^{\Omega}, w_5)_{\lambda_i \mathchar`-\wt=j}. 
	\end{align}
Here the weight shift by $\langle \lambda_i, \hH^1(\mathbb{T}_{\fM}|_{y})^{\lambda_i >0}\rangle$
is due to the equivalence (\ref{otimes:O}). 
From the diagram (\ref{dia:lambda}), we have 
the distinguished triangle 
\begin{align*}
	\ev_1^{\ast}\mathbb{L}_{[\widehat{\fU}_y/G_y]}
	\to \mathbb{L}_{[\widehat{\fU}_y^{\lambda_i \ge 0}/G_y^{\lambda_i \ge 0}]}
	\to \mathbb{L}_{\ev_1}
	\end{align*}
which gives the identity 
\begin{align*}
\wt(\sigma^{\ast}\det\mathbb{L}_{\ev_1}^{\vee})
	=\langle\lambda_i, \hH^{0}(\mathbb{T}_{\fM}|_{y})^{\lambda_i <0}
	-\hH^{1}(\mathbb{T}_{\fM}|_{y})^{\lambda_i <0}
	-\hH^{-1}(\mathbb{T}_{\fM}|_{y})^{\lambda_i <0}
	 \rangle. 
	\end{align*}
It follows that, by (\ref{etai}) we have the identities
\begin{align*}
	&\frac{1}{2}\wt(\sigma^{\ast}\det\mathbb{L}_{\ev_1}^{\vee})
	+\wt(\tau^{\ast}\overline{\lL})
	-\langle \lambda_i, \hH^1(\mathbb{T}_{\fM}|_{y})^{\lambda_i>0}\rangle \\
	&=\frac{1}{2}\langle \lambda_i, \hH^{0}(\mathbb{T}_{\fM}|_{y})^{\lambda_i <0}
	-\hH^{1}(\mathbb{T}_{\fM}|_{y})^{\lambda_i >0}
	-\hH^{-1}(\mathbb{T}_{\fM}|_{y})^{\lambda_i <0}  \rangle
	+\langle \lambda_i, \overline{\lL}|_{y}\rangle
	-\frac{1}{2}\langle \lambda_i, \hH^1(\mathbb{T}_{\fM}|_{y})\rangle \\
	&=-\frac{1}{2}\eta_i+\left\langle \lambda_i, \frac{1}{2}\det(\hH^1(\mathbb{T}_{\fM}|_{y}))^{\vee}+
	\overline{\lL}|_{y} \right\rangle. 
	\end{align*}
Together with the condition (\ref{irrat}), 
we see that the condition (\ref{Phiy:sss}) is equivalent to the 
condition (\ref{cond:jP}) for $\pP=\widehat{\Phi}_y(\widehat{\iota}_y^{\ast}\eE)$.  
Therefore the proposition holds. 	
	\end{proof}

\subsection{Inclusions of window subcategories}
Let us take another $\lL' \in \Pic(\mM)$ 
with $l'=c_1(\lL')$, and set 
\begin{align*}
	l_{\pm} \cneq l\pm \varepsilon l', \ 
	0<\varepsilon \ll 1. 
	\end{align*}
	Then we have 
	$\nN^{l_{\pm}\sss} \subset \nN^{l\sss}$, 
	and the $\Theta$-stratification for $(l_{\pm}, b)$ is a 
	refinement of the $\Theta$-stratification (\ref{N:theta}) for $(l, b)$
\begin{align}\label{Theta:pm}
	\nN=\sS_1^{\Omega} \sqcup \cdots \sqcup \sS_{N}^{\Omega}
	\sqcup \sS_{N+1}^{\Omega \pm} \sqcup \cdots 
	\sqcup \sS_{N+k_{\pm}}^{\Omega \pm} \sqcup \nN^{l_{\pm} \sss},  
	\end{align}
where 
\begin{align}\label{Theta:pm2}
		\nN^{l\sss}=\sS_{N+1}^{\Omega \pm} \sqcup \cdots 
	\sqcup \sS_{N+k_{\pm}}^{\Omega \pm} \sqcup \nN^{l_{\pm} \sss}
	\end{align}
is the $\Theta$-stratification of $\nN^{l\sss}$ for 
$(l_{\pm}, b)$ restricted to $\nN^{l\sss}$. 

Let us consider the composition 
\begin{align*}
	\nN^{l\sss} \hookrightarrow \nN \stackrel{p_0}{\to}
	\mM \stackrel{\pi_{\mM}}{\to} M. 
	\end{align*}
We take a closed point $x \in \nN^{l\sss}$, 
and denote by $y \in M$ its image under 
the above composition. 
As before, we use the same symbol $y \in \mM$ 
to denote the unique closed point in the fiber of 
$\pi_{\mM} \colon \mM \to M$ at $y$. 
Note that $p_0(x) \in \mM$ may not be 
a closed point so that it may not 
be isomorphic to $y \in \mM$. 
In the notation of the diagram (\ref{dia:comX}),
the closed point $x\in \nN^{l\sss}$ corresponds
to a closed point $x \in \xX_1^{l\sss}$, where
$\xX_1^{l\sss}\subset \xX_1$ is the 
semistable locus with respect to the 
$G_y$-character 
$\lL|_{y}$. 
	Let $G_x \subset G_y$ be the stabilizer subgroup of $x$. 
	We define 
	$W_x$ to be the tangent space of the stack 
	$\xX_1^{l\sss}$ at $x$, i.e. 
	\begin{align}\label{rep:Wx}
		W_x \cneq \hH^0(\mathbb{T}_{\xX_1^{l\sss}}|_{x})
		\end{align}
	which is a $G_x$-representation. 
	Here $\mathbb{T}_{\xX_1^{l\sss}}$ is the 
	tangent complex of $\xX_1^{l\sss}$.
	\begin{rmk}
		The tangent complex of $\xX_1^{l\sss}$ is given by  
	\begin{align*}
		\mathbb{T}_{\xX_1^{l\sss}}=\left(
		\hH^{-1}(\mathbb{T}_{\fM}|_{y}) \otimes \oO_{\xX_1^{l\sss}}
		\stackrel{\mu}{\to} (	\hH^{0}(\mathbb{T}_{\fM}|_{y}) \oplus 
			\hH^{1}(\mathbb{T}_{\fM}|_{y})^{\vee})\otimes \oO_{\xX_1^{l\sss}}
			\right).
		\end{align*}
	The restriction of $\mathbb{T}_{\xX_1^{l\sss}}$
	to $x$ is the two term complex of $G_x$-representations 
	\begin{align*}
		\mathbb{T}_{\xX_1^{l\sss}}|_{x}=\left(
		\hH^{-1}(\mathbb{T}_{\fM}|_{y})
		\stackrel{\mu_x}{\to} \hH^{0}(\mathbb{T}_{\fM}|_{y}) \oplus 
		\hH^{1}(\mathbb{T}_{\fM}|_{y})^{\vee} \right).
		\end{align*} 
	The kernel of $\mu_x$ is the Lie algebra of $G_x$. 
	The $G_x$-representation (\ref{rep:Wx}) is given by the cokernel of 
	$\mu_x$. 
	\end{rmk}
	We impose the following assumption:
	\begin{assum}\label{assum:Wx}
		For each closed point $x \in \nN^{l\sss}$, 
		there is a decomposition of $G_x$-representations 
		$W_x=\bS_x \oplus \bU_x$ such that 
		$\bS_x$ is a symmetric $G_x$-representation, 
		i.e. $\bS_x \cong \bS_x^{\vee}$ as $G_x$-representations, and 
		\begin{align}\label{Wsss}
			W_x^{l_- \sss}=\bS_x^{l_-\sss} \oplus \bU_x. 
			\end{align}
		Here the $l_{-}$-semistable 
		loci on $G_x$-representations 
		are defined with respect to 
		the $G_x$-character 
		$(\lL \otimes (\lL')^{-\varepsilon})|_{x}$. 
				\end{assum}
	
	For $1\le i\le N$
	and $N<i \le N+k_{\pm}$, 
	we set $m_i$
	and $m_i^{\pm}$ 
	as in (\ref{def:mi}), 
	with a choice of $\overline{l}$ to be
	\begin{align}\label{l:bar}
		\overline{l}=\varepsilon_+ \cdot l_+ + \varepsilon_- \cdot l_-
		\end{align}
	for generic $(\varepsilon_+, \varepsilon_-) \in \mathbb{R}^2$. 
	We denote the above choice by
	$m_{\bullet}^{\pm}$, i.e. 
	\begin{align*}
		m_{\bullet}^{\pm}=\{m_1, \ldots, m_N, m_{N+1}^{\pm}, \ldots, m_{N+k_{\pm}}^{\pm}\}. 
		\end{align*} 
	By Corollary~\ref{cor:sod}, 
	we have subcategories
	\begin{align*}
		\wW_{m_{\bullet}^{\pm}}^{l_{\pm}} \subset
		\wW_{m_{\bullet}}^{l} \subset \dDT^{\C}(\nN)
		\end{align*}
	such that the compositions 
	\begin{align*}
		\wW_{m_{\bullet}^{\pm}}^{l_{\pm}} \hookrightarrow 
		\dDT^{\C}(\nN) \twoheadrightarrow \dDT^{\C}(\nN^{l_{\pm}\sss})
		\end{align*}
	are equivalences. 
	\begin{thm}\label{thm:inclu}
		Under Assumption~\ref{assum:Wx}, 
		we have the inclusion 
		$\wW_{m_{\bullet}^{-}}^{l_{-}} \subset 
		\wW_{m_{\bullet}^{+}}^{l_{+}}$. 
		In particular, we have the fully-faithful functor 
		\begin{align*}
			\dDT^{\C}(\nN^{l_{-}\sss})
			\hookrightarrow \dDT^{\C}(\nN^{l_{+}\sss}). 
			\end{align*}
		\end{thm}
	\begin{proof}
		For $i>N$, 
		we use the subscript $\pm$ to denote corresponding 
		objects associated with the strata $\sS_i^{\Omega\pm}$ in (\ref{Theta:pm}), 
		e.g. $\zZ_i^{\Omega \pm}$ for the center of $\sS_i^{\Omega \pm}$, 
		$\lambda_{i}^{\pm}$ for one parameter subgroups 
		in (\ref{lambda:i}) for the $i$-th $\Theta$-strata 
		with $N<i\le N+k_{\pm}$, 
		denote $\eta_i^{\pm}$ for (\ref{etai}), etc.
		By Proposition~\ref{prop:W:loc}, it is enough 
		to show the inclusion 
		$\widehat{\wW}_y^{l_-} \subset \widehat{\wW}_y^{l_+}$
		for any closed point $y \in M$. 
		We define 
		\begin{align}\notag
			\widehat{\wW}_y^{\mathbb{Z}/2, l_{\pm}} \subset
				\widehat{\wW}_y^{\mathbb{Z}/2, l} \subset 
			 \MF_{\coh}^{\mathbb{Z}/2}(\xX_1, w_1)
			\end{align} 
		to be the subcategories defined by the same condition (\ref{cond:jP})
		with respect to $l_{\pm}$ and $l$ for 
		the $\mathbb{Z}/2$-periodic triangulated categories 
		of factorizations.  
		Since the forgetting functor 
		$\MF_{\coh}^{\C}(\xX_1, w_1) \to \MF_{\coh}^{\mathbb{Z}/2}(\xX_1, w_1)$ 
		is conservative, 
		it is enough to show the inclusion 
		\begin{align}\label{inclu:Z2}
				\widehat{\wW}_y^{\mathbb{Z}/2, l_{-}}
				\subset 	\widehat{\wW}_y^{\mathbb{Z}/2, l_{+}}. 
			\end{align}
		
		Let $\xX_1^{l\sss} \to X_1^{l\sss}$ be the good moduli 
		space for $\xX_1^{l\sss}$. 
		For a closed point $x \in X_1^{l\sss}$, 
		we also denote by $x \in \xX_1^{l\sss}$ the unique 
		closed point contained in the fiber of 
		$\xX_1^{l\sss} \to X_1^{l\sss}$ at $x$. 
		Let $\widehat{W}_x$ be the formal fiber 
		of $W_x \to W_x \ssslash G_x$ at zero. 
		By Luna's \'etale slice theorem, we have the Cartesian square 
		\begin{align}\label{luna:W}
			\xymatrix{
		[\widehat{W}_x/G_x] \ar[r]^-{\iota_x} \ar[d] \diasquare & \xX_1^{l\sss} \ar[d] \\
		\widehat{W}_x \ssslash G_x \ar[r] & X_1^{l\sss},	
		}
			\end{align}
		which identifies the left vertical arrow 
		with the formal fiber of the right vertical arrow at $x$. 
		Let 
		\begin{align*}
			\iota_x^{\ast}w_1 \colon [\widehat{W}_x/G_x] \to \mathbb{C}
			\end{align*}
		be the pull-back of 
		$w_1 \colon \xX_1 \to \mathbb{C}$ by the top 
		morphism in (\ref{luna:W}). 
		Then its critical locus 
		is isomorphic to a formal fiber along with 
		the good moduli space morphism 
		$\nN^{l\sss} \to N^{l\sss}$. 
						The $\Theta$-stratification (\ref{Theta:pm2})
		is local on the good moduli space, so 
		it induces the $\Theta$-stratification 
		on $\Crit(\iota_x^{\ast}w_1)$
		\begin{align}\label{theta:W}
			\Crit(\iota_x^{\ast}w_1)
			=\widehat{\sS}_{N+1, x}^{\Omega\pm} \sqcup \cdots \sqcup 
			\widehat{\sS}_{N+k_{\pm}, x}^{\Omega\pm} \sqcup 
			\Crit(\iota_x^{\ast}w_1)^{l_{\pm}\sss}. 
			\end{align}
		
		We denote by 
		\begin{align*}
			\widehat{\wW}_x^{\mathbb{Z}/2, l_{\pm}} \subset 
			\MF_{\coh}^{\mathbb{Z}/2}([\widehat{W}_x/G_x], \iota_x^{\ast}w_1)
			\end{align*}
		the window subcategory in Subsection~\ref{subsec:crit:window} 
		associated with the $\Theta$-stratification (\ref{theta:W}), i.e. 
		the subcategory of objects 
	$\pP \in  \MF_{\coh}^{\mathbb{Z}/2}([\widehat{W}_x/G_x], \iota_x^{\ast}w_1)$
	such that for all $N<i\le N+k_{\pm}$ and the morphism 
	\begin{align*}
	\tau_{\pm} \colon 
	[\widehat{W}_x^{\lambda_i^{\pm}=0}/G_x^{\lambda_i^{\pm}=0}]
	\setminus \widehat{\sS}_{\le i-1, x}^{\Omega \pm} \to 
	[\widehat{W}_x/G_x]
	\setminus \widehat{\sS}_{\le i-1, x}^{\Omega \pm}
\end{align*}
of the inclusion from the $\lambda_i^{\pm}$-fixed part, 
we have 
\begin{align}\notag
	\tau^{\ast}_{\pm}\pP \in 
	\bigoplus_{j\in \left[ -\frac{1}{2}\eta_i^{\pm}+\langle \lambda_i^{\pm}, \delta \rangle, 
		\frac{1}{2}\eta_i^{\pm}+\langle \lambda_i^{\pm}, \delta \rangle\right]}
	\MF_{\coh}^{\C}([\widehat{W}_x^{\lambda_i^{\pm}=0}/G_x^{\lambda_i^{\pm}=0}]
	\setminus \widehat{\sS}_{\le i-1, x}^{\Omega \pm} , \iota_x^{\ast}w_1)_{\lambda_i^{\pm} 
		\mathchar`- \wt =j}. 
\end{align}
for all $N<i\le N+k_{\pm}$, where 
$\delta \in \Pic(BG_x)_{\mathbb{R}}$ is
the $G_y$-character (\ref{Gy:delta}) restricted to $G_x \subset G_y$
with the choice of $\overline{l}$ given by (\ref{l:bar}).  
From the equivalence of the composition 
\begin{align*}
		\widehat{\wW}_y^{\mathbb{Z}/2, l}
		\subset \MF_{\coh}^{\mathbb{Z}/2}(\xX_1, w_1)
		\twoheadrightarrow \MF_{\coh}^{\mathbb{Z}/2}(\xX_1^{l\sss}, w_1)
	\end{align*}
we can identify the subcategories 
$\widehat{\wW}_y^{\mathbb{Z}/2, l_{\pm}} \subset \widehat{\wW}_y^{\mathbb{Z}/2, l}$
with the subcategories in 
$\MF_{\coh}^{\mathbb{Z}/2}(\xX_1^{l\sss}, w_1)$
consisting of objects
satisfying the condition (\ref{cond:jP}) for all $N<i\le N+k_{\pm}$
with respect to $l_{\pm}$. 
Since the latter condition is local
on the good moduli space $\xX_1^{l\sss} \to X_1^{l\sss}$, 
an object $\eE \in \MF_{\coh}^{\mathbb{Z}/2}(\xX_1^{l\sss}, w_1)$
is an object in $\widehat{\wW}_y^{\mathbb{Z}/2, l_{\pm}}$ if and only 
if for any closed point $x \in \xX_1^{l\sss}$
we have 
$\iota_x^{\ast}\eE
	\in \widehat{\wW}_x^{\mathbb{Z}/2, l_{\pm}}$. 
On the other hand, it is proved in~\cite[Proposition~5.1.7]{TocatDT} 
(also see~\cite[Proposition~2.6]{KoTo})
that, under Assumption~\ref{assum:Wx}, we have the inclusion 
\begin{align*}
		\widehat{\wW}_x^{\mathbb{Z}/2, l_{-}}
		\subset 	\widehat{\wW}_x^{\mathbb{Z}/2, l_{+}}. 
	\end{align*}
Therefore the inclusion (\ref{inclu:Z2}) holds. 
	\end{proof}

\begin{rmk}
	The results of Theorem~\ref{thm:sod} and Theorem~\ref{thm:inclu}
	rely on the existence of good moduli space of $\mM$. 
	In a situation we are interested in, in many cases
	$\mM$ does not admit a good moduli space. 
	However it is sometimes possible to 
	find an open embedding $\fM \subset \fM'$ 
	for a quasi-smooth derived stack $\fM'$
	such that $\mM'=t_0(\fM')$ admits a good moduli space. 
	Then using Lemma~\ref{lem:replace0}, we can work with the
	singular support quotient of $\Dbc(\fM')$
	and apply the above results. We will apply 
	this idea in Section~\ref{sec:catwall} using 
	the moduli stacks of perverse coherent systems. 
	\end{rmk}

\section{Categorical wall-crossing at $(-1, -1)$-curve}\label{sec:catwall}
In~\cite[Section~6]{TocatDT}, 
we proved Conjecture~\ref{intro:conj:FF:PT} 
for reduced curve classes
using categorified Hall products. 
In this section, we use Theorem~\ref{thm:inclu}
to prove
Conjecture~\ref{intro:conj:FF:PT}
under wall-crossing at $(-1, -1)$-curves for relative 
reduced curve classes. 
\subsection{Categorical PT theory for local surfaces}
Let $S$ be a smooth projective surface over $\mathbb{C}$,
and $X$ the associated local surface
\begin{align*}
	\pi \colon X \cneq \mathrm{Tot}_{S}(\omega_S) \to S. 
	\end{align*}
Let $\Coh_{\le 1}(X) \subset \Coh(X)$ be the 
abelian 
subcategory of compactly supported
coherent sheaves $F$ on $X$ with $\dim \Supp(F) \le 1$. 
For each $(\beta, n) \in \mathrm{NS}(S) \oplus \mathbb{Z}$, 
we denote by 
\begin{align*}
	P_n(X, \beta)
	\end{align*}
the moduli space of Pandharipande-Thomas stable pairs~\cite{PT}
parameterizing pairs $(F, s)$ where $F \in \Coh_{\le 1}(X)$
is a pure one dimensional sheaf and $s \colon \oO_X \to F$ is surjective 
in dimension one, 
satisfying $(\pi_{\ast}[F], \chi(F))=(\beta, n)$. 

For $(\beta,  n) \in \mathrm{NS}(S) \oplus \mathbb{Z}$, 
let
\begin{align*}
	\fM_n^{\dag}(S, \beta)
	\end{align*} be the derived moduli 
stack of pairs $(F, s)$, 
where $F \in \Coh_{\le 1}(S)$
and $s \colon \oO_S \to F$
satisfying $([F], \chi(F))=(\beta, n)$ (see~\cite[Subsection~4.1.1]{TocatDT} for its precise 
formulation). 
It is proved in~\cite[Theorem~4.1.3]{TocatDT}
that $\fM_n^{\dag}(S, \beta)$ is quasi-smooth, 
and the classical truncation of its $(-1)$-shifted 
cotangent
is isomorphic to the component 
of the moduli stack of 
\textit{D0-D2-D6 bound states}, 
i.e. objects in the following subcategory
\begin{align*}
	\aA_X \cneq \langle \oO_{\overline{X}}, \Coh_{\le 1}(X)[-1] 
	\rangle_{\rm{ex}} \subset \Dbc(\overline{X}).
	\end{align*}
Here $X \subset \overline{X}$ is a projective compactification of 
$X$. In particular, we have the 
open immersion 
\begin{align*}
	P_n(X, \beta) \subset t_0(\Omega_{\fM_n^{\dag}(S, \beta)}[-1]), 
	\end{align*}
sending a stable pair $(F, s)$ to the 
two term complex $(\oO_{\overline{X}}\stackrel{s}{\to}F)$. 

Let $\fM_n^{\dag}(S, \beta)_{\mathrm{qc}} \subset \fM_n^{\dag}(S, \beta)$
be a quasi-compact derived open substack 
such that
\begin{align*}
	P_n(X, \beta)
	\subset 
	t_0(\Omega_{\fM_n^{\dag}(S, \beta)_{\rm{qc}}}[-1])
	\subset 	t_0(\Omega_{\fM_n^{\dag}(S, \beta)}[-1]).
	\end{align*}
We have the following conical closed substack
\begin{align}\label{ZP:us}
	\zZ^{P\us} \cneq t_0(\Omega_{\fM_n^{\dag}(S, \beta)_{\rm{qc}}}[-1])
	\setminus  P_n(X, \beta) \subset t_0(\Omega_{\fM_n^{\dag}(S, \beta)_{\rm{qc}}}[-1]). 
	\end{align}
Following Definition~\ref{defi:DTcat}, the DT category 
for PT moduli spaces on the local 
surface is defined in~\cite{TocatDT} as follows: 
\begin{defi}\emph{(\cite[Definition~4.2.1]{TocatDT})}\label{defi:PTt}
	The $\C$-equivariant DT category for $P_n(X, \beta)$ is defined by
	\begin{align*}
		\dDT^{\C}(P_n(X, \beta)) \cneq 
		\Dbc(\fM_n^{\dag}(S, \beta))/\cC_{\zZ^{P\us}}. 
		\end{align*}
	\end{defi}

\subsection{Wall-crossing of D0-D2-D6 bound states}\label{subsec:PTwall}
There are some variants of stable pair moduli spaces, 
depending on choices of stability conditions. 
Let 
\begin{align*}
	\cl \colon K(\aA_X) \to \mathbb{Z} \oplus \mathrm{NS}(S) \oplus \mathbb{Z}
	\end{align*}
be a group homomorphism defined by 
$\cl(\oO_X)=(1, 0, 0)$
and $\cl(F[-1])=(\pi_{\ast}[F], \chi(F))$
for $F \in \Coh_{\le 1}(X)$. 
We also fix an ample divisor $H$ on $S$. 
For $E \in \aA_X$ with $\cl(E)=(r, \beta, n)$
and $t \in \mathbb{R}$, we set 
$\mu_t^{\dag}(E) \in \mathbb{Q} \cup \{\infty\}$
to be 
\begin{align*}
	\mu_t^{\dag}(E) \cneq \begin{cases}
		t, & \mbox{ if } r>0, \\
		n/H\cdot \beta, & \mbox{ if } r=0. 
		\end{cases}
	\end{align*}
\begin{defi}\emph{(\cite[Definition~4.2.5]{TocatDT})}
	An object $E \in \aA_X$ is $\mu_t^{\dag}$-(semi)stable if for any 
	exact sequence $0 \to E_1 \to E \to E_2 \to 0$ in $\aA_X$ with $E_i \neq 0$, 
	we have $\mu_t^{\dag}(E_1)<(\le) \mu_t^{\dag}(E_2)$. 
	\end{defi}
Let $\pP_n^t(X, \beta)$ be the 
classical moduli stack of $\mu_t^{\dag}$-semistable objects in $\aA_X$
with $\cl(-)=(1, \beta, n)$. 
By~\cite[Theorem~7.25]{AHLH}, it admits a good moduli space (also see Remark~\ref{rmk:funcl})
\begin{align}\label{gmoduli:Pn}
	\pP_n^t(X, \beta) \to P_n^t(X, \beta). 
	\end{align}
There is a finite set of walls $W \subset \mathbb{R}$ such that 
the above morphism is an isomorphism if $t \notin W$. 
If $t \in W \cap \mathbb{R}_{>0}$, then 
we have the following wall-crossing diagram 
\begin{align}\label{dia:PT}
	\xymatrix{
P_n^{t_+}(X, \beta) \ar[rd] & & P_n^{t_-}(X, \beta)\ar[ld] \\
& P_n^t(X, \beta) &
	}
	\end{align}
which is a \textit{d-critical flip} (see~\cite[Theorem~9.13]{Toddbir}), i.e. 
a d-critical analogue of usual flip in birational geometry.  
Moreover we have (see~\cite[Theorem~3.21]{Tolim2})
\begin{align*}
	P_n^t(X, \beta)=P_n(X, \beta), \ t \gg 0
	\end{align*} 
By taking $t_1>t_2> \cdots>t_k>0$ which do not lie on walls, 
we have the sequence of d-critical flips
\begin{align}\label{seq:PT}
	P_n(X, \beta) \dashrightarrow P_n^{t_1}(X, \beta) \dashrightarrow 
	\cdots \dashrightarrow P_n^{t_k}(X, \beta). 
	\end{align}
For each $\beta \in \mathrm{NS}(S)$
we set 
\begin{align}\label{nbeta}
	n(\beta) \cneq \mathrm{min}\left\{ \chi(\oO_Z) : 
	\begin{array}{l}
		Z \subset X \mbox{ is a compactly supported } \\
		\mbox{closed subscheme with } \pi_{\ast}[Z] \le \beta
	\end{array}  \right\}. 
\end{align}
Then $n(\beta)>-\infty$ by~\cite[Lemma~3.10]{Tolim}. 
We will use the following lemma: 
\begin{lem}\label{lem:nbeta}\emph{(\cite[Lemma~4.2.7]{TocatDT})}
	If $\pP_{n'}^t(X, \beta') \neq \emptyset$
	for $0<\beta'\le \beta$
	and $t>0$, we have 
	$n'\ge n(\beta)$. 
	\end{lem}

Similarly to the moduli spaces of PT stable pairs, 
there exists a quasi-compact derived open substack 
$\fM_n^{\dag}(S, \beta)_{\rm{qc}} \subset \fM_n^{\dag}(S, \beta)$
such that 
\begin{align*}
	\pP_n^t(X, \beta) 
	\subset t_0(\Omega_{\fM_n^{\dag}(S, \beta)_{\rm{qc}}}[-1])
	\subset t_0(\Omega_{\fM_n^{\dag}(S, \beta)}[-1]). 
	\end{align*}
Similarly to (\ref{ZP:us}), we have the 
conical closed substack 
\begin{align*}
	\zZ^{t\us} \cneq  t_0(\Omega_{\fM_n^{\dag}(S, \beta)_{\rm{qc}}}[-1])
	\setminus 	\pP_n^t(X, \beta)
	\subset t_0(\Omega_{\fM_n^{\dag}(S, \beta)_{\rm{qc}}}[-1]). 	
	\end{align*}
\begin{defi}\emph{(\cite[Definition~4.2.9]{TocatDT})}\label{defi:PTt2}
	The $\C$-equivariant DT category for $\pP_n^t(X, \beta)$ is defined 
	by 
	\begin{align*}
		\dDT^{\C}(\pP_n^t(X, \beta)) \cneq 
		\Dbc(\fM_n^{\dag}(S, \beta))/\cC_{\zZ^{t\us}}. 
		\end{align*}
	\end{defi}
\begin{rmk}\label{rmk:indep}
	By Lemma~\ref{lem:replace0}, the DT categories in Definition~\ref{defi:PTt} and Definition~\ref{defi:PTt2}
	are independent of a choice of $\fM_n^{\dag}(S, \beta)_{\rm{qc}}$ up to equivalence. 
	\end{rmk}
As an analogy of D/K equivalence conjecture in birational 
geometry~\cite{B-O2, MR1949787}, 
the following conjecture is proposed in~\cite{TocatDT}. 
\begin{conj}\emph{(\cite[Conjecture~4.24]{TocatDT})}\label{conj:FF:PT}
	In the diagram (\ref{dia:PT}),
	there exists a fully-faithful functor 
	\begin{align*}
		\dDT^{\C}(P_n^{t_-}(X, \beta)) \hookrightarrow 
		\dDT^{\C}(P_n^{t_+}(X, \beta)).
		\end{align*} 
	In particular in the diagram (\ref{seq:PT}), 
	we have the chain of fully-faithful functors
	\begin{align*}
		\dDT^{\C}(P_n^{t_k}(X, \beta)) \hookrightarrow 
		\cdots \hookrightarrow 
		\dDT^{\C}(P_n^{t_1}(X, \beta))
		\hookrightarrow \dDT^{\C}(P_n(X, \beta)). 
		\end{align*}
	\end{conj}

\subsection{Wall-crossing at $(-1, -1)$-curve}
Suppose that the surface $S$ contains a $(-1)$-curve 
$C \subset S$. Let
\begin{align*}
	f \colon S \to T
\end{align*}
be a birational contraction which contracts $C$
to a smooth point in $T$.
We have the closed embeddings
\begin{align*}
	C \hookrightarrow S \stackrel{i}{\hookrightarrow}X, \ 
	N_{C/X}=\oO_C(-1) \oplus \oO_C(-1), 
\end{align*}
where $i$ is the zero section. 
Let $h$ be an ample divisor on $S$. 
For $X=\mathrm{Tot}_S(\omega_S)$, we have the commutative diagram 
\begin{align*}
	\xymatrix{
		X \ar[r]^-{g} \ar[d]_-{\pi} & Y \cneq \Spec_{T}\left(\bigoplus_{k\ge 0} 
		m_p^k \otimes \omega_{T}^{-k} \right) \ar[d] \\
		S \ar[r]^-{f} & T
	}
\end{align*}
where $p=f(C)\in T$, $m_p \subset \oO_{T}$ is the ideal 
sheaf of $p$. 
The morphism $g$ is a flopping contraction which 
contracts a $(-1, -1)$-curve 
$C$ to a conifold point in $Y$. 
We take an ample class $H \in \NS(S)$ of the form
\begin{align*}
	H=a \cdot f^{\ast}h-C, \ a \gg 0. 
\end{align*}
Let $W \subset \mathbb{R}$ be the set of walls 
with respect to the $\mu_t^{\dag}$-stability 
and the above choice of $H$. 
\begin{lem}\label{lem:abig}
	For a fixed $(\beta, n) \in \mathrm{NS}(S) \oplus \mathbb{Z}$ such that $\beta$
	is effective, we take $a \gg 0$
	so that 
	\begin{align}\label{a:gg}
		a> \mathrm{max}\left\{ \frac{2n-2n(\beta)+C \cdot \beta'}{h \cdot 
		f_{\ast}\beta'} : 0<\beta' \le \beta, f_{\ast}\beta' \neq 0
	   \right\}. 
		\end{align}
	Then we have $W \cap \mathbb{R}_{\ge 1/2} \subset \mathbb{Z}_{\ge 1}$. 
	For $k \in \mathbb{Z}_{\ge 1}$ and $t=k$,
	any strictly $\mu_t^{\dag}$-polystable 
	object in $\aA_X$ with $\cl(-)=(1, \beta, n)$ is of the form 
	\begin{align}\label{polyI}
		I=I_1 \oplus \oO_C(k-1)^{\oplus m}[-1],
		\end{align}
	where $m \in \mathbb{Z}_{\ge 1}$ and $I_1 \in \aA_X$ is 
	$\mu_t^{\dag}$-stable. 
	\end{lem}
\begin{proof}
	For $t\ge 1/2$, let 
	$I \in \aA_X$ be a strictly $\mu_t^{\dag}$-polystable object
	with $\cl(I)=(1, \beta, n)$. 
	It is of the form 
	\begin{align*}
		I=I_1 \oplus F_2[-1]
		\end{align*}
	where $I_1 \in \aA_X$ is $\mu_t^{\dag}$-stable, 
	$F_2[-1] \in \Coh_{\le 1}(X)[-1]$ is 
	$\mu_t^{\dag}$-semistable, 
	such that 
	\begin{align*}
	\cl(I_1)=(1, \beta_1, n_1), \ 
	\cl(F_2[-1])=(0, \beta_2, n_2), \
	\frac{n_2}{H \cdot \beta_2}=t \ge \frac{1}{2}. 
	\end{align*}
	Since we have
	$(\beta, n)=(\beta_1, n_1)+(\beta_2, n_2)$,  
	we have 
	\begin{align}\label{n1:ineq}
		n_1=n-n_2 \le n-\frac{1}{2}(ah \cdot f_{\ast}\beta_2-C \cdot \beta_2). 
	\end{align}
	Since $\beta_1 \le \beta$, we have 
	$n_1 \ge n(\beta)$ by Lemma~\ref{lem:nbeta}. 
	Therefore 
	 for $a \gg 0$ satisfying (\ref{a:gg}), 
	 we have $f_{\ast}\beta_2=0$, so 
	$\beta_2$ is of the form $\beta_2=m[C]$. 
Then $F_2$ is a semistable sheaf supported on $C$, so 
it is of the form 
$F_2=\oO_C(k-1)^{\oplus m}$ for some 
$m \in \mathbb{Z}_{\ge 1}$ and $k \in \mathbb{Z}$. 
As $\mu_t^{\dag}(F_2)=t=k$, we have 
$t \in \mathbb{Z}_{\ge 1}$, 
and obtain the desired form (\ref{polyI}) for $\mu_t^{\dag}$-polystable objects. 
\end{proof}

\begin{lem}\label{lem:agg2}
	For $a\gg 0$ satisfying (\ref{a:gg}) and $t\ge 1/2$, 
	an object $I \in \aA_X$ with $\cl(I)=(1, \beta, n)$
	is $\mu_t^{\dag}$-semistable if and only if 
	the following conditions are satisfied: 
	\begin{enumerate}
		\item $g_{\ast}\hH^1(I)$ is zero dimensional, 
		\item $\Hom(\oO_C(m)[-1], I)=0$ for $m+1>t$,
		\item $\Hom(I, \oO_C(m)[-1])=0$ for $t>m+1$, 
		\item $\Hom(\oO_x[-1], I)=0$ for $x \in X \setminus C$.  
		\end{enumerate}
	\end{lem}
\begin{proof}
For a $\mu_t^{\dag}$-semistable object $I \in \aA_X$, the 
conditions (ii), (iii) and (iv) are obviously satisfied. 
As for (i), suppose that $g_{\ast}\hH^1(I)$ is 
not zero dimensional. 
By the $\mu_t^{\dag}$-stability we have 
$\mu_t^{\dag}(\hH^1(I)) \ge t\ge 1/2$. 
By setting $\cl(\hH^0(I))=(1, \beta'', n'')$
and $\cl(\hH^1(I)[-1])=(0, \beta', n')$, 
we have $n''=n-n' \ge n(\beta)$ 
since $\hH^0(I)=I_{Z}$ for a
compactly supported closed subscheme $Z \subset X$ with 
$\pi_{\ast}[Z] \le \beta$. 
	Therefore we obtain 
	\begin{align*}
		\frac{n-n(\beta)}{(af^{\ast}h-C) \cdot \beta'} \ge
		\frac{n'}{(af^{\ast}h-C) \cdot \beta'} \ge \frac{1}{2}, 
		\end{align*}
	which contradicts to (\ref{a:gg}). 
	
	Conversely, suppose that $I \in \aA_X$ satisfies (i) to (iv). 
	Let $I\twoheadrightarrow F[-1]$ be a surjection 
	in $\aA_X$ such that $F \in \Coh_{\le 1}(X)$ with 
	 $\mu_t^{\dag}(F)<t$. 
	As $g_{\ast}\hH^1(I)$ is zero dimensional, 
	$g_{\ast}F$ is also zero dimensional. 
	If $F\neq 0$, then there is a surjection 
	$F\twoheadrightarrow \oO_C(m)$
	for $t>m+1$, which contradicts to (iii). 
	Therefore $F=0$. 
Then 
	by taking the 
	Harder-Narasimhan filtration of $I$ 
	for $\mu_t^{\dag}$-stability, we have the 
	exact sequence 
	$0 \to F_1[-1] \to I \to I_2 \to 0$
	in $\aA_X$ 
	such that $F_1 \in \Coh_{\le 1}(X)$
	with $\mu_t^{\dag}(F_1)>t$ and $I_2$
	is $\mu_t^{\dag}$-semistable
	of rank one.  
	We set $\cl(F_1[-1])=(0, \beta_1, n_1)$
	and 
	$\cl(I_2)=(1, \beta_2, n_2)$. 
	Then $n_2=n-n_1 \ge n(\beta)$
	by Lemma~\ref{lem:nbeta}. 
	It follows that 
	\begin{align*}
			\frac{n-n(\beta)}{(af^{\ast}h-C) \cdot \beta_1}
			\ge \frac{n_1}{(af^{\ast}h-C) \cdot \beta_1}
			=\mu_t^{\dag}(F_1[-1]) >t \ge \frac{1}{2}
		\end{align*}
	which contradicts to (\ref{a:gg})
	if $f_{\ast}\beta_1 \neq 0$. 
	Therefore $f_{\ast}\beta_1=0$ so 
	$g_{\ast}F_1$ is zero dimensional. 
	If $F_1 \neq 0$, then there is an injection 
	$\oO_C(m) \hookrightarrow F_1$ for 
	$m+1>t$ or $\oO_x \hookrightarrow F_1$
	for $x \in X \setminus C$, 
	which contradicts to (ii) or (iv). 
	Therefore we conclude that $F_1=0$, i.e. 
	$I$ is $\mu_t^{\dag}$-semistable. 
	\end{proof}

Below we take $a\gg 0$ as in Lemma~\ref{lem:abig}, and set
\begin{align*}
P_n^{(k)}(X, \beta) \cneq P_n^{k+1-0}(X, \beta)	
	\end{align*}
for $k\in \mathbb{Z}_{\ge 0}$, which equals to $P_n^{k+0}(X, \beta)$
if $k\ge 1$ by Lemma~\ref{lem:abig}. 
The wall-crossing diagram (\ref{dia:PT}) 
in this case at $t=k$ is 
\begin{align}\notag
	\xymatrix{
		P_n^{(k)}(X, \beta) \ar[rd] & & P_n^{(k-1)}(X, \beta)\ar[ld] \\
		& P_n^{t=k}(X, \beta). &
	}
\end{align}
The assertion in Conjecture~\ref{conj:FF:PT} is specialized
as follows: 
\begin{conj}\label{conj:-1}
	For each $k\in \mathbb{Z}_{\ge 1}$, there exists a fully-faithful functor 
	\begin{align*}
		\dDT^{\C}(P_n^{(k-1)}(X, \beta)) \hookrightarrow \dDT^{\C}(P_n^{(k)}(X, \beta)). 
		\end{align*}
	\end{conj}

\subsection{Perverse coherent sheaves}
We will describe the moduli spaces $P_n^{(k)}(X, \beta)$ in terms of 
perverse coherent sheaves introduced in~\cite{Br1, MR2057015}, defined as follows: 
\begin{defi}\emph{(\cite{Br1, MR2057015})}
	The subcategory of perverse coherent sheaves
\begin{align*}
	\PPer(S/T) 
	\subset \Dbc(S)
	\end{align*}
is defined to be the subcategory of objects $E \in \Dbc(S)$
satisfying the following conditions: 
\begin{enumerate}
	\item we have
	$\dR f_{\ast}E \in \Coh(T)$,
	\item we have $\Hom^{<0}(\oO_C(-1)[1], E)=\Hom^{<0}(E, \oO_C(-1)[1])=0$.  
\end{enumerate}
\end{defi}
The subcategory $\PPer(S/T)$ is the heart of a bounded t-structure 
on $\Dbc(S)$
which contains $\oO_S$. 
We will denote by $\pH^i(-)$ the $i$-th cohomology with 
respect to the above perverse t-structure. 
By setting 
\begin{align*}
	\eE\cneq \oO_S \oplus \oO_S(-C), \ 
	A_S \cneq f_{\ast}\eE nd(\eE),
	\end{align*}
we have the equivalence (see~\cite{MR2057015})
\begin{align}\label{equiv:pervS}
	\dR f_{\ast}\dR \hH om(\eE, -) \colon 
	\PPer(S/T) \stackrel{\sim}{\to} \Coh(A_S). 
\end{align}
We also have the following abelian subcategories for $i=0, 1$
\begin{align*}
	\PPer_{\le i} (S/T)\cneq \{
	E \in \PPer(S/T) : \dim \Supp(\dR f_{\ast}E) \le i \}. 
\end{align*}

The subcategories
\begin{align*}
	\PPer_{\le i}(X/Y) \subset \PPer(X/Y) \subset \Dbc(X)
	\end{align*}
are also defined in a similar way, using the 
flopping contraction $g \colon X \to Y$ instead of $f \colon S \to T$, 
imposing an additional condition that $\dR g_{\ast}E$ is compactly supported 
for $i=1$. 
Similarly to (\ref{equiv:pervS}), we have the equivalence 
\begin{align}\label{equiv:X}
	\dR g_{\ast} \dR \hH om(\pi^{\ast}\eE, -) \colon 
	\PPer(X/Y) \stackrel{\sim}{\to} \Coh(A_X)
\end{align}
where $A_X=g_{\ast}\eE nd(\pi^{\ast}\eE)$. 
In particular, the push-forward gives a functor 
\begin{align*}
	\pi_{\ast} \colon \PPer_{\le i}(X/Y) \to \PPer_{\le i}(S/T). 
\end{align*}
For $i=0$, 
the abelian categories $\PPer_{\le 0}(S/T)$, $\PPer_{\le 0}(X/Y)$
are generated by their simple objects, 
and described as the extension closures 
\begin{align}\label{per0:gen}
	\PPer_{\le 0}(S/T)
	=\langle S_0, S_1, \oO_{x} : x \in S \setminus C \rangle_{\rm{ex}}, \ 
	 	\PPer_{\le 0}(X/Y)
	 =\langle S_0, S_1, \oO_{x} : x \in X \setminus C \rangle_{\rm{ex}}
\end{align}
where $S_0 \cneq \oO_C$
and $S_1 \cneq \oO_C(-1)[1]$. 

\begin{rmk}\label{rmk:tilting}
	By~\cite{MR2057015}, the heart $\PPer_{\le 1}(X/Y)$ is obtained as a tilting 
	of $\Coh_{\le 1}(X)$ by the torsion pair $(\tT, \fF)$, i.e. 
	\begin{align*}
		\PPer_{\le 1}(X/Y)=\langle \fF[1], \tT \rangle_{\rm{ex}} \subset 
		\Dbc(X)
	\end{align*}
	where $\fF$ consists of sheaves $F$ with $g_{\ast}F=0$. 
	By taking Harder-Narasimhan filtration, we see that 
	$\fF$ is the extension closure of objects 
	$\oO_C(m)$ with $m\le -1$. 
\end{rmk}

\subsection{Moduli stacks of perverse coherent systems}\label{subsec:percoh}
The notion of perverse coherent systems is defined as follows. 
\begin{defi}\label{def:psystem}
	A \textit{perverse coherent system} is a pair
	\begin{align}\label{pcoh}
		(F, s), \ F \in \PPer_{\le 1}(S/T), \ 
		s \colon \oO_S \to F. 
	\end{align}
	A perverse coherent system $(F, s)$
	is called a \textit{pure quotient} if 
	$s$ is surjective in $\PPer(S/T)$ and 
	$F$ does not have a non-zero subobject in $\PPer_{\le 0}(S/T)$.  
\end{defi}
We denote by 
\begin{align*}
	\fM_n^{\mathrm{per}\dag}(S, \beta)
	\end{align*}
the derived moduli stack of perverse
coherent systems $(F, s)$ 
where $F \in \PPer_{\le 1}(S/T)$ 
and $s \colon \oO_S \to F$, 
satisfying $([F], \chi(F))=(\beta, n) \in \mathrm{NS}(S) \oplus \mathbb{Z}$. 
\begin{lem}\label{lem:pqsmooth}
	The derived stack $\fM_n^{\mathrm{per}\dag}(S, \beta)$ is quasi-smooth. 
	\end{lem}
\begin{proof}
The tangent complex of $\fM_n^{\mathrm{per}\dag}(S, \beta)$
at a perverse coherent system $(\oO_S \to F)$ is given by 
$\dR \Gamma(\oO_S\to F, F)$. 
Therefore it is enough to show that 
$\dR \Gamma^{\ge 2}(F)=\Hom^{\ge 3}(F, F)=0$
for any $F \in \PPer_{\le 1}(S/T)$. 
  The first vanishing holds since 
  $\dR f_{\ast}F \in \Coh_{\le 1}(T)$. 
  As for the second vanishing, 
  by the Serre duality we have 
  \begin{align}\label{Fomega}
  	\Hom(F, F[i])=\Hom(F, F \otimes \omega_S[2-i])^{\vee}.
  	  	\end{align}
 Lett 
  $(\tT, \fF)$ be a torsion pair of $\Coh_{\le 1}(S)$ as in Remark~\ref{rmk:tilting}. 
  Then $\fF \otimes \omega_S \subset \fF$
  and $\tT \otimes \omega_S \subset \Coh_{\le 1}(X)$, so we have 
  \begin{align*}
  	\PPer_{\le 1}(S/T) \otimes \omega_S \subset \langle 
  	\PPer_{\le 1}(S/T), \PPer_{\le 1}(S/T)[-1] \rangle_{\rm{ex}}. 
  	\end{align*}
  Therefore (\ref{Fomega}) vanishes for $i\ge 3$. 
\end{proof}
We have the derived open substack 
\begin{align}\notag
	\fP_n^{\rm{per}}(S, \beta) \subset \fM_n^{\mathrm{per}\dag}(S, \beta)
\end{align}
corresponding to perverse coherent systems (\ref{pcoh})
such that $\pcok(s) \in \PPer_{\le 0}(S/T)$, 
where $\pcok(-)$ denotes the cokernel in the heart $\PPer(S/T)$. 
Its classical truncation is denoted by 
$\pP_n^{\rm{per}}(S, \beta)$. 
\begin{lem}\label{lem:persys}
The moduli stack 
$\pP_n^{\rm{per}}(S, \beta)$ admits a good moduli space
\begin{align}\label{gmoduli:P}
	\pP_n^{\rm{per}}(S, \beta) \to P_n^{\rm{per}}(S, \beta)
\end{align}
where each closed point of $P_n^{\rm{per}}(S, \beta)$ corresponds to a direct sum 
\begin{align}\label{points:per2}
	I=I_0 \oplus 
	(V_0 \otimes S_0[-1]) \oplus( V_1 \otimes S_1[-1]) \oplus \bigoplus_{i=1}^k (W_i \otimes \oO_{x_i}[-1])
\end{align}
where $I_0=(\oO_S \stackrel{s_0}{\to}F_0)$ for a pure quotient $(F_0, s_0)$, 
$F[-1] \cneq (0 \to F)$ for $F \in \PPer_{\le 1}(S/T)$
and $V_0, V_1, W_i$ are finite dimensional vector spaces 
and $x_1, \ldots, x_i \in S \setminus C$
are distinct points.
Moreover the good moduli space morphism (\ref{gmoduli:P}) satisfies the 
formal neighborhood theorem. 
\end{lem}
\begin{proof}	
	Let $\PPer^{\dag}_{\le 1}(S/T)$
	be the abelian category of pairs 
	$(\oO_S^{\oplus r} \to F)$
	for $r \in \mathbb{Z}_{\ge 0}$ and $F \in \PPer_{\le 1}(S/T)$. 
	Similarly 
	to 
	the moduli stack at MNOP/PT wall in~\cite[Subsection~4.2.1]{TocatDT}, 
	the moduli stack $\pP_n^{\rm{per}}(S, \beta)$
	is identified with the 
	moduli stack 
	of semistable objects in $\PPer^{\dag}_{\le 1}(S/T)$
with respect to the function 
	\begin{align}\label{slope:Phi}
		p_{\nu} \colon 
	(\oO_S^{\oplus r} \to F) \mapsto 
	\begin{cases}
		0, & r>0, \\
		[f_{\ast}F] \cdot h, & r=0. 
		\end{cases}
		\end{align}
	Namely a perverse coherent system $I=(\oO_S \to F)$ is an object 
	in $\pP_n^{\rm{per}}(S, \beta)$ if and only if 
	for any exact sequence 
	$0 \to I_1 \to I \to I_2 \to 0$ in $\PPer^{\dag}_{\le 1}(S/T)$
	we have $p_{\nu}(I_1) \le 0 \le p_{\nu}(I_2)$. 
	Then the existence of good moduli space follows from~\cite[Theorem~7.25]{AHLH}  
	(also see Remark~\ref{rmk:funcl}). 

For any perverse coherent system $(F, s)$
with $\pcok(s) \in \PPer_{\le 0}(S/T)$, 
the object $E=(\oO_S \stackrel{s}{\to} F)$
admits a unique filtration 
\begin{align*}
	E_1 \subset E_2 \subset E, \ 
	E_1=(0 \to F''), \ 
	E_2/E_1=(\oO_S \stackrel{s'}{\to}F'), \ 
	E/E_2=(0 \to F''')
\end{align*} 
where $F'', F'''$ are objects in $\PPer_{\le 0}(S/T)$ and 
$(F', s')$ is a pure quotient. 
The objects $(0 \to F'')$, $(0 \to F''')$ are 
semistable with respect to (\ref{slope:Phi}), 
$(\oO_S \stackrel{s'}{\to}F')$
is stable with respect to (\ref{slope:Phi}), and all of them 
satisfy $p_{\nu}(-)=0$. 
Therefore any stable object with respect to $p_{\nu}$ is either 
a pure quotient or $F'[-1]$ for a simple object in $\PPer_{\le 0}(S/T)$. 
Then from the description of simple objects
of $\PPer_{\le 0}(S/T)$ in (\ref{per0:gen}), 
we obtain the description of polystable objects (\ref{points:per2}). 
The formal neighborhood theorem holds by the argument of~\cite[Lemma~7.4.3]{TocatDT}. 
\end{proof}

\begin{rmk}\label{rmk:funcl}
	We need to take a little case in applying~\cite[Theorem~7.25]{AHLH}, since 
	the function (\ref{slope:Phi}) is not additive on $K(\PPer_{\le 1}^{\dag}(S/T))$. 
	However the argument in \textit{loc.~cit.} is applied as follows. 
	For a map $f \colon \Theta \to \fM_n^{\rm{per}\dag}(S, \beta)$
	corresponding to a $\mathbb{Z}$-weighted filtration 
	 $I_{\bullet}=\cdots \to I_{w+1} \to I_w \to I_{w-1} \to \cdots$
	in $\PPer^{\dag}_{\le 1}(S/T)$, 
	we set 
\begin{align*}
	l(f) \cneq 
	\sum_{w \in \mathbb{Z}} w \cdot p_{\nu}(I_{w}/I_{w+1}).
	\end{align*}
	For $I=(\oO_S^{\oplus r} \to F) \in \PPer^{\dag}_{\le 1}(S/T)$, 
	we set $\rk(I) \cneq r$. 
	Since there is no scaling automorphism if $r>0$, 
	we have $\rk(I_{0}/I_{1})=1$ and $\rk(I_w/I_{w+1})=0$ if $w\neq 0$. 
	Then we immediately see that 
	$I=(\oO_S \stackrel{s}{\to} F)$ in $\fM_n^{\rm{per}\dag}(S, \beta)$
		satisfies $\pcok(s) \in \PPer_{\le 0}(S/T)$ 
	if and only if for any map $f \colon \Theta \to 
	\fM_n^{\rm{per}\dag}(S, \beta)$ with $f(1) \sim I$, 
	we have $l(f) \le 0$. 
	The function $l$ obviously satisfies the condition in~\cite[Remark~6.16]{AHLH}, so 
	we can apply~\cite[Theorem~7.25]{AHLH} as mentioned in \textit{loc.~cit.}. 
	A similar argument also applies to the construction of 
	the good moduli space (\ref{gmoduli:Pn}). 
	\end{rmk}

The notion of perverse coherent systems and pure quotients 
are similarly defined for $\PPer(X/Y)$. 
We denote by $\pP_n^{\rm{per}}(X, \beta)$ the classical moduli stack 
of perverse coherent systems 
$(F, s)$ for $F \in \PPer(X/Y)$ 
and $s \colon \oO_X \to F$
satisfies $\pcok(s) \in \PPer_{\le 0}(X/Y)$, where 
$\pcok(-)$ is the cokernel in $\PPer(X/Y)$. 
Similarly to Lemma~\ref{lem:persys}, we have 
the following: 
\begin{lem}\label{lem:persys2}
	The moduli stack 
	$\pP_n^{\rm{per}}(X, \beta)$ admits a good moduli space 
\begin{align}\label{perp:gmoduli}
	\pP_n^{\rm{per}}(X, \beta) \to P_n^{\rm{per}}(X, \beta)
\end{align}
whose closed points correspond to 
direct sums of the form 
\begin{align}\label{points:per3}
	I=
	I_0 \oplus 
	(V_0 \otimes S_0[-1])\oplus (V_1 \otimes S_1[-1]) \oplus \bigoplus_{i=1}^k (W_i \otimes \oO_{x_i}[-1])
\end{align}
where 
$I_0=(\oO_X \stackrel{s_0}{\to}F_0)$ for a pure quotient $(F_0, s_0)$, 
$V_0, V_1, W_i$
are finite dimensional vector spaces
and $x_1, \ldots, x_k \in X \setminus C$ are distinct points. 
\end{lem}

We note that under the equivalence (\ref{equiv:X}), 
the moduli stack $\pP_n^{\rm{per}}(X, \beta)$
is identified with the moduli 
stack of 
pairs 
\begin{align}\label{pair:NC}
	(F, s), \ 
	F \in \Coh_{\le 1}(A_X), \ 
	s \colon g_{\ast}\pi^{\ast}\eE^{\vee} \to F
\end{align}
such that $F \in \Coh_{\le 1}(A_X)$ has 
a one dimensional 
support on $Y$, 
$s$ is a morphism in $\Coh(A_X)$ 
with 
$\Cok(s) \in \Coh_{\le 0}(A_X)$.

\begin{lem}\label{lem:ncan}
	The moduli stack $\pP_n^{\rm{per}}(X, \beta)$
	is isomorphic to the moduli stack of 
	objects in the extension closure 
	\begin{align}\label{ext:NC}
		I\in \langle \oO_{\overline{X}}, \PPer_{\le 1}(X/Y)[-1] 
		\rangle_{\rm{ex}}
	\end{align}
	satisfying $\pH^1(I) \in \PPer_{\le 0}(X/Y)$
	and $\cl(I)=(1, \beta, n)$. 
	\end{lem}
\begin{proof}
	From the identification of $\pP_n^{\rm{per}}(X, \beta)$ with 
	the moduli stack of pairs (\ref{pair:NC}), 
	the lemma follows from the same arguments
	of~\cite[Lemma~4.2.2]{TocatDT}
	applied for the non-commutative scheme $(Y, A_{X})$. 
	\end{proof}

\begin{lem}\label{lem:gmoduli:perv}
	For $t\ge 1/2$, we have
	the open immersion
	\begin{align}\label{open:perv}
		\pP_n^t(X, \beta) \subset \pP_n^{\rm{per}}(X, \beta). 
	\end{align}
\end{lem}
\begin{proof}
	Let $I \in \aA_X$ be an object corresponding to a point in $\pP_n^t(X, \beta)$. 
	We have the exact sequence
	\begin{align*}
		0 \to \hH^0(I) \to I \to \hH^1(I)[-1] \to 0
	\end{align*}
	in $\aA_X$ 
	such that any Harder-Narasimhan factor $F$ of $\hH^1(I)$ satisfies 
	$\chi(F)/H \cdot [F] \ge t$. 
	Then as in the proof of Lemma~\ref{lem:abig},
	a choice $a \gg 0$ implies that $F$ is supported on $C$, 
	hence $F$ is a direct sum 
	of $\oO_C(k-1)$ for $k \ge 1$. 
	In particular we have $\hH^1(I) \in \PPer_{\le 0}(X/Y)$. 
	
	We also have $\hH^0(I) \cong I_Z$ for a compactly supported
	 closed subscheme
	$Z \subset X$ 
	with $\dim Z \le 1$. 
	Since 
	$R^1 g_{\ast}\oO_X=0$, we have 
	$R^1 g_{\ast}\oO_Z=0$.
	Moreover $\Hom(\oO_Z, \oO_C(-1))=0$ since 
	$\oO_X \twoheadrightarrow \oO_Z$ is a surjection of coherent sheaves
	and $\Hom(\oO_X, \oO_C(-1))=0$. 
	It follows that $\oO_Z \in \PPer_{\le 1}(X/Y)$. 
	Therefore   
	the morphism $\oO_X \to \oO_Z$ is a morphism in 
	$\PPer(X/Y)$
	which is generically surjective outside $C$, 
	hence its   
	cokernel in $\PPer(X/Y)$ is an object in 
	$\PPer_{\le 0}(X/Y)$. 
	It follows that $I$ is an object in (\ref{ext:NC})
	satisfying $\pH^1(I) \in \PPer_{\le 0}(X/Y)$, 
	so we obtain the the desired open immersion from Lemma~\ref{lem:ncan}. 
\end{proof}

\begin{lem}\label{lem:agg3}
	In the setting of Lemma~\ref{lem:gmoduli:perv}, 
	an object $I=(\oO_{\overline{X}} \stackrel{s}{\to} F)$ in 
	$\pP_n^{\rm{per}}(X, \beta)$ is an object in 
	$\pP_n^t(X, \beta)$ if and only if the following conditions 
	are satisfied: 
	\begin{enumerate}
		\item $F$ is a coherent sheaf, 
		\item $\Hom(\oO_C(m), F)=0$ for $m+1>t$, 
		\item $\Hom(I, \oO_C(m)[-1])=0$ for $t>m+1 \ge 1$, 
		\item $\Hom(\oO_x, F)=0$ for $x \in X \setminus C$.  
		\end{enumerate}
	\end{lem}
\begin{proof}
	The only if direction is obvious, so we only prove the 
	if direction. 
	Supopse that $I$ satisfies (i) to (iv). 
	Note that the condition (ii) is 
	equivalent to 
	$\Hom(\oO_C(m)[-1], I)=0$
	for $m+1>t$ 
	by applying $\Hom(\oO_C(m), -)$ to the 
	exact sequence 
	\begin{align*}
		0 \to F[-1] \to I \to \oO_X \to 0
		\end{align*} in $\aA_X$. 
	Similarly (iii) is equivalent to $\Hom(\oO_x[-1], I)=0$ for $x \in X \setminus C$. 
	Therefore by Lemma~\ref{lem:agg2}, it is enough 
	to show that $g_{\ast}\hH^1(I)$ is zero dimensional 
	and $\Hom(I, \oO_C(m)[-1])=0$ for $m \le -1$. 
	The second condition is obvoius since 
	for $m \le -1$ we have $\oO_C(m)[-1] \in \PPer_{\le 0}(X/Y)[-2]$.
	As for the first condition, 
	let us take the exact sequence
	$0 \to \Imm(s) \to F \to \Cok(s) \to 0$ in $\Coh(X)$. 
	As $F$ is a sheaf, we have $\Cok(s) \in \Coh(X) \cap \PPer(X/Y)$. 
	By the argument of Lemma~\ref{lem:gmoduli:perv}, we also 
	have $\Imm(s) \in \Coh(X) \cap \PPer(X/Y)$. 
	It follows that we have the exact sequence of 
	perverse coherent systems
	\begin{align*}
		0 \to (\oO_X \to \Imm(s)) \to (\oO_X \stackrel{s}{\to}F) 
		\to (0 \to \Cok(s)) \to 0. 
		\end{align*}
	Therefore we obtain the surjection 
	$\pH^1(I) \twoheadrightarrow \Cok(s)$
	in $\PPer(X/Y)$. 
	As $\pH^1(I) \in \PPer_{\le 0}(X/Y)$, 
	we have $\Cok(s) \in \PPer_{\le 0}(X/Y)$, 
	hence $\hH^1(I) =\Cok(s)$ satisfies that $g_{\ast}\hH^1(I)$ is 
	zero dimensional. 
	\end{proof}

We also define the open substack
\begin{align}\label{Pn:nc}
	P_n^{\rm{nc}}(X, \beta) 
	\subset \pP_n^{\rm{per}}(X, \beta)
	\end{align}
to be 
consisting of pairs $(F, s)$ in $\pP_n^{\rm{per}}(X, \beta)$
satisfying
$\Hom(\PPer_{\le 0}(X/Y), F)=0$. 
Then under the equivalence (\ref{equiv:X}),
the above stack is identified
with the moduli stack of pairs (\ref{pair:NC})
such that 
$F \in \Coh_{\le 1}(A_X)$ is pure one dimensional and 
$s$ is surjective in dimension one. 
Therefore $P_n^{\rm{nc}}(X, \beta)$ is 
regraded as an analogue of PT stable pair 
moduli space for the non-commutative scheme $(Y, A_X)$.  

\begin{lem}\label{lem:PTnc}
	For $t=1-0$, we have the isomorphism 
	\begin{align*}
		P_n^{t=1-0}(X, \beta) \cong P_n^{\rm{nc}}(X, \beta). 
		\end{align*}
	\end{lem}
\begin{proof}
	For $F \in \PPer_{\le 1}(X/Y)$, 
	by Remark~\ref{rmk:tilting}
	there is an exact sequence 
	\begin{align*}
		0 \to \hH^{-1}(F)[1] \to F \to \hH^0(F) \to 0
		\end{align*}
	in $\PPer_{\le 1}(X/Y)$
	such that $\hH^{-1}(F)[1] \in \PPer_{\le 0}(X/Y)$. 
	Therefore the condition 
	$\Hom(\PPer_{\le 0}(X/Y), F)$ implies that 
	$F$ is a sheaf. 
The above condition also implies that 
	\begin{align*}
		\Hom(\oO_C(m), F)=0  \ (m\ge 0), \ 
		\Hom(\oO_x, F)=0  \ (x \in X \setminus C).
		\end{align*}
It follows that a pair $(F, s)$ in $P_n^{\rm{nc}}(X, \beta)$
is a $\mu_t^{\dag}$-semistable object in $\aA_X$
by Lemma~\ref{lem:agg3}. 

Conversely for a pair $(F, s)$ in $P_n^{t=1-0}(X, \beta)$, 
we have $\Hom(\PPer_{\le 0}(X/Y), F)=0$ by 
Lemma~\ref{lem:agg3} and the description of 
generators of $\PPer_{\le 0}(X/Y)$ in (\ref{per0:gen}). 
Therefore it is a pair in $P_n^{\rm{nc}}(X, \beta)$. 
 	\end{proof}

By the above lemma, the sequence of d-critical flips in (\ref{seq:PT})
is in this case a sequence
\begin{align*}
	P_n(X, \beta)=P_n^{(k)}(X, \beta) \dashrightarrow 
	P_n^{(k-1)}(X, \beta) \dashrightarrow \cdots \dashrightarrow P_n^{(0)}(X, \beta)=P_n^{\rm{nc}}(X, \beta)
\end{align*}
where $k\gg 0$, 
which connects commutative stable pair moduli space $P_n(X, \beta)$
with non-commutative stable pair moduli space $P_n^{\rm{nc}}(X, \beta)$
by d-critical minimal model program. 

\subsection{The case of relative reduced curve classes}
We define the following curve classes on $S$: 
\begin{defi}\label{def:f-red}
	A class $\beta \in \mathrm{NS}(S)$ is called $f$-reduced if 
	$f_{\ast}\beta$ is a reduced class, i.e. 
	any effective divisor on $T$ with class
	$f_{\ast}\beta \in \mathrm{NS}(T)$ is a reduced divisor. 
\end{defi}
If an effective divisor $D$ on $S$ is of class $\beta$
which is $f$-reduced, then $D$ is of the form 
$k[C]+[C']$ where $k \in \mathbb{Z}_{\ge 0}$ and 
$C'$ is an effective divisor which does not contain $C$. 
We have the following lemma. 
\begin{lem}\label{lem:fred}
	Suppose that $\beta$ is $f$-reduced. 
	Then we have the isomorphism over $\pP_n^{\rm{per}}(S, \beta)$
	\begin{align}\label{isom:pervP}
		\pP_n^{\rm{per}}(X, \beta) \stackrel{\cong}{\to} t_0(\Omega_{\fP_n^{\rm{per}}(S, \beta)}[-1]). 
	\end{align}
\end{lem}
\begin{proof}
Similarly to Lemma~\ref{lem:ncan}, 
the left hand side is identified with the moduli stack
of pairs (\ref{pair:NC}) 
such that $F$ has a reduced one dimensional support on $Y$
and $\Cok(s) \in \Coh_{\le 0}(A_X)$. 
Then 
the lemma 
follows from the argument of 
~\cite[Lemma~5.5.4]{TocatDT} for the 
non-commutative scheme $(Y, A_X)$. 
\end{proof}

For $t \ge 1/2$, we define the conical closed substack 
\begin{align*}
	\zZ^{{\rm{per}}, t\us} \subset  t_0(\Omega_{\fP_n^{\rm{per}}(S, \beta)}[-1])
	\end{align*}
to be the complement of the open embedding (\ref{open:perv}) via the 
isomorphism in Lemma~\ref{lem:fred}. 
We have the following alternative description of DT categories 
in this case: 
\begin{lem}\label{lem:altenative}
	Suppose that $\beta$ is $f$-reduced. 
	Then for $t\ge 1/2$, there is an equivalence 
	\begin{align*}
		\dDT^{\C}(\pP_n^t(X, \beta)) \stackrel{\sim}{\to}
		\Dbc(\fP_n^{\rm{per}}(S, \beta))/\cC_{\zZ^{{\rm{per}}, t\us} }. 
		\end{align*}
	\end{lem}
\begin{proof}
	Let $\mathfrak{Q}_n^{\rm{per}}(S, \beta) \subset \fP_n^{\rm{per}}(S, \beta)$
	be the derived open substack of $(\oO_S \to F)$ such that 
	$F \in \Coh(S)$. 
	Then by Lemma~\ref{lem:gmoduli:perv} and Lemma~\ref{lem:fred}, 
	we can take $\fM_n^{\dag}(S, \beta)_{\rm{qc}}$ to be 
	$\mathfrak{Q}_n^{\rm{per}}(S, \beta)$
	in Definition~\ref{defi:PTt2}. 
	Then the lemma follows from Lemma~\ref{lem:replace0}. 
	\end{proof}

Below we describe the open substack (\ref{open:perv}) in terms of semistable 
points with respect to some $\mathbb{R}$-line bundle on $\pP_n^{\rm{per}}(X, \beta)$. 
We denote by 
\begin{align*}
	(\oO_{S \times \pP} \to \fF_{\pP}) 
	\in \Dbc(S \times \pP_n^{\rm{per}}(S, \beta))
\end{align*}
the universal perverse coherent 
systems. 
For $t \ge 1/2$, 
we define the following $\mathbb{R}$-line bundle on 
$\pP_n^{\rm{per}}(S, \beta)$
\begin{align*}
\lL_{t} \cneq \det \dR p_{\pP_{\ast}}(\fF_{\pP})^{t-1} \otimes 
	\det \dR p_{\pP_{\ast}}(\fF_{\pP} \boxtimes \oO_S(C))^{-t}. 
\end{align*}
Here $p_{\pP} \colon S \times \pP_n^{\rm{per}}(S, \beta) \to 
\pP_n^{\rm{per}}(S, \beta)$ is the projection. 
We also denote its pull-back to the $(-1)$-shifted cotangent
$t_0(\Omega_{\fP_n^{\rm{per}}(S, \beta)}[-1])$
by $\lL_{t}$, and also regard it
as a $\mathbb{R}$-line bundle on 
$\pP_n^{\rm{per}}(X, \beta)$
via the isomorphism (\ref{isom:pervP}). 
By setting $l_t \cneq c_1(\lL_t)$, we have the following proposition: 
\begin{prop}\label{prop:sss}
	The open substack (\ref{open:perv}) is the $l_{t}$-semistable locus, i.e. 
	\begin{align}\label{id:PM:sss}
			\pP_n^t(X, \beta) = \pP_n^{\rm{per}}(X, \beta)^{l_{t}\sss}. 	
	\end{align}
\end{prop}
\begin{proof} 
	Let $I=(\oO_X \to F)$ be an object corresponding 
	to a point in $\pP_n^{\rm{per}}(X, \beta)$ which is $l_t$-semistable. 
	Suppose that it is not an object in $\pP_n^t(X, \beta)$. 
	Then by Lemma~\ref{lem:agg3}, it violates one of the conditions (i) to (iv). 
	Suppose that it violates (i), i.e. $F$ is not a sheaf. 
	Then by Remark~\ref{rmk:tilting}, 
	there is an exact sequence
	\begin{align*}
		0 \to \oO_C(m)[1] \to F \to F' \to 0
		\end{align*}
	in $\PPer_{\le 1}(X/Y)$ 
	for some $m\le -1$. 
	So there is a map $f \colon \Theta \to \pP_n^{\rm{per}}(X, \beta)$
	such that $f(1) \sim I$ and 
	\begin{align*}
		f(0) \sim (0 \to \oO_C(m)[1]) \oplus (\oO_X \to F'), 
		\end{align*}
	where $(0 \to \oO_C(m)[1])$ has $\C$-weight $1$ and $(\oO_X \to F')$ has $\C$-weight zero. 
	Then we have 
	\begin{align*}
		q^{-1}f^{\ast}l_t &=(t-1)\chi(\oO_C(m)[1])-t \chi(\oO_C(m-1)[1]) \\
		&=-t+m+1<0
		\end{align*}
	which violates the $l_t$-semistability. 
	Similarly one can show that if $I$ violates one of
	other conditions in Lemma~\ref{lem:agg3} then it violates the $l_t$-semistability. 
	Therefore $I$ is an object in $\pP_n^t(X, \beta)$. 
	
	Conversely suppose that $I=(\oO_X \to F)$ is an object in $\pP_n^t(X, \beta)$, 
	and take a map $f \colon \Theta \to \pP_n^{\rm{per}}(X, \beta)$
	with $f(1) \sim I$. 
	The map $f$ corresponds to a $\mathbb{Z}$-weighted filtration 
	of perverse coherent systems
	\begin{align*}
		I_{\bullet}=( \cdots \to I_{w+1} \to I_{w} \to I_{w-1} \to \cdots )
		\end{align*}
	such that 
	$I_0/I_1$ is of the form $(\oO_X \to F_0)$
	and $I_{w}/I_{w+1}$ for $w\neq 0$ is of the form 
	$(0 \to F_{w})$ for $F_{w} \in \PPer_{\le 0}(X/Y)$. 
	We have 
	\begin{align}\notag
		q^{-1}f^{\ast}l_t &=
		\sum_{w \in \mathbb{Z} \setminus \{0\}}
		w \cdot \left\{ (t-1)\chi(I_{w}/I_{w+1})-t \chi(I_w/I_{w+1}(C))\right\} \\
		\label{q-1l}&=\sum_{w<0} -(t-1)\chi(I/I_{w})+t \chi(I/I_w(C))
		+\sum_{w>0}(t-1)\chi(I_w)-t\chi(I_w(C)). 
		\end{align}
	For $w>0$, the object $I_w$ is of the form $(0 \to F_w')$ for $F_w' \in \PPer_{\le 0}(X/Y) \cap \Coh(X)$
	satisfying $\Hom(\oO_C(m), F_w')=0$ for $m+1>t$ and $\Hom(\oO_x, F_w')=0$
	for $x \notin X\setminus C$ 
	by Lemma~\ref{lem:agg3}. 
	Therefore $F_{w}'$ lies in the extension closure of $\oO_C(m)$ for $m+1 \le t$. 
	Since we have 
	\begin{align*}
		(t-1)\chi(\oO_C(m))-t \chi(\oO_C(m-1))=t-m-1 \ge 0,
		\end{align*}
	the second sum in (\ref{q-1l}) is non-negative. 
	Similary $I/I_w$ for $w<0$ is of the form 
	$(0 \to F_w')$ where $F_{w}'$ lies in the extension closure of 
	$\oO_C(m)$ for $m+1 \ge t$, $\oO_x$ for $x \in X \setminus C$
	and $\oO_C(m)[1]$ for $m\le -1$. 
	Therefore the first sum in (\ref{q-1l}) is also non-negative, 
	and we conclude that $I$ is $l_t$-semistable. 
\end{proof}

We also set $b \in H^4(\pP_n^{\rm{per}}(S, \beta), \mathbb{Q})$ by 
\begin{align}\label{b:pos}
	b\cneq \ch_2(\dR p_{\pP\ast}(\fF_{\pP}))
	+\ch_2(\dR p_{\pP\ast}(\fF_{\pP} \boxtimes \oO_S(C))). 
	\end{align}
\begin{lem}\label{lem:posi}
	The element (\ref{b:pos}) is positive definite in the sense 
	of Definition~\ref{defi:bpos}. 
	\end{lem}
\begin{proof}
	Let $f \colon B\C \to \pP_n^{\rm{per}}(S, \beta)$ be a non-degenerate
	morphism. 
	The morphism $f$ corresponds to a perverse coherent system 
	$I$
	with $\C$-automorphisms. By taking the decomposition into 
	$\C$-weight part, it decomposes into 
	\begin{align*}
	I=(\oO_S \to F_0) \oplus \bigoplus_{i=1}^k F_i[-1]. 
		\end{align*}
	Here
	$(\oO_S \to F_0)$ has weight $0$ and 
	$F_i \in \ \PPer_{\le 0}(S/T)$ 
	has non-zero weight $w_i$. 
	It follows that 
	\begin{align}\label{q-2:b}
		q^{-2} f^{\ast}b=\sum_{i=1}^k w_i^2 \cdot \chi(F_i \otimes (\oO_S
		\oplus \oO_S(C))). 
		\end{align}
	As $f$ is non-degenerate, we have $k\ge 1$. 
We also have 
\begin{align*}
	\chi(F_i \otimes (\oO_S
	\oplus \oO_S(C)))
	=\chi(\dR f_{\ast}\dR \hH om(\eE, F_i))>0
	\end{align*}	
since $\dR f_{\ast} \dR \hH om(\eE, F_i)$ is a zero 
dimensional sheaf, which is non-zero by the equivalence (\ref{equiv:pervS}). 
Therefore (\ref{q-2:b}) is positive, and the lemma holds. 
	\end{proof}

\begin{rmk}\label{rmk:Ppull}
	We have the $\Theta$-stratifications of $\pP_n^{\rm{per}}(X, \beta)$
	and $\fP_n^{\rm{per}}(S, \beta)$
	with respect to $(l_t, b)$
	\begin{align*}
		&\pP_n^{\rm{per}}(X, \beta)=\sS_1^{\Omega} \sqcup \cdots 
		\sqcup \sS_N^{\Omega} \sqcup \pP_n^t(X, \beta), \\ 
		&\fP_n^{\rm{per}}(S, \beta)=\fS_1 \sqcup \cdots 
		\sqcup \fS_M \sqcup \fP_n^t(S, \beta). 
		\end{align*}
	The second $\Theta$-stratification 
	does not necessary satisfy the weight condition in Remark~\ref{rmk:sod}
	so that it is not pulled-back to the first 
	$\Theta$-stratification. 
	Indeed for $t\gg 0$, 
	a PT stable pair on $X$ does not push-forward to a 
	PT stable pair on $S$, since there exist stable pairs on $X$ 
	thickened into the fiber direction of $X \to S$ along $C$. 
	So $\pP_n^t(X, \beta)$ is strictly bigger than the pull-back of 
	$\pP_n^t(S, \beta)$, and 
	$\dDT^{\C}(\pP_n^t(X, \beta))$ is not necessary 
	equivalent to $\Dbc(\fP_n^t(S, \beta))$. 
	\end{rmk}

\begin{lem}\label{lem:assum:per}
	Assumption~\ref{assum:Wx} is satisfied for 
	$\nN=\pP_n^{\rm{per}}(X, \beta)$, 
	$l=l_t$ for $t=k\in \mathbb{Z}_{\ge 1}$
	and $l_{\pm}=l_{t\pm \varepsilon}$ for $0<\varepsilon \ll 1$. 	
	\end{lem}
\begin{proof}
	Let us take a closed point $y \in P_n^{\rm{per}}(S, \beta)$
	corresponding to a polystable object (\ref{points:per2}), and 
	we use the same symbol $y \in \pP_n^{\rm{per}}(S, \beta)$
	for the corresponding closed point. 
	Then as in~\cite[Remark~5.4.2]{TocatDT}, the
	$G_y$-representation 
	\begin{align}\notag
		\hH^0(\mathbb{T}_{\fP_n^{\rm{per}}(S, \beta)}|_{y}) \oplus 
		\hH^1(\mathbb{T}_{\fP_n^{\rm{per}}(S, \beta)}|_{y})^{\vee}
		\end{align} 
	is the space of 
	representations of the Ext-quiver $Q_y$ associated with the collection 
	in $\Dbc(\overline{X})$
	\begin{align*}
		\{I_0'=(\oO_{\overline{X}} 
		\to i_{\ast}F_0), S_0[-1], S_1[-1], i_{\ast}\oO_{x_1}[-1], \ldots, i_{\ast}\oO_{x_k}[-1]\} 
		\end{align*}
	where $i \colon S \hookrightarrow X$ is the zero section, 
	with dimension vector $\vec{v}$
	given by 
	\begin{align*}
		\vec{v}=
		(1, \dim V_0, \dim V_1, \dim W_1, \ldots, \dim W_k).
		\end{align*} 
For $t=k \in \mathbb{Z}_{\ge 1}$, 
let us take a closed point
 $x \in \pP_n^{\rm{per}}(X, \beta)^{l_t \sss}$ which 
 maps to $y$ by the composition 
 \begin{align*}
 	\pP_n^{\rm{per}}(X, \beta) \stackrel{\pi_{\ast}}{\to}
 	 \pP_n^{\rm{per}}(S, \beta) \stackrel{\pi_{\pP}}{\to}
 	  P_n^{\rm{per}}(S, \beta). 
 	\end{align*}
 By Lemma~\ref{lem:abig} and Proposition~\ref{prop:sss}, 
 the point $x$ is represented by a $\mu_t^{\dag}$-polystable perverse 
 coherent system of the form (\ref{polyI}). 
 The corresponding closed point 
 \begin{align*}
 	x \in 	\left(\hH^0(\mathbb{T}_{\fP_n^{\rm{per}}(S, \beta)}|_{y}) \oplus 
 	\hH^1(\mathbb{T}_{\fP_n^{\rm{per}}(S, \beta)}|_{y})^{\vee}\right)^{l_t \sss}
 	\end{align*}
is represented by a polystable $Q_y$-representation of the form 
\begin{align*}
	R=R_{\infty} \oplus (R_1 \otimes W).
	\end{align*}
Here  
$W=\mathbb{C}^m$, 
$R_{\infty}$ is a $Q_y$-representation 
of dimension vector of the form 
$(1, \ast, \ldots, \ast)$,
and $R_1$
is a $Q_y$-representation 
corresponding to $\oO_C(k-1)$. 
From the exact sequence 
$0 \to \oO_C(-1)^{\oplus k-1} \to \oO_C^{\oplus k} \to \oO_C(k-1) \to 0$, 
the above $R_1$
fits into the exact sequence 
$0 \to \mathbf{e}_0^{\oplus k} \to R_1 \to 
\mathbf{e}_1^{\oplus k-1} \to 0$, 
where $\mathbf{e}_i$ is the simple $Q_y$-representation
corresponding to $S_i[-1]$. 
In particular, $R_1$ has dimension vector $(0, k, k-1, 0, \ldots, 0)$. 
The $G_y$-character 
$\lL_{t\pm \varepsilon}|_{y}$ restricted to $G_x=\GL(W) \subset G_y$ is 
calculated as 
\begin{align}\label{ltx}
	\lL_{t\pm \varepsilon}|_{x} \colon \GL(W) \to \C, \ g \mapsto (\det g)^{\varepsilon}. 
	\end{align}
The $G_x$-representation $W_x$ in Assumption~\ref{assum:Wx}
is given by $W_x=\Ext_{Q_y}^1(R, R)$, that is 
\begin{align*}
	W_x 
	=\Ext^1_{Q_y}(R_{\infty}, R_{\infty}) &\oplus 
	\Ext^1_{Q_y}(R_{\infty}, R_1) \otimes W \\
	& \oplus 
	\Ext^1_{Q_y}(R_1, R_{\infty}) \otimes W^{\vee}
	\oplus \Ext^1_{Q_y}(R_{1}, R_1) \otimes \Hom(W, W). 
	\end{align*}
By the above 
computation of dimension vectors for $R_{\infty}$ and $R_1$, 
the 
 Euler pairing computation of $Q_y$-representations (see~\cite[Corollary~1.4.3]{Brion})
gives 
\begin{align*}
	\dim \Ext_{Q_y}^1(R_{\infty}, R_1)-\dim 
	\Ext_{Q_y}^1(R_1, R_{\infty})=k \ge 1. 
	\end{align*}
Therefore we have the decomposition of $G_x$-representations 
$W_x=\mathbb{S}_x \oplus \mathbb{U}_x$, 
\begin{align*}
	\mathbb{S}_x 
	=\Ext^1_{Q_y}(R_{\infty}, R_{\infty}) &\oplus 
	\Ext^1_{Q_y}(R_1, R_{\infty}) \otimes W \\
	& \oplus 
	\Ext^1_{Q_y}(R_1, R_{\infty}) \otimes W^{\vee}
	\oplus \Ext^1_{Q_y}(R_{1}, R_1) \otimes \Hom(W, W),
\end{align*}
and $\mathbb{U}_x=W^{\oplus k}$. 
Note that $\mathbb{S}_x$ is a symmetric $G_x$-representation. 
From the description of the character (\ref{ltx}), 
it is easy to see that (\ref{Wsss}) holds (see~\cite[Lemma~5.1.12]{TocatDT}). 
	\end{proof}

By the above arguments so far, 
we have the following main result in this section: 
\begin{thm}\label{thm:-1-1}
	Conjecture~\ref{conj:-1} is true if $\beta$ is $f$-reduced. 
	\end{thm}
\begin{proof}
	By Lemma~\ref{lem:altenative}, Proposition~\ref{prop:sss}, Lemma~\ref{lem:posi} and Lemma~\ref{lem:assum:per}, 
	the result follows from Theorem~\ref{thm:inclu}. 
	\end{proof}

		\bibliographystyle{amsalpha}
	\bibliography{math}

\providecommand{\bysame}{\leavevmode\hbox to3em{\hrulefill}\thinspace}
\providecommand{\MR}{\relax\ifhmode\unskip\space\fi MR }
\providecommand{\MRhref}[2]{%
  \href{http://www.ams.org/mathscinet-getitem?mr=#1}{#2}
}
\providecommand{\href}[2]{#2}
\begin{thebibliography}{PTVV13}

\bibitem[AG15]{MR3300415}
D.~Arinkin and D.~Gaitsgory, \emph{Singular support of coherent sheaves and the
  geometric {L}anglands conjecture}, Selecta Math. (N.S.) \textbf{21} (2015),
  no.~1, 1--199.

\bibitem[AHLH]{AHLH}
J.~Alper, D.~Halpern-Leistner, and J.~Heinloth, \emph{Existence of moduli
  spaces for algebraic stacks}, arXiv:1812.01128.

\bibitem[Alp13]{MR3237451}
J.~Alper, \emph{Good moduli spaces for {A}rtin stacks}, Ann. Inst. Fourier
  (Grenoble) \textbf{63} (2013), no.~6, 2349--2402.

\bibitem[BBBJ15]{MR3352237}
O.~B. Bassat, C.~Brav, V.~Bussi, and D.~Joyce, \emph{A `{D}arboux theorem' for
  shifted symplectic structures on derived {A}rtin stacks, with applications},
  Geom. Topol. \textbf{19} (2015), no.~3, 1287--1359.

\bibitem[BFK19]{MR3895631}
M.~Ballard, D.~Favero, and L.~Katzarkov, \emph{Variation of geometric invariant
  theory quotients and derived categories}, J. Reine Angew. Math. \textbf{746}
  (2019), 235--303.

\bibitem[BIK08]{MR2489634}
D.~Benson, S.~B. Iyengar, and H.~Krause, \emph{Local cohomology and support for
  triangulated categories}, Ann. Sci. \'{E}c. Norm. Sup\'{e}r. (4) \textbf{41}
  (2008), no.~4, 573--619.

\bibitem[BO]{B-O2}
A.~Bondal and D.~Orlov, \emph{Semiorthogonal decomposition for algebraic
  varieties}, preprint, arXiv:9506012.

\bibitem[Bri02]{Br1}
T.~Bridgeland, \emph{Flops and derived categories}, Invent. Math \textbf{147}
  (2002), 613--632.

\bibitem[Bri12]{Brion}
M.~Brion, \emph{Representations of quivers}, Geometric methods in
  representation theory. {I}, S\'{e}min. Congr., vol.~24, Soc. Math. France,
  Paris, 2012, pp.~103--144.

\bibitem[Cal]{Calab}
J.~Calabrese, \emph{Donaldson-{T}homas invariants and flops}, preprint,
  arXiv:1111.1670.

\bibitem[Cal19]{Calaque}
D.~Calaque, \emph{Shifted cotangent stacks are shifted symplectic}, Annales de
  la faculte des sciences de Toulouse \textbf{28} (2019), 67--90.

\bibitem[DG13]{MR3037900}
V.~Drinfeld and D.~Gaitsgory, \emph{On some finiteness questions for algebraic
  stacks}, Geom. Funct. Anal. \textbf{23} (2013), no.~1, 149--294.

\bibitem[EP15]{MR3366002}
A.~I. Efimov and L.~Positselski, \emph{Coherent analogues of matrix
  factorizations and relative singularity categories}, Algebra Number Theory
  \textbf{9} (2015), no.~5, 1159--1292.

\bibitem[Gai13]{MR3136100}
D.~Gaitsgory, \emph{ind-coherent sheaves}, Mosc. Math. J. \textbf{13} (2013),
  no.~3, 399--528, 553.

\bibitem[GR17]{MR3701352}
D.~Gaitsgory and N.~Rozenblyum, \emph{A study in derived algebraic geometry.
  {V}ol. {I}. {C}orrespondences and duality}, Mathematical Surveys and
  Monographs, vol. 221, American Mathematical Society, Providence, RI, 2017.

\bibitem[Hir17]{MR3631231}
Y.~Hirano, \emph{Derived {K}n\"orrer periodicity and {O}rlov's theorem for
  gauged {L}andau-{G}inzburg models}, Compos. Math. \textbf{153} (2017), no.~5,
  973--1007.

\bibitem[HLa]{HalpK32}
D.~Halpern-Leistner, \emph{Derived {$\Theta$}-stratifications and the
  {$D$}-equivalence conjecture}, arXiv:2010.01127.

\bibitem[HLb]{Halpinstab}
\bysame, \emph{On the structure of instability in moduli theory},
  arXiv:1411.0627.

\bibitem[HL15]{MR3327537}
\bysame, \emph{The derived category of a {GIT} quotient}, J. Amer. Math. Soc.
  \textbf{28} (2015), no.~3, 871--912.

\bibitem[HLP]{Halpstack}
D.~Halpern-Leistner and A.~Preygel, \emph{Mapping stacks and categorical
  notions of properness}, arXiv:1402.3204.

\bibitem[HLS20]{HLKSAM}
D.~Halpern-Leistner and S.~V. Sam, \emph{Combinatorial constructions of derived
  equivalences}, J. Amer. Math. Soc. \textbf{33} (2020), no.~3, 735--773.

\bibitem[Isi13]{MR3071664}
M.~U. Isik, \emph{Equivalence of the derived category of a variety with a
  singularity category}, Int. Math. Res. Not. IMRN (2013), no.~12, 2787--2808.

\bibitem[Joy15a]{JoyceD}
D.~Joyce, \emph{A classical model for derived critical loci}, J.~Differential
  Geom.~ \textbf{101} (2015), 289--367.

\bibitem[Joy15b]{MR3399099}
\bysame, \emph{A classical model for derived critical loci}, J. Differential
  Geom. \textbf{101} (2015), no.~2, 289--367.

\bibitem[JS12]{JS}
D.~Joyce and Y.~Song, \emph{A theory of generalized {D}onaldson-{T}homas
  invariants}, Mem. Amer. Math. Soc. \textbf{217} (2012), no.~1020, iv+199.

\bibitem[JT17]{MR3607000}
Y.~Jiang and R.~Thomas, \emph{Virtual signed {E}uler characteristics}, J.
  Algebraic Geom. \textbf{26} (2017), no.~2, 379--397.

\bibitem[Kaw02]{MR1949787}
Y.~Kawamata, \emph{{$D$}-equivalence and {$K$}-equivalence}, J. Differential
  Geom. \textbf{61} (2002), no.~1, 147--171.

\bibitem[KS]{K-S}
M.~Kontsevich and Y.~Soibelman, \emph{Stability structures, motivic
  {D}onaldson-{T}homas invariants and cluster transformations}, preprint,
  arXiv:0811.2435.

\bibitem[KT]{KoTo}
N.~Koseki and Y.~Toda, \emph{Derived categories of {T}haddeus pair moduli
  spaces via d-critical flips}, arXiv:1904.04949.

\bibitem[Lur09]{Ltopos}
J.~Lurie, \emph{Higher topos theory}, Annals of Mathematics Studies, vol. 170,
  Princeton University Press, Princeton, NJ, 2009.

\bibitem[Orl09]{Orsin}
D.~Orlov, \emph{Derived categories of coherent sheaves and triangulated
  categories of singularities}, Algebra, arithmetic, and geometry: in honor of
  {Y}u. {I}. {M}anin. {V}ol. {II}, Progr. Math., vol. 270, Birkh\"{a}user
  Boston, Inc., Boston, MA, 2009, pp.~503--531.

\bibitem[Orl12]{Ornonaff}
D.~Orlov, \emph{Matrix factorizations for nonaffine {LG}-models}, Math. Ann.
  \textbf{353} (2012), no.~1, 95--108.

\bibitem[P{\u{a}}da]{Tudor2}
T.~P{\u{a}}durairu, \emph{Generators for categorical {H}all algebras of
  surfaces}, arXiv:2106.05176.

\bibitem[P{\u{a}}db]{Tudor}
\bysame, \emph{K-theoretic {H}all algebras for quivers with potential},
  arXiv:1911.05526.

\bibitem[PS]{PoSa}
M.~Porta and F.~Sala, \emph{Two dimensional categorified {H}all algebras},
  arXiv:1903.07253.

\bibitem[PT09]{PT}
R.~Pandharipande and R.~P. Thomas, \emph{Curve counting via stable pairs in the
  derived category}, Invent.~Math.~ \textbf{178} (2009), 407--447.

\bibitem[PTVV13]{MR3090262}
T.~Pantev, B.~To\"{e}n, M.~Vaqui\'{e}, and G.~Vezzosi, \emph{Shifted symplectic
  structures}, Publ. Math. Inst. Hautes \'{E}tudes Sci. \textbf{117} (2013),
  271--328.

\bibitem[PV11]{MR3112502}
A.~Polishchuk and A.~Vaintrob, \emph{Matrix factorizations and singularity
  categories for stacks}, Ann. Inst. Fourier (Grenoble) \textbf{61} (2011),
  no.~7, 2609--2642.

\bibitem[Shi12]{MR2982435}
I.~Shipman, \emph{A geometric approach to {O}rlov's theorem}, Compos. Math.
  \textbf{148} (2012), no.~5, 1365--1389.

\bibitem[Tab05]{Tab}
G.~Tabuada, \emph{Une structure de cat\'{e}gorie de mod\`eles de {Q}uillen sur
  la cat\'{e}gorie des dg-cat\'{e}gories}, C. R. Math. Acad. Sci. Paris
  \textbf{340} (2005), no.~1, 15--19.

\bibitem[Toda]{Toddbir}
Y.~Toda, \emph{Birational geometry for d-critical loci and wall-crossing in
  {C}alabi-{Y}au 3-folds}, arXiv:1805.00182.

\bibitem[Todb]{TocatDT}
\bysame, \emph{Categorical {D}onaldson-{T}homas theory for local surfaces},
  arXiv:1907.09076.

\bibitem[Todc]{TocatDT2}
\bysame, \emph{Categorical {D}onaldson-{T}homas theory for local surfaces:
  $\mathbb{Z}/2$-periodic version}, preprint.

\bibitem[Tod09]{Tolim}
\bysame, \emph{Limit stable objects on {C}alabi-{Y}au 3-folds}, Duke Math.~J.~
  \textbf{149} (2009), 157--208.

\bibitem[Tod10]{Tolim2}
\bysame, \emph{Generating functions of stable pair invariants via
  wall-crossings in derived categories}, Adv.~Stud.~Pure Math.~ \textbf{59}
  (2010), 389--434, New developments in algebraic geometry, integrable systems
  and mirror symmetry (RIMS, Kyoto, 2008).

\bibitem[Tod12]{Tsurvey}
\bysame, \emph{Stability conditions and curve counting invariants on
  {C}alabi-{Y}au 3-folds}, Kyoto Journal of Mathematics \textbf{52} (2012),
  1--50.

\bibitem[Tod13]{Tcurve2}
\bysame, \emph{Curve counting theories via stable objects~{II}. {DT}/nc{DT}
  flop formula}, J.~Reine Angew.~Math.~ \textbf{675} (2013), 1--51.

\bibitem[Tod20]{THtype}
\bysame, \emph{Hall-type algebras for categorical {D}onaldson-{T}homas theories
  on local surfaces}, Selecta Math. (N.S.) \textbf{26} (2020), no.~4, 64.

\bibitem[To{\"{e}}07]{Todg}
B.~To{\"{e}}n, \emph{The homotopy theory of {$dg$}-categories and derived
  {M}orita theory}, Invent. Math. \textbf{167} (2007), no.~3, 615--667.

\bibitem[To{\"{e}}14]{MR3285853}
\bysame, \emph{Derived algebraic geometry}, EMS Surv. Math. Sci. \textbf{1}
  (2014), no.~2, 153--240.

\bibitem[VdB04]{MR2057015}
M.~Van~den Bergh, \emph{Three-dimensional flops and noncommutative rings}, Duke
  Math. J. \textbf{122} (2004), no.~3, 423--455.

\end{thebibliography}

	\vspace{5mm}
	
	Kavli Institute for the Physics and 
	Mathematics of the Universe (WPI), University of Tokyo,
	5-1-5 Kashiwanoha, Kashiwa, 277-8583, Japan.

	\textit{E-mail address}: yukinobu.toda@ipmu.jp
\end{document}